\documentclass[12pt]{amsart}

\usepackage{amssymb,latexsym}
\usepackage{bm}
\usepackage{mathrsfs}
\usepackage{graphicx}
\usepackage{psfrag}
\usepackage{wrapfig}
\usepackage{color}
\usepackage{enumitem}
\setlist[enumerate]{parsep=0pt plus 4pt,topsep=0pt plus 4pt}
\usepackage{url}
\usepackage{hyperref}
\usepackage{tikz}
\definecolor{darkblue}{RGB}{0,0,160}
\hypersetup{colorlinks,citecolor=black,filecolor=black,%
            linkcolor=darkblue,urlcolor=darkblue}
\allowdisplaybreaks

\newcommand{\excise}[1]{}

\newtheorem{thm}{Theorem}[section]
\newtheorem{lemma}[thm]{Lemma}

\newtheorem{cor}[thm]{Corollary}
\newtheorem{prop}[thm]{Proposition}

\newtheorem{hyp}[thm]{Hypothesis}

\theoremstyle{definition}
\newtheorem{example}[thm]{Example}
\newtheorem{remark}[thm]{Remark}
\newtheorem{defn}[thm]{Definition}
\newtheorem{conv}[thm]{Convention}

\numberwithin{equation}{section}

\newcounter{separated}

\newcommand{\Ring}[1]{\ensuremath{\mathbb{#1}}}

\renewcommand\>{\rangle}

\newcommand\0{\mathbf{0}}
\newcommand\DD{\mathscr{D}}
\newcommand\EE{\mathscr{E}}
\newcommand\FF{\mathscr{F}}
\newcommand\GG{\mathscr{G}}
\newcommand\NN{\Ring{N}}
\newcommand\OO{\mathcal{O}}

\newcommand\QQ{{\mathbb Q}}
\newcommand\RR{{\mathbb R}}
\newcommand\UU{\mathscr{U}}
\newcommand\ZZ{{\mathbb Z}}
\newcommand\bb{{\mathbf b}}

\newcommand\ee{{\mathbf e}}

\newcommand\ii{{\mathbf i}}
\newcommand\jj{{\mathbf j}}
\newcommand\kk{\Bbbk}
\newcommand\mm{{\mathfrak m}}

\newcommand\qq{{\mathbf q}}

\newcommand\xx{{\mathbf x}}

\newcommand\cA{{\mathcal A}}
\newcommand\cB{{\mathcal B}}

\newcommand\cH{H}

\newcommand\cM{M}
\newcommand\cN{N}
\newcommand\cP{P}
\newcommand\cQ{Q}

\newcommand\cZ{Z}

\newcommand\oD{\hspace{.3ex}\ol{\hspace{-.3ex}D\hspace{-.05ex}}\hspace{.05ex}}

\newcommand\ob{\hspace{.3ex}^{\raisebox{.8ex}{$\scriptscriptstyle\circ$}}%
	\hspace{-1.22ex}b}
\newcommand\vC{\check{\mathcal C}}

\newcommand\del{\partial}

\renewcommand\aa{{\mathbf a}}
\renewcommand\phi{\varphi}

\newcommand\ale{{\mathord{\mathrm{ale}}}}

\newcommand\con{{\mathord{\mathrm{con}}}}

\newcommand\ord{{\mathord{\mathrm{ord}}}}

\newcommand\too{\longrightarrow}

\newcommand\from{\leftarrow}
\newcommand\into{\hookrightarrow}
\newcommand\otni{\hookleftarrow}

\newcommand\onto{\twoheadrightarrow}

\newcommand\spot{{\hbox{\raisebox{1pt}{\tiny$\scriptscriptstyle\bullet$}}}}

\newcommand\minus{\smallsetminus}

\newcommand\simto{\mathrel{\!\ooalign{$\fillrightmap$\cr\raisebox{.75ex}{$\,\sim\ \hspace{.2ex}$}}}}
\newcommand\cupdot{\ensuremath{\mathbin{\mathaccent\cdot\cup}}}
\newcommand\goesto{\rightsquigarrow}
\newcommand\dirlim{\varinjlim}

\newcommand\nothing{\varnothing}

\newcommand\filleftmap{\mathord\leftarrow \mkern-6mu
	\cleaders\hbox{$\mkern-2mu \mathord- \mkern-2mu$}\hfill
	\mkern-6mu \mathord-}
\newcommand\fillrightmap{\mathord- \mkern-6mu
	\cleaders\hbox{$\mkern-2mu \mathord- \mkern-2mu$}\hfill
	\mkern-6mu \mathord\rightarrow}

\newcommand{\lhookdownarrow}{\rotatebox[origin=c]{-90}{$\into$}}
\newcommand{\twoheaddownarrow}{\rotatebox[origin=c]{-90}{$\onto$}}
\renewcommand\iff{\Leftrightarrow}
\renewcommand\epsilon{\varepsilon}
\renewcommand\implies{\Rightarrow}

\newcommand\dx[1][]{\delta^{\hspace{.1ex}\xi}}

\newcommand\ol[1]{{\overline{#1}}}

\definecolor{lightred}{rgb}{1,.3,.3}
\newcommand\red{\color{lightred}}
\newcommand\blu{\color{blue}}

\newcommand{\aoverb}[2]{{\genfrac{}{}{0pt}{1}{#1}{#2}}}
\def\twoline#1#2{\aoverb{\scriptstyle {#1}}{\scriptstyle {#2}}}

\DeclareMathOperator\Hom{Hom} 
\DeclareMathOperator\Mor{Mor} 

\DeclareMathOperator\coker{coker} 

\newcommand\monomialmatrix[3]{{
\begin{array}{@{}r@{\:}r@{}c@{}l@{}}
  \begin{array}{@{}c@{}}		
	\begin{array}{@{}r@{}}
	\\
	#1
	\end{array}\!
  \end{array}						
&
  \begin{array}{@{}c@{}}		
	\begin{array}{@{}l@{}}\\				
	\end{array}						
	\\							
	\left[\begin{array}{@{}l@{}}				
	#3							
	\end{array}\!						
	\right.							
  \end{array}							
&
  #2					
&
  \begin{array}{@{}c@{}}		
	\begin{array}{@{}l@{}}\\				
	\end{array}						
	\\							
	\left.\!\begin{array}{@{}l@{}}				
	#3							
	\end{array}						
	\right]							
  \end{array}							
\end{array}
}}

\begin{document}

\title[Modules over posets: commutative and homological algebra]%
      {Modules over posets:\\commutative and homological algebra}
\author{Ezra Miller}
\address{Mathematics Department\\Duke University\\Durham, NC 27708}
\urladdr{\url{http://math.duke.edu/people/ezra-miller}}

\makeatletter
  \@namedef{subjclassname@2010}{\textup{2010} Mathematics Subject Classification}
\makeatother

\date{25 August 2019}

\begin{abstract}
The commutative and homological algebra of modules over posets is
developed, as closely parallel as possible to the algebra of finitely
generated modules over noetherian commutative rings, in the direction
of finite presentations, primary decompositions, and resolutions.
Interpreting this finiteness in the language of derived categories of
subanalytically constructible sheaves proves two conjectures due to
Kashiwara and Schapira concerning sheaves with microsupport in a given
cone.

The motivating case is persistent homology of arbitrary filtered
topological spaces, especially the case of multiple real parameters.
The algebraic theory yields computationally feasible, topologically
interpretable data structures, in terms of birth and death of homology
classes, for persistent homology indexed by arbitrary posets.

The exposition focuses on the nature and ramifications of a suitable
finiteness condition to replace the noetherian hypothesis.  The
tameness condition introduced for this purpose captures finiteness for
variation in families of vector spaces indexed by posets in a way that
is characterized equivalently by distinct topological, algebraic,
combinatorial, and homological manifestations.  Tameness serves both
the theoretical and computational purposes: it guarantees finite
primary decompositions, as well as various finite presentations and
resolutions all related by a syzygy theorem, and the data structures
thus produced are computable in addition to being interpretable.

The tameness condition and its resulting theory are new even in the
finitely generated discrete setting, where being tame is materially
weaker than being noetherian.
\end{abstract}
\maketitle

\setcounter{tocdepth}{2}
\tableofcontents

\section{Introduction}\label{s:intro}

\subsection*{Overview}\label{sub:overview}

A module over a poset is a family of vector spaces indexed by the
poset elements, with a homomorphism for each poset relation.  The
setup is inherently commutative: the homomorphism for a poset relation
$p \preceq q$ is the composite of the homomorphisms for the relations
$p \preceq r$ and $r \preceq q$ whenever $r$ lies between $p$ and~$q$.
This paper is a study of the extent to which modules over posets
behave like modules over noetherian commutative rings, particularly
when it comes to finite presentations, primary decompositions, and
resolutions.  Interpreting this finiteness in the language of derived
categories of subanalytically constructible sheaves proves two
conjectures due to Kashiwara and Schapira concerning sheaves with
microsupport in a given cone.

The algebraic and homological investigations can be viewed as testing
the frontier of multigraded commutative algebra regarding how far one
can get without a ring and with minimal hypotheses on the
multigrading.  However, beyond this abstract mathematical interest,
the impetus lies in data science applications, where the poset
consists of ``parameters'' indexing a family of topological subspaces
of a fixed topological space.  Taking homology of the subspaces in
this topological filtration yields a poset module, called the
persistent homology of the filtration, referring to how homology
classes are born, persist for a while, and then die as the parameter
moves up in~the~poset.

In ordinary persistent homology, the poset is totally
ordered---usually the real numbers~$\RR$, the integers~$\ZZ$, or a
subset $\{1,\dots,m\}$.  This case is well studied (see
\cite{edelsbrunner-harer2010}, for example), and the algebra is
correspondingly simple \cite{crawley-boevey2015}.  Persistence with
multiple totally ordered parameters, introduced by Carlsson and
Zomorodian \cite{multiparamPH}, has been developed in various ways,
often assuming that the poset is $\NN^n$.  That discrete framework is
preferred in part because it arises frequently when filtering finite
simplicial complexes, but also because settings involving continuous
parameters---including the application that drives the advances
here---unavoidably produce modules that fail to be finitely presented
in several fundamental ways.

The foundations lain here take the lack of noetherian hypotheses
head~on, to open the~possibility of working directly with modules over
arbitrary posets.  The focus is therefore on the nature and
ramifications of a suitable finiteness condition to replace the
noetherian hypothesis.  The tameness condition introduced here appears
to be the natural candidate, capturing finiteness of variation in a
way that is characterized equivalently by distinct topological,
algebraic, combinatorial, and homological manifestations.  Tameness
serves both the theoretical and computational purposes: it guarantees
finite primary decompositions, as well as various finite presentations
and resolutions all related by a syzygy theorem, and the data
structures thus produced are computationally feasible, in addition to
being topologically interpretable in terms of birth and death of
homology classes.  The tameness condition, and its resulting primary
decompositions, syzygy theorem, and interpretable computational
structures, are new even in the finitely generated discrete setting,
where tame is much weaker~than~\mbox{noetherian}.

No restriction on the underlying poset is required except for primary
decomposition, which needs the poset to be a group to allow analogues
of prime ideals and localization.  There is no difference in any of
the theory for cases that are, for example, not locally finite, such
as~$\RR^n$ or other partially ordered real vector spaces.  Moreover,
the data structures and transitions between the topological,
algebraic, combinatorial, and homological perspectives take advantage
of and preserve supplementary geometry, be it subanalytic,
semialgebraic, or piecewise-linear, for instance, if the ambient
structures---the partial orderings and the modules---have such
geometry to begin with.  The subanalytic and piecewise-linear cases
are crucial for the applications to the conjectures of Kashiwara and
Schapira.

\setcounter{tocdepth}{-1}
\subsection*{Acknowledgements}\label{sub:acknowledgements}
\setcounter{tocdepth}{2}

First, a special acknowledgement goes to Ashleigh Thomas, who has been
and continues to be a long-term collaborator on this project.  She was
listed as an author on earlier drafts of \cite{qr-codes} (of which
this is roughly the first third), but her contributions lie more
properly beyond these preliminaries (see \cite{primary-distance}, for
example), so she declined in the end to be named as an author on this
installment.  Early in the development of the ideas here, Thomas put
her finger on the continuous rather than discrete nature of
multiparameter persistence modules for fly wings.  She computed the
first examples explicitly, namely those in
Example~\ref{e:toy-model-fly-wing}, and produced the biparameter
persistence diagrams there as well as some of the figures in
Example~\ref{e:flange-switching}.

Justin Curry pointed out connections from the combinatorial viewpoint
taken here, in terms of modules over posets, to higher notions in
algebra and category theory, particularly those involving
constructible sheaves, which are in the same vein as Curry's proposed
uses of them in persistence \cite{curry-thesis}; see
Remarks~\ref{r:curry}, \ref{r:indicator}, \ref{r:lurie},
and~\ref{r:kan-extension}.

The author is indebted to David Houle, whose contribution to this
project was seminal and remains ongoing; in particular, he and his lab
produced the fruit fly wing images \cite{houle03}.  Paul Bendich and
Joshua Cruz took part in the genesis of this project, including early
discussions concerning ways to tweak persistent (intersection
\cite{bendich-harer2011}) homology for the fly wing investigation.
Ville Puuska discovered several errors in an early version of
Section~\ref{s:encoding}, resulting in substantial correction and
alteration; see Examples~\ref{e:puuska-nonconstant-isotypic}
and~\ref{e:puuska-nontransitive}.  Pierre Schapira gave helpful
comments on Section~\ref{s:derived}.  Banff International Research
Station provided an opportunity for valuable feedback and suggestions
at the workshop there on Topological Data Analysis (August, 2017) as
this research was being completed; many participants, especially the
organizers, Uli Bauer and Anthea Monod, as well as Michael Lesnick,
shared important perspectives and insight.  Thomas Kahle requested
that \mbox{Proposition}~\ref{p:determined} be an equivalence instead
of merely the one implication it had stated.  Hal Schenck gave helpful
comments on an earlier version of the Introduction.  Some passages in
Section~\ref{sub:motivation} are based on or taken verbatim
from~\cite{fruitFlyModuli}.  Portions of this work were funded by NSF
grant~DMS-1702395.

\subsection{Motivation and examples}\label{sub:motivation}

The developments here grew out of investigation of data structures for
real multiparameter persistence modules, where both senses of the word
``real'' are intended: actual---from genuine data, with a particular
dataset in mind---and with parameters taking continuous as opposed to
discrete values.  Instead of reviewing the numerous possible reasons
for considering multiparameter persistence, many already having been
present from the outset \cite[\S1.1]{multiparamPH}, what follows is a
description of the instance of real multiparameter persistence that
arises in the biological problem that the theory here is specifically
designed to serve.

\begin{example}\label{e:fly-wing-filtration}
The veins in a fruit fly
wing can be presented as an embedded planar graph, with a location for
each vertex and an algebraic curve for each arc \cite{houle03}.  There
are many options for statistical summaries of fly wings, some of them
elementary, such as a linear model taking into account a weighted sum
of (say) the number of vertices and the total edge length.  Whatever
the chosen method, it has to grapple with the topological vein
variation, giving appropriate weight to new or deleted singular points
in addition to varying shape, as in the following images.
$$%
\includegraphics[height=25mm]{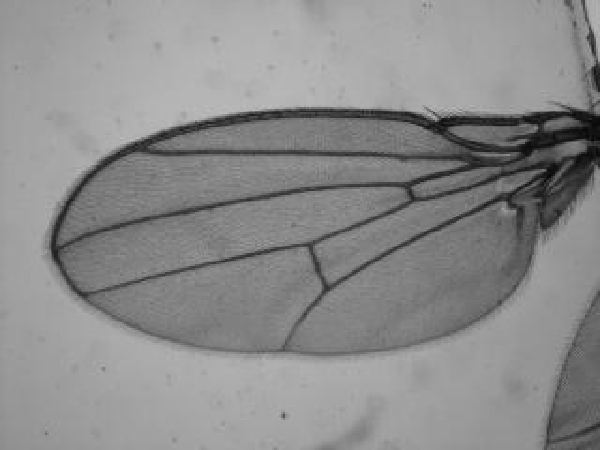}
\qquad
\includegraphics[height=25mm]{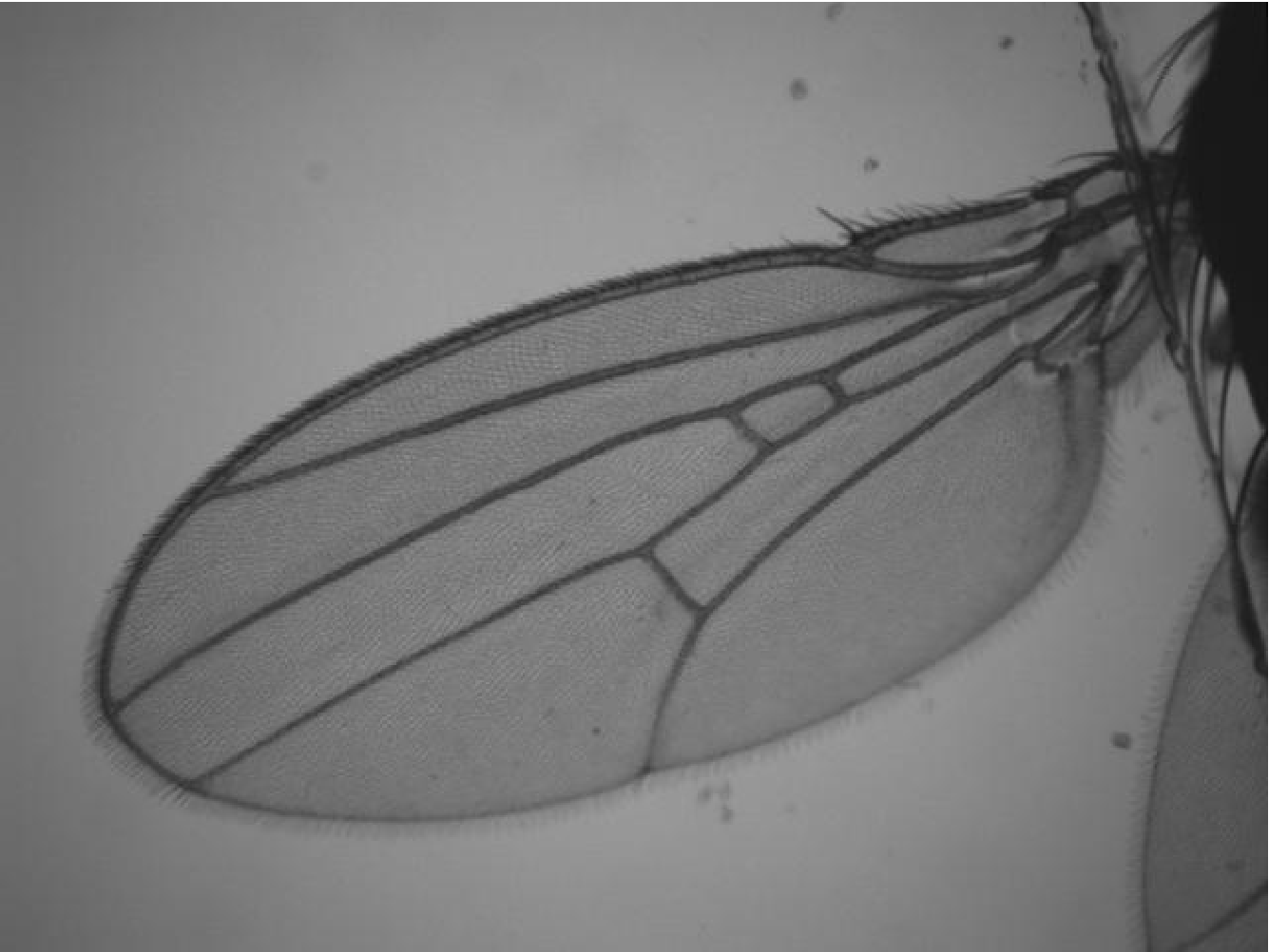}
\qquad
\includegraphics[height=25mm]{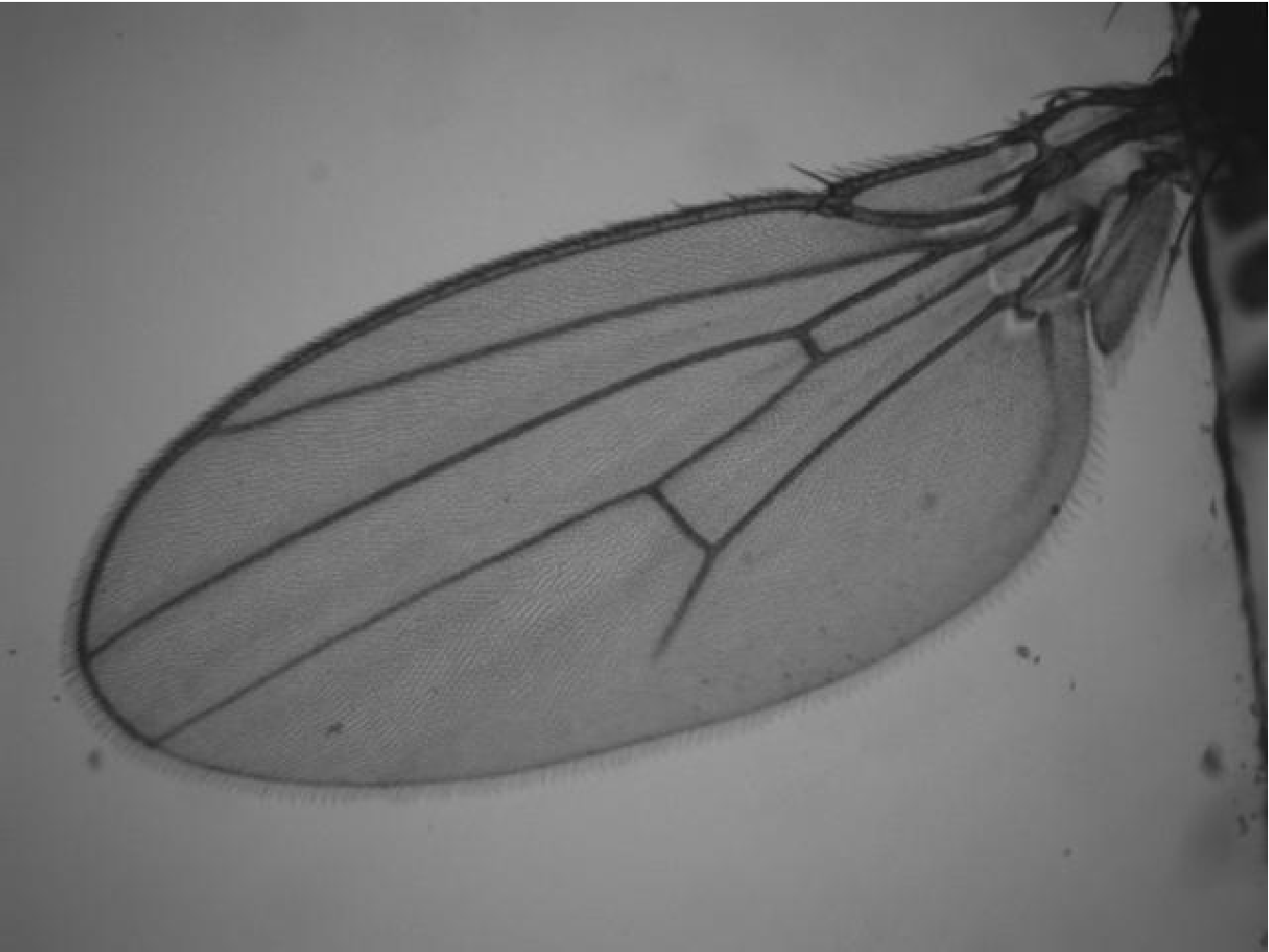}
$$
The nature of wing vein formation from gene expression levels during
embryonic development (see \cite{blair07} for background) provides
rationale for believing that continuous biparameter persistent
homology models biological reality reasonably faithfully.

Let $\cQ = \RR_- \times \RR_+$ with the coordinatewise partial order,
so $(r,s) \in \cQ$ for any nonnegative real numbers~$-r$ and~$s$.  Let
$X = \RR^2$ be the plane in which the fly wing is embedded and define
$X_{rs} \subseteq X$ to be the set of points at distance at least~$-r$
from every vertex and within~$s$ of some edge.  Thus $X_{rs}$ is
obtained by removing the union of the balls of radius~$r$ around the
vertices from the union of $s$-neighborhoods of the edges.  In the
following portion of a fly wing, $-r$ is approximately twice~$s$:
$$%
\includegraphics[height=23mm]{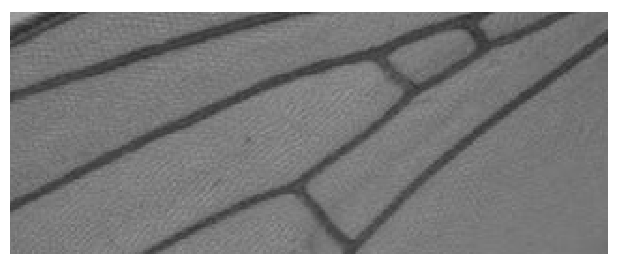}
\quad
\raisebox{10mm}{$\goesto$}
\quad
\includegraphics[height=23mm]{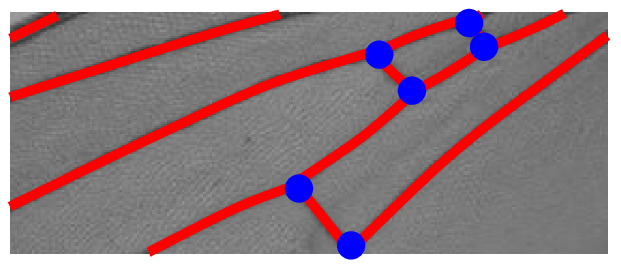}
\vspace{-1ex}
$$
The biparameter persistent homology module $\cM_{rs} = H_0(X_{rs})$
summarizes wing vein structure for the intended biological purposes.
\end{example}

Relevant properties of these modules are best highlighted in a
simplified setting.

\begin{example}\label{e:toy-model-fly-wing}
Using the setup from Example~\ref{e:fly-wing-filtration}, the zeroth
persistent homology for the toy-model ``fly wing'' at left in
Figure~\ref{f:toy-model-fly-wing} is the $\RR^2$-module $\cM$ shown at
\begin{figure}[ht]
\vspace{-8.5ex}
$$%
\begin{array}[b]{@{}c@{}}
\mbox{}\\[28.5pt]
\includegraphics[height=30mm]{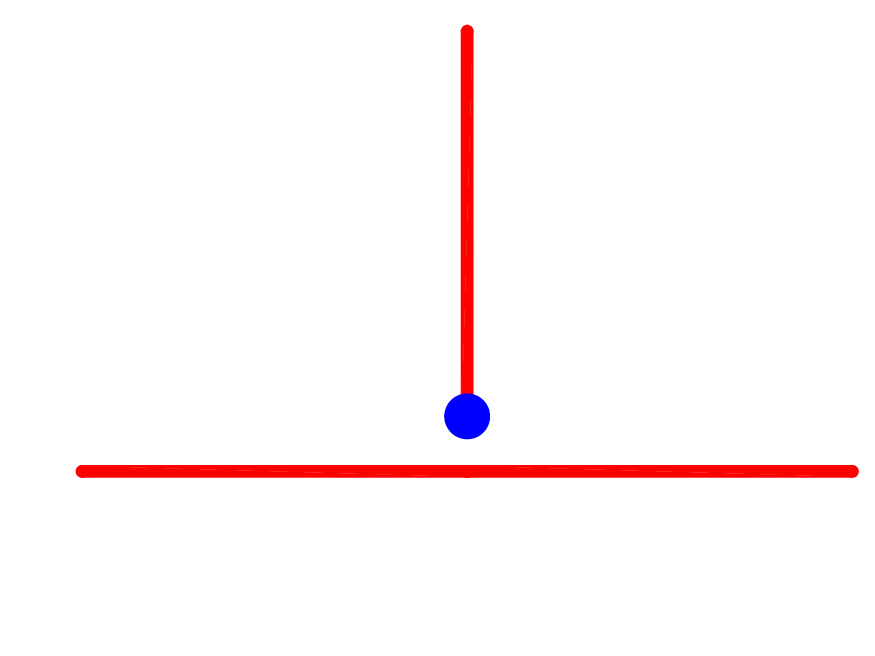}
\\[-22.5pt]
\end{array}
\begin{array}[b]{@{\ \ }c@{}}
\hspace{-3pt}\goesto\\[7mm]\mbox{}
\end{array}
\qquad
\begin{array}[b]{@{\hspace{-10pt}}r@{\hspace{-10pt}}|@{}l@{}}
\begin{array}{@{}c@{}}
\psfrag{r}{\tiny$r \to$}
\psfrag{s}{\tiny$\begin{array}{@{}c@{}}\uparrow\\[-.5ex]s\end{array}$}
\includegraphics[height=30mm]{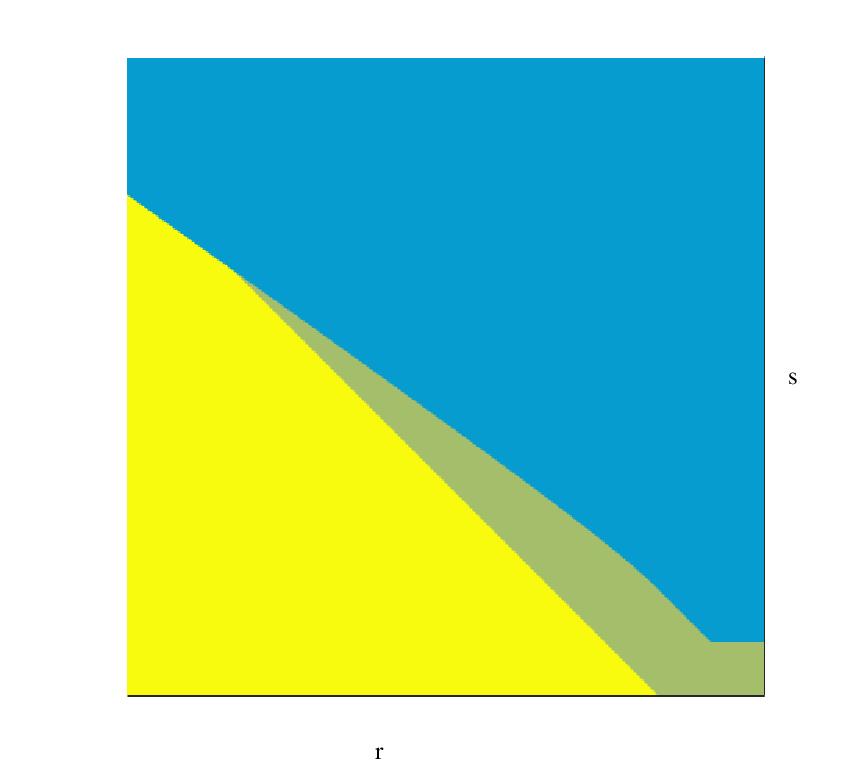}\\[-7.7pt]
\end{array}
&\,\,\,\\[-6pt]\hline
\end{array}
\qquad
\begin{array}[b]{@{\ \ }c@{}}
\hspace{-3pt}\goesto\\[7mm]\mbox{}
\end{array}
\qquad
\begin{array}[b]{@{\hspace{-10pt}}r@{\hspace{-10pt}}|@{}l@{}}
\begin{array}{@{}c@{}}
\psfrag{1}{\tiny$\kk$}
\psfrag{2}{\tiny$\kk^2$}
\psfrag{3}{\tiny$\kk^3$}
\includegraphics[height=30mm]{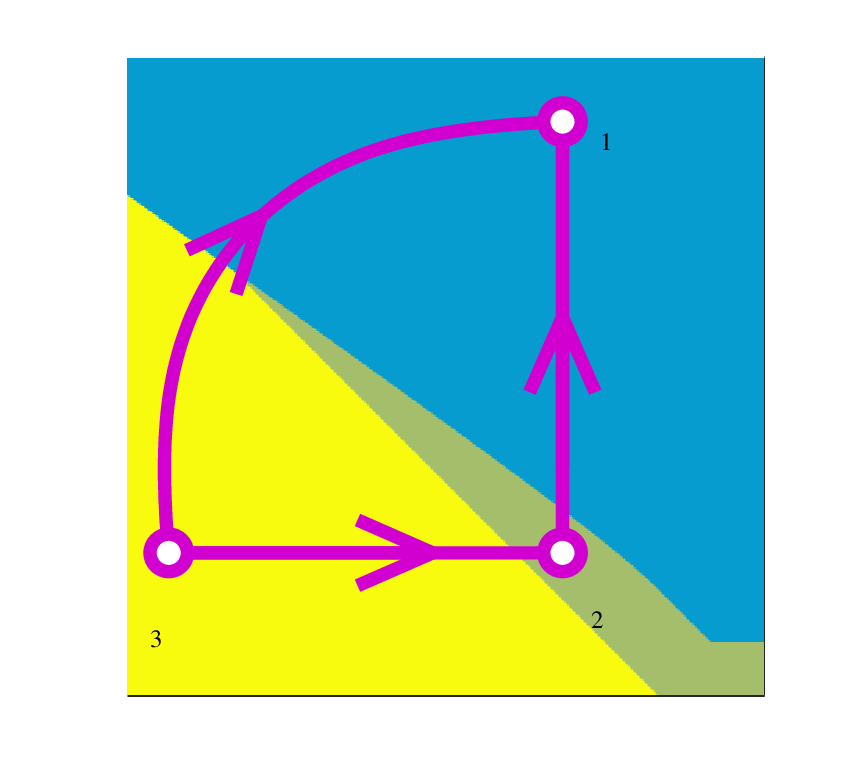}\\[-7.7pt]
\end{array}
&\,\,\,\\[-6pt]\hline
\end{array}
\vspace{-2ex}
$$
\caption{Biparameter persistence module and finite encoding\vspace{-1ex}}
\label{f:toy-model-fly-wing}
\end{figure}
center.  Each point of $\RR^2$ is colored according to the dimension
of its associated vector space in~$\cM$, namely $3$, $2$, or~$1$
proceeding up (increasing~$s$) and to the right (increasing~$r$).  The
structure homomorphisms $\cM_{rs} \to \cM_{r's'}$ are all surjective.

This $\RR^2$-module fails to be finitely presented for three
fundamental reasons.  First, the three generators sit infinitely far
back along the $r$-axis.  (Fiddling with the sign on~$r$ does not
help: the natural maps on homology proceed from infinitely large
radius to~$0$ regardless of how the picture is drawn.)  Second, the
relations that specify the transition from vector spaces of
dimension~$3$ to those of dimension~$2$ or~$1$ lie along a real
algebraic curve, as do those specifying the transition from
dimension~$2$ to dimension~$1$.  These curves have uncountably many
points.  Third, even if the relations are discretized---restrict~$\cM$
to a lattice $\ZZ^2$ superimposed on~$\RR^2$, say---the relations
march off to infinity roughly diagonally away from the origin.  (See
Example~\ref{e:encoding} for the right-hand~image.)
\end{example}

\subsection{Modules over posets}\label{sub:modules}

There are many essentially equivalent ways to think of a poset module.
The definition in the first line of this Introduction is among the
more elementary formulations; see Definition~\ref{d:poset-module} for
additional precision.  Others include a
\begin{itemize}
\item%
representation of a poset \cite{nazarova-roiter};

\item%
functor from a poset to the category of vector spaces (e.g., see
\cite{curry2019});

\item%
vector-space valued sheaf on a poset (e.g., see
Section~\ref{ssub:topology} or \cite[\S4.2]{curry-thesis});

\item%
representation of a quiver with (commutative) relations (e.g., see
\cite[\S A.6]{oudot2015});


\item%
representation of the incidence algebra of a poset
\cite{doubilet-rota-stanley1972}; or

\item%
module over a directed acyclic graph \cite{chambers-letscher2018}.
\end{itemize}
The premise here is that commutative algebra provides an elemental
framework out of which flows corresponding structure in these other
contexts, in which the reader is encouraged to interpret all of the
results.  Section~\ref{s:derived} provides an example of how that can
look, in this case from the sheaf perspective.  Expressing the
foundations via commutative algebra is natural for its infrastructure
surrounding resolutions.  And as the objects are merely graded vector
spaces with linear maps among them---there are no rings to act---it is
also the most elementary language available, to make the exposition
accessible to a wide audience, including statisticians applying
persistent homology in addition to topologists, combinatorialists,
algebraists, geometers, and programmers.

Some of the formulations of poset module are only valid when the poset
is assumed to be locally finite (see \cite{doubilet-rota-stanley1972},
for instance), or when the object being acted upon satisfies a
finitary hypothesis \cite{khripchenko-novikov2009} in which the
algebraic information is nonzero on only finitely many points in any
interval.  This is not a failing of any particular formulation, but
rather a signal that the theory has a different focus.  Combinatorial
formulations are built for enumeration.  Representation theories are
built for decomposition into and classification of irreducibles.
While commutative algebra appreciates a direct sum decomposition when
one is available, such as over a noetherian ring of dimension~$0$, its
initial impulse is to relate arbitrary modules to simpler ones by less
restrictive decomposition, such as primary decomposition, or by
resolution, such as by projective or injective modules.  That is the
tack taken here.

\subsection{Topological tameness}\label{sub:tame}

The finiteness condition introduced here generalizes the notion of
topological tameness for ordinary persistence in a single parameter
(see \mbox{\cite[\S3.8]{chazal-deSilva-glisse-oudot2016}}, for
example), reflecting the intuitive notion that given a filtration of a
topological space from data, only finitely many topologies should
appear.  The \emph{tame} condition
(Definitions~\ref{d:constant-subdivision} and~\ref{d:tame}) partitions
the poset into finitely many domains of constancy for a given module.
Tameness is a topological concept, designed to control the size and
variation of homology groups in subspaces of a fixed topological
space.

\begin{example}\label{e:tame}
The $\RR^2$-module in Example~\ref{e:toy-model-fly-wing} is tame, with
four constant regions: over the bottom-left region (yellow) the vector
space is~$\kk^3$; over the middle (olive) region the vector space
is~$\kk^2$; over the upper-right (blue) region the vector space
is~$\kk$; and over the remainder of~$\RR^2$ the vector space is~$0$.
The homomorphisms between these vector spaces do not depend on which
points in the regions are selected to represent them.  For instance,
$\kk^3 \to \kk^2$ always identifies the two basis vectors
corresponding to the connected components that are the left and right
halves of the horizontally~infinite~red~strip.
\end{example}

It is worth noting that in ordinary totally ordered persistence,
tameness means simply that the bar code (see
Section~\ref{sub:bar-codes}) has finitely many bars, or equivalently,
the persistence diagram has finitely many off-diagonal dots:
finiteness of the set of constant regions precludes infinitely many
non-overlapping bars (the bar code can't be ``too long''), while the
vector space having finite dimension precludes a parameter value over
which lie infinitely many bars (the bar code can't be ``too thick'').

In principle, tameness can be reworked to serve as a data structure
for algorithmic computation, especially in the presence of an
auxiliary hypothesis to regulate the geometry of the constant
regions---when they are semialgebraic or piecewise linear
(Definition~\ref{d:auxiliary-hypotheses}.\ref{i:semialgebraic}
or~\ref{d:auxiliary-hypotheses}.\ref{i:PL}), for example.  The
algorithms would generalize those for polyhedral ``sectors'' in the
discrete case \cite{injAlg} (or see~\cite[Chapter~13]{cca}).

\subsection{Combinatorial tameness: finite encoding}\label{sub:combin-tame}

Whereas the topological notion of tameness requires little more than
an arbitrary subdivision of the poset into regions of constancy
(Definition~\ref{d:constant-subdivision}), the combinatorial
incarnation imposes additional structure on the constant regions,
namely that they should be partially ordered in a natural way.  More
precisely, tt stipulates that the module~$\cM$ should be pulled back
from a $\cP$-module along a poset morphism $\cQ \to \cP$ in which
$\cP$ is a finite poset and the $\cP$-module has finite dimension as a
vector space over the field~$\kk$ (Definition~\ref{d:encoding}).

\begin{example}\label{e:encoding}
The right-hand image in Example~\ref{e:toy-model-fly-wing} is a finite
encoding of~$\cM$ by a three-element poset $\cP$ and the $\cP$-module
$H = \kk^3 \oplus \kk^2 \oplus \kk$ with each arrow in the image
corresponding to a full-rank map between summands of~$H$.
Technically, this is only an encoding of~$\cM$ as a module over $\cQ =
\RR_- \times \RR_+$.  The poset morphism $\cQ \to \cP$ takes all of
the yellow rank~$3$ points to the bottom element of~$\cP$, the olive
rank~$2$ points to the middle element of~$\cP$, and the blue rank~$1$
points to the top element of~$\cP$.  (To make this work over all
of~$\RR^2$, the region with vector space dimension~$0$ would have to
be subdivided, for instance by introducing an antidiagonal extending
downward from the origin, thus yielding a morphism from~$\RR^2$ to a
five-element poset.)  This encoding is semialgebraic
(Definition~\ref{d:auxiliary-hypotheses}): its fibers are real
semialgebraic sets.
\end{example}

In general, constant regions need not be situated in a manner that
makes them the fibers of a poset morphism (Example~\ref{e:subdivide}).
Nonetheless, over arbitrary posets, modules that are tame by virtue of
admitting a finite constant subdivision (Definition~\ref{d:tame})
always admit finite encodings (Theorem~\ref{t:tame}), although the
constant regions are typically subdivided by the encoding poset
morphism.  In the case where the poset is a real vector space, if the
constant regions have additional geometry
(Definition~\ref{d:auxiliary-hypotheses}), then a similarly geometric
finite encoding is possible.

The framework of poset modules arising from filtrations of topological
spaces is more or less an instance of MacPherson's exit path category
\cite[\S1.1]{treumann2009}.  In that context, Lurie defined a notion
of constructibility in the Alexandrov topology \cite[Definitions~A.5.1
and~A.5.2]{lurie2017}, independently of the developments here and for
different purposes.  It would be reasonable to speculate that tameness
should correspond to Alexandrov constructibility, given that encoding
of a poset module is defined by pulling back along a poset morphism
(in Lurie's language, a continuous morphism of posets), but it does
not; see Remark~\ref{r:lurie}.  The difference between constant in the
sense of tameness via constant subdivision
(Section~\ref{sub:constant}) and locally constant in the
sheaf-theoretic sense with the Alexandrov topology makes tameness---in
the equivalent finitely encoded formulation---rather than Alexandrov
constructibility the right notion of finiteness for the syzygy theorem
as well as for algorithmic computation and data analysis applications
of multipersistence.  That contrasts with the comparison between
tameness and subanalytic constructibility in the usual topology on
real vector spaces, which are essentially the same notion for the
relevant sheaves; see Section~\ref{sub:dervied}.

Finite encoding has its roots in combinatorial commutative algebra in
the form of sector partitions \cite{injAlg} (or
see~\cite[Chapter~13]{cca}).  Like sector partitions, finite encoding
is useful, theoretically, for its enumeration of all topologies
encountered as the parameters vary.  However, enumeration leads to
combinatorial explosion outside of the very lowest numbers of
parameters.  And beyond its inherent inefficiency, poset encoding
lacks many of the features that have come to be expected from
persistent homology, including the most salient: a description of
homology classes in terms of their persistence, meaning birth, death,
and lifetime.

\subsection{Discrete persistent homology by birth and death}\label{sub:disc-PH}

$\ZZ^n$-modules have been present in commutative algebra for over half
a century \cite{GWii}, but the perspective arising from their
equivalence with multipersistence is relatively new
\cite{multiparamPH}.  Initial steps have included descriptions of the
set of isomorphism classes \cite{multiparamPH}, presentations
\cite{csv17} and algorithms for computing \cite{computMultiPH,csv12}
or visualizing \cite{lesnick-wright2015} them, as well as interactions
with homological algebra of modules, such as persistence invariants
\cite{knudson2008} and certain notions of multiparameter noise
\cite{scolamiero-chacholski-lundman-ramanujam-oberg16}.

Algebraically, viewing persistent homology as a module rather than
(say) a diagram or a bar code, a birth is a generator.  In ordinary
persistence, with one parameter, a death is more or less the same as a
relation.  However, in multipersistence the notion of death diverges
from that of relation.  The issue is partly one of geometric shape in
the parameter poset, say~$\ZZ^n$ (the uppermost shaded regions
indicate where classes die): \vspace{-.2ex}
$$%
\psfrag{birth}{\tiny birth}
\psfrag{death}{\tiny death}
\psfrag{relation}{\tiny relation}
\includegraphics[height=30mm]{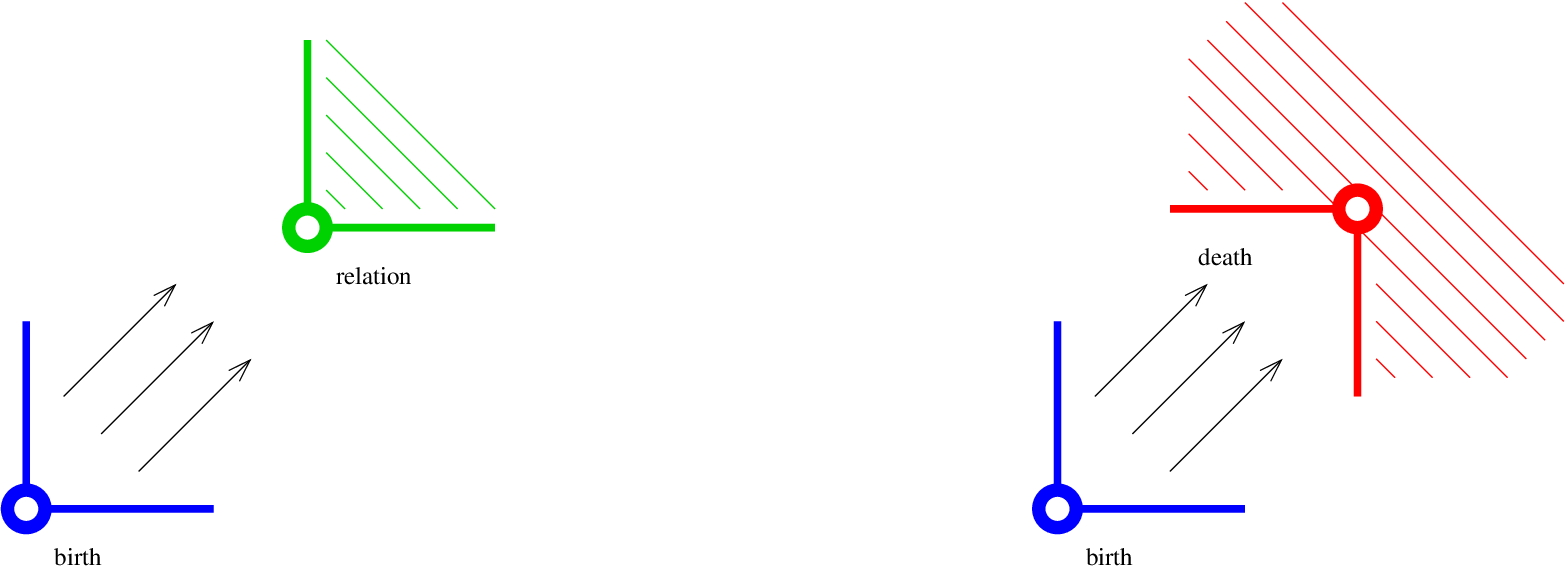}
$$
\vspace{-.25ex}%
If death is to be dual to birth, then a nonzero homology class at some
parameter should die if it moves up along any direction in the poset.
Birth is not the bifurcation of a homology class into two independent
ones; it is the creation of a new class from zero.  Likewise, genuine
death is not the joining of two classes into one; it is annihilation.
And death should be stable, in the sense that wiggling the parameter
and then pushing up should still kill the homology class.

In algebraic language, death is a ``cogenerator'' rather than a
relation.  For finitely generated $\NN^n$-modules, or slightly more
generally for finitely determined modules
(Example~\ref{e:convex-projection} and Definition~\ref{d:determined}),
cogenerators are irreducible components, cf.~\cite[Section~5.2]{cca}.
Indeed, irreducible decomposition suffices as a dual theory of death
in the finitely generated case; this is more or less the content of
Theorem~\ref{t:finitely-determined}.  The idea there is that
surjection from a free module covers the module by sending basis
elements to births in the same (or better, dual) way that inclusion
into an injective module envelops the module by capturing deaths as
injective summands.  The geometry of this process in the parameter
poset on the injective side is as well understood as it is on the free
side \cite[Chapter~11]{cca}, and in finitely generated situations it
is carried out theoretically or algorithmically by finitely generated
truncations of injective modules~\cite{irredRes,injAlg}.

Combining birth as free cover and death as injective hull leads
naturally to flange presentation (Definition~\ref{d:flange}), which
composes the augmentation map $F \onto \cM$ of a flat resolution with
the augmentation map $\cM \into E$ of an injective resolution to get a
homomorphism $F \to E$ whose image is~$\cM$.  The indecomposable
summands of~$F$ capture births and those of~$E$ deaths.  Flange
presentation splices a flat resolution to an injective one in the same
way that Tate resolutions (see \cite{coanda2003}, for example) segue
from a free resolution to an injective one over a Gorenstein local
ring of dimension~$0$.

Why a flat cover $F \onto \cM$ instead of a free one?  There are two
related reasons: first, flat modules are dual to injective ones
(Remark~\ref{r:flat}), so in the context of finitely determined
modules the entire theory is self-dual; and second, births can lie
infinitely far back along axes, as in the toy-model fly wing from
Example~\ref{e:toy-model-fly-wing}.

\subsection{Algebraic tameness:~fringe presentation}\label{sub:alg-tame}%

That multipersistence modules can fail to be finitely generated, like
Example~\ref{e:toy-model-fly-wing} does, in situations reflecting
reasonably realistic data analysis was observed by Patriarca,
Scolamiero, and Vaccarino \cite[Section~2]{psv12}.  Their ``monomial
module'' view of persistence covers births much more efficiently, for
discrete parameters, by keeping track of generators not individually
but gathered together as generators of monomial ideals.  Huge numbers
of predictable syzygies among generators are swallowed and hence are
present only implicitly.  And that is good, as nothing topologically
new about persistence of homology classes is taught by the well known
syzygies of monomial ideals, which in this setting are merely an
interference pattern from the merging of separate birth points of the
same class.

Translating to the setting of continuous parameters, and including the
dual view of deaths, which works just as well, suggests an uncountably
more efficient way to cover births and deaths than listing them
individually.  In fact this urge to gather births or deaths does not
really depend on the transition to continuous parameters from discrete
ones.  Indeed, any $\RR^n$-module~$\cM$ can be approximated by
a~$\ZZ^n$-module, the result of restricting $\cM$ to, say, the
rescaled lattice~$\epsilon\ZZ^n$.  Suppose, for the sake of argument,
that~$\cM$ is bounded, in the sense of being zero at parameters
outside of a bounded subset of~$\RR^n$; think of
Example~\ref{e:toy-model-fly-wing}, ignoring those parts of the module
there that lie outside of the depicted square.
\begin{figure}[h]
\vspace{-2ex}
$$%
\begin{array}[b]{@{\hspace{-10pt}}r@{\hspace{-10pt}}|@{}l@{}}
\begin{array}{@{}c@{}}
\psfrag{r}{}
\psfrag{s}{}
\includegraphics[height=30mm]{toy-model}\\[-7.7pt]
\end{array}
&\,\,\,\\[-6pt]\hline
\end{array}
\quad\
\begin{array}[b]{@{\ \ }c@{}}
\hspace{-3pt}\goesto\\[7mm]\mbox{}
\end{array}
\qquad
\begin{array}[b]{@{\hspace{-10pt}}r@{\hspace{-10pt}}|@{}l@{}}
\begin{array}{@{}c@{}}
\includegraphics[height=30mm]{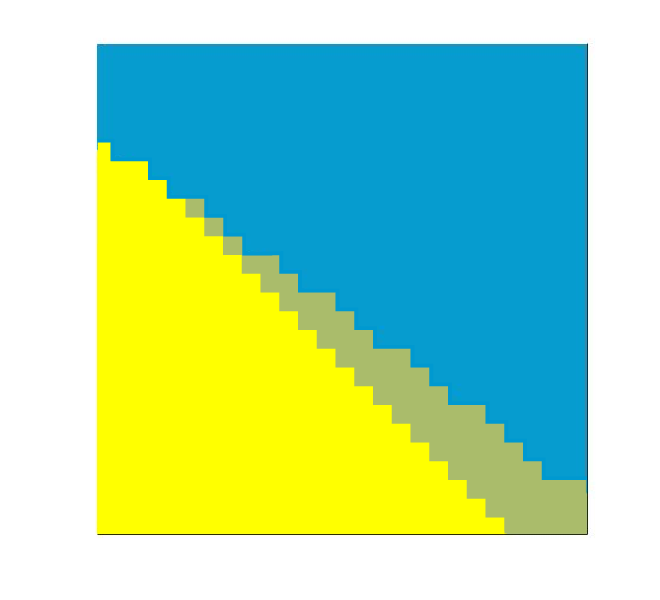}\\[-7.7pt]
\end{array}
&\,\,\,\\[-6pt]\hline
\end{array}
$$\vspace{-4ex}
\end{figure}
Ever better approximations, by smaller~$\epsilon \to 0$, yield sets of
lattice points ever more closely hugging an algebraic curve.
Neglecting the difficulty of computing where those lattice points lie,
how is a computer to store or manipulate such a set?  Listing the
points individually is an option, and perhaps efficient for
particularly coarse approximations, but in~$n$ parameters the
dimension of this storage problem is~$n-1$.  As the approximations
improve, the most efficient way to record such sets of points is
surely to describe them as the allowable ones on one side of an
algebraic curve.  And once the computer has the curve in memory, no
approximation is required: just use the (points on the) curve itself.
In this way, even in cases where the entire topological filtration
setup can be approximated by finite simplicial complexes,
understanding the continuous nature of the un-approximated setup is
both more transparent and more efficient.

Combining flange presentation with this monomial module view of births
and deaths yields \emph{fringe presentation}
(Definition~\ref{d:fringe}), the analogue for modules over an
arbitrary poset~$\cQ$ of flange presentation for finitely determined
modules over $\cQ = \ZZ^n$.  The role of indecomposable free or flat
modules is played by upset modules (Example~\ref{e:indicator}) which
have $\kk$ in degrees from an upset~$U$ and~$0$ elsewhere.  The role
of indecomposable injective modules is played similarly by downset
modules.

Fringe presentation is expressed by a \emph{monomial matrix}
(Definition~\ref{d:monomial-matrix-fr}), an array of scalars with rows
labeled by upsets and columns labeled by downsets.  For example,
\vspace{7ex}
$$%
\monomialmatrix
	{\begin{array}[t]{@{}r@{\hspace{-3.1pt}}|@{}l@{}}
	 \includegraphics[height=15mm]{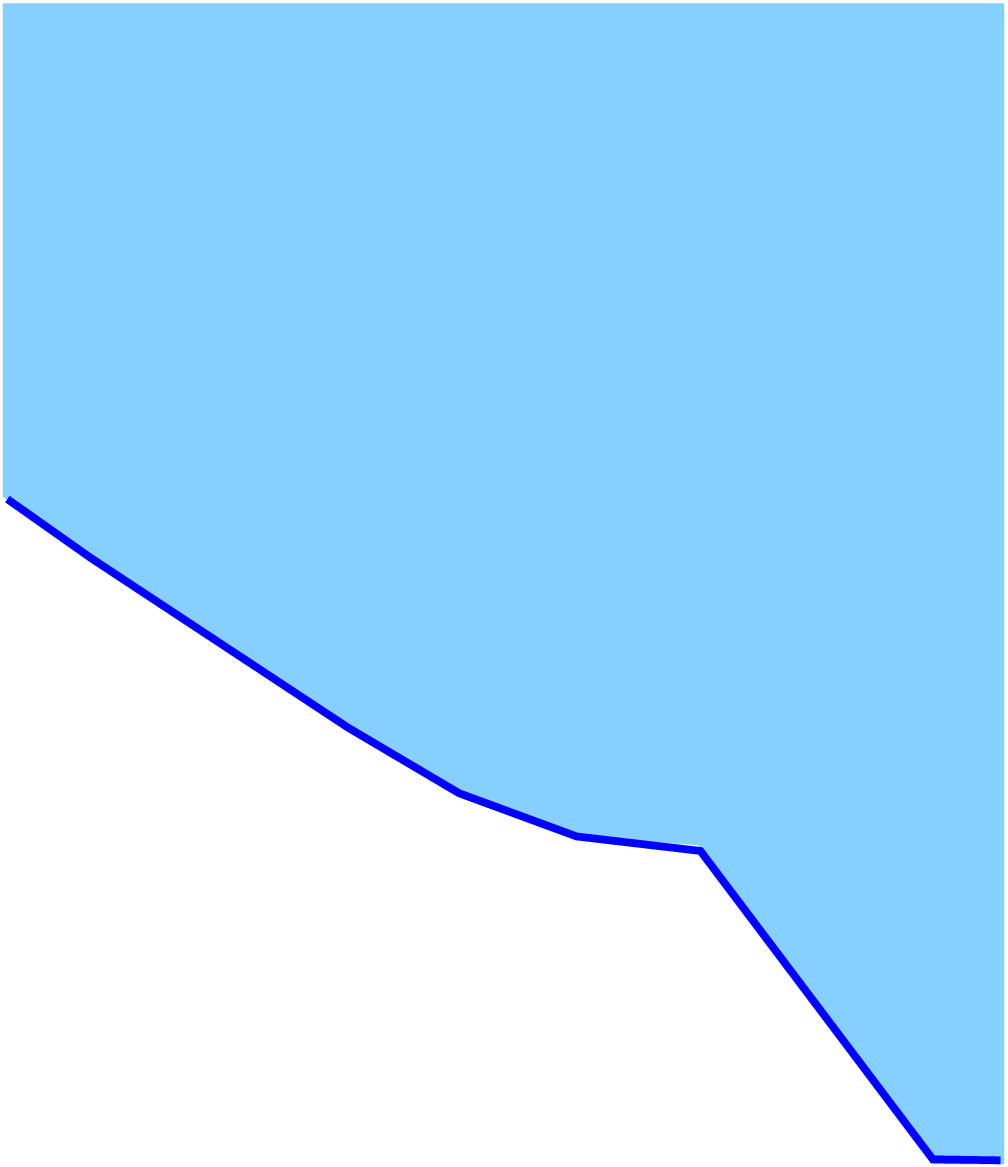}&\ \hspace{-.2pt}\\[-4.3pt]\hline
	\end{array}}
	{\!\!
	 \begin{array}{c}
	 \\[-10ex]
	 \begin{array}[b]{@{}r@{\hspace{-.2pt}}|@{}l@{}}
		\raisebox{-5mm}{\includegraphics[height=17mm]{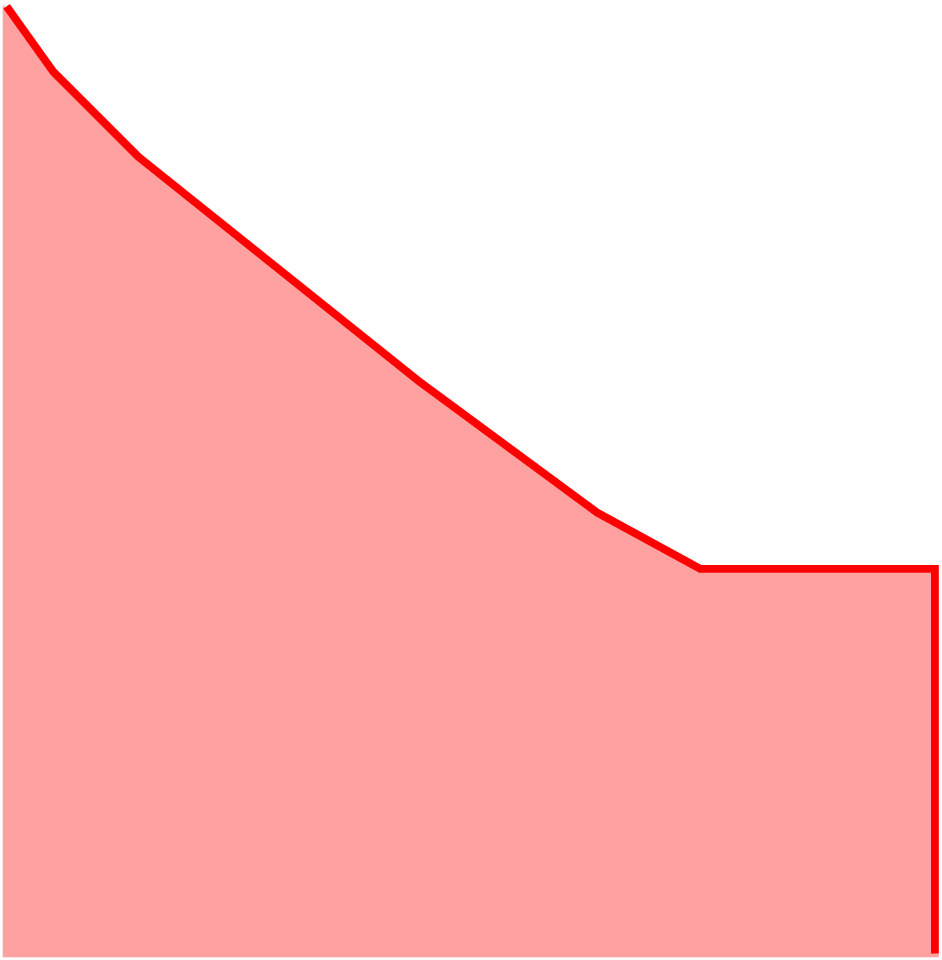}}
		&\ \,\\[-6.3pt]\hline
	 \end{array}
	 \!\!
	 \\[4ex]
	 \phi_{11}
	 \\[3ex]
	 \end{array}}
	{\\\\\\}
\begin{array}{@{}c@{}}
\hspace{.1pt}\ \text{represents a fringe presentation of}\ \hspace{.2pt}
\cM = \kk\!\!
\left[
\begin{array}{@{\ }c@{\,}}
\\[-2.2ex]
\begin{array}{@{}r@{\hspace{-.4pt}}|@{}l@{}}
\includegraphics[height=15mm]{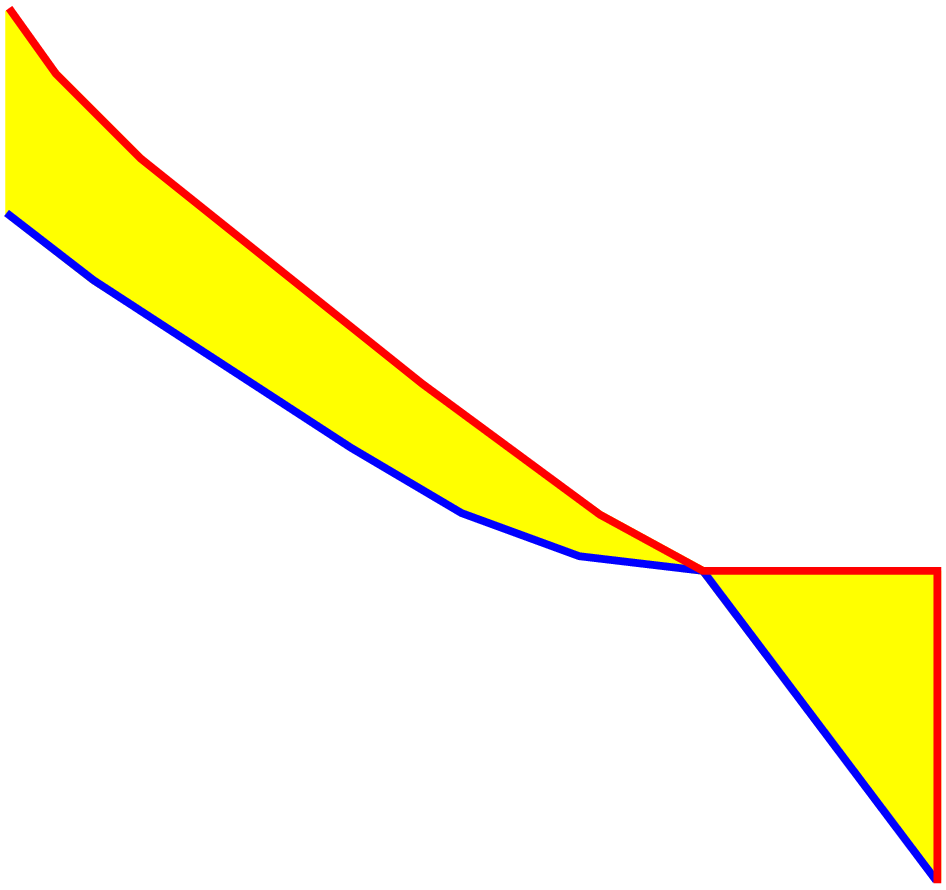}&\ \,\hspace{-.3pt}\\[-4.7pt]\hline
\end{array}
\\[-1ex]\mbox{}
\end{array}
\right]
\\[-3ex]
\end{array}
\vspace{.5ex}
$$
as long as $\phi_{11} \in \kk$ is nonzero.  The monomial matrix
notation specifies a homomorphism
\(
  \kk\bigl[\,
  \begin{array}{@{}c@{}}
  \\[-3ex]
  \begin{array}{@{}r@{\hspace{-1.4pt}}|@{}l@{}}
  \raisebox{-2.2pt}{\includegraphics[height=5mm]{blue-upset}}
  &\hspace{1.7pt}\\[-2pt]\hline
  \end{array}\,
  \end{array}
  \bigr]\!
\to
  \kk\bigl[\,
  \begin{array}{@{}c@{}}
  \\[-3ex]
  \begin{array}{@{}r@{\hspace{-.2pt}}|@{}l@{}}
  \raisebox{-6.2pt}{\includegraphics[height=5mm]{red-downset}}
  &\hspace{1.5pt}\\[-2pt]\hline
  \end{array}
  \end{array}
  \,\bigr]
\)
whose image is~$\cM$, which has $\cM_\aa = \kk$ over the yellow
parameters~$\aa$ and~$0$ elsewhere.  The blue upset specifies the
births at the lower boundary of~$\cM$; unchecked, the classes would
persist all the way up and to the right.  But the red downset
specifies the deaths along the upper boundary of~$\cM$.

When the birth upsets and death downsets are semialgebraic, or
piecewise linear, or otherwise manageable algorithmically, monomial
matrices render fringe presentations effective data structures for
real multipersistence.  Fringe presentations have the added benefit of
being topologically interpretable in terms of birth and death.

Although the data structure of fringe presentation is aimed at
$\RR^n$-modules, it is new and lends insight already for finitely
generated $\NN^n$-modules (even when $n = 2$), where monomial matrices
have their origins \cite[Section~3]{alexdual}.  The context there is
more or less that of finitely determined modules; see
Definition~\ref{d:monomial-matrix-fl} in particular, which is really
just the special case of fringe presentation in which the upsets are
localizations of~$\NN^n$ and the downsets are duals---that is,
negatives---of those.

\subsection{Homological tameness: the syzygy theorem}\label{sub:homalg}

Even in the case of filtrations of finite simplicial complexes by
products of intervals---that is, multifiltrations
(Example~\ref{e:RR+}) of finite simplicial complexes---persistent
homology is not naturally a module over a polynomial ring in $n$ (or
any finite number of) variables.  This is for the same reason that
single-parameter persistent homology is not naturally a module over a
polynomial ring in one variable: though there might only be finitely
many topological transitions, they can (and often do) occur at
incommensurable real numbers.  That said, observe that filtering a
finite simplicial complex automatically induces a finite encoding.
Indeed, the parameter space maps to the poset of simplicial
subcomplexes of the original simplicial complex by sending a parameter
to the simplicial subcomplex it indexes.  That is not the smallest
poset, of course, but it leads to a fundamental point: one can and
should do homological algebra over the finite encoding poset rather
than (only) over the original parameter space.

This line of thinking culminates in a syzygy theorem
(Theorem~\ref{t:syzygy} for modules; Theorem~\ref{t:syzygy-complexes}
for complexes) to the effect that, remarkably, the topological,
algebraic, combinatorial, and homological notions of tameness
available respectively via
\begin{itemize}
\item%
constant subdivision (Definition~\ref{d:tame}),
\item%
fringe presentation (Definition~\ref{d:fringe}),
\item%
poset encoding (Definition~\ref{d:encoding}), and
\item%
indicator resolution (Definition~\ref{d:resolutions})
\end{itemize}
are equivalent.  The moral is that the tame condition over arbitrary
posets appears to be the right notion to stand in lieu of the
noetherian hypothesis over~$\ZZ^n$: the tame condition is robust, has
multiple characterizations from different mathematical perspectives,
and enables algorithmic computation in principle.

The syzygy theorem directly reflects the more usual syzygy theorem for
finitely determined $\ZZ^n$-modules
(Theorem~\ref{t:finitely-determined}), with upset and downset
resolutions being the arbitrary-poset analogues of free and injective
resolutions, respectively, and fringe presentation being the
arbitrary-poset analogue of flange presentation.

Topological tameness via constant subdivision is a~priori weaker (that
is, more inclusive) than combinatorial tameness via finite encoding,
and algebraic tameness via fringe presentation is a~priori weaker than
homological tameness via upset or downset resolution.  Thus the syzygy
theorem leverages relatively weak topological structure into powerful
homological structure.  One consequence is a proof of two conjectures
due to Kashiwara and Schapira; see Section~\ref{sub:dervied}.  The
developments there rely on the fact that, although the
characterizations of tameness require no additional structure on the
underlying poset, any additional structure that is
present---subanalytic, semialgebraic, or piecewise-linear---is
preserved by the transitions between the characterizations of tameness
in the syzygy theorem.

The proof of the syzygy theorem works by reducing to the finitely
determined (Section~\ref{s:ZZn}) case over~$\ZZ^n$.  The main point is
that given a finite encoding of a module over an arbitrary
poset~$\cQ$, the encoding poset can be embedded in~$\ZZ^n$.  The proof
is completed by pushing the data forward to~$\ZZ^n$, applying the more
usual syzygy theorem to finitely determined modules there, and pulling
back to~$\cQ$.

\subsection{Geometric algebra over partially ordered abelian groups}\label{sub:pogroup}

This paper develops commutatve algebra for modules over posets from
scratch.  Alas, the part of the theory relating to primary
decomposition is not amenable to arbitrary posets.  Rather, the
natural setting to carry out primary decomposition is, in the best
tradition of classical mathematics
\cite{birkhoff42,clifford40,riesz40}, over partially ordered abelian
groups (Definition~\ref{d:pogroup}).  Those provide an optimally
general context in which posets have a notion of ``face'' along which
to localize without altering the ambient poset.  That is, a partially
ordered group~$\cQ$ has an origin---namely its identity~$\0$---and
hence a positive cone~$\cQ_+$ of elements it precedes.  A~\emph{face}
of~$\cQ$ is a submonoid of $\cQ_+$ that is also a downset therein
(Definition~\ref{d:face}).  And as everything takes place inside of
the ambient group~$\cQ$, every localization of a $\cQ$-module along a
face (Definition~\ref{d:support}) remains a $\cQ$-module.

In persistence language, a single element in a module over a partially
ordered group can a~priori be mortal or immortal in more than one way.
But some elements die ``pure deaths'' of only a single type~$\tau$.
These are the \emph{$\tau$-coprimary elements} for~a~face~$\tau$.  In
the concrete setting of a partially ordered real vector space with
closed positive cone, a coprimary element is characterized
(Example~\ref{e:closed}, Definition~\ref{d:elementary-coprimary}, and
Theorem~\ref{t:elementary-coprimary})~as
\begin{enumerate}
\item%
\emph{$\tau$-persistent}: it lives when pushed up arbitrarily along
the face~$\tau$; and

\item%
\emph{$\ol\tau$-transient}: it eventually dies when pushed up in any
direction outside~of~$\tau$.
\end{enumerate}

\begin{example}\label{e:hyperbola-GD}
The downset $D$ in~$\RR^2$ consisting of all points beneath the upper
branch of the hyperbola $xy = 1$ canonically decomposes
(Theorem~\ref{t:PF}) as the union\vspace{-.4ex}%
$$%
\psfrag{x}{\tiny$x$}
\psfrag{y}{\tiny$y$}
  \begin{array}{@{}c@{}}\includegraphics[height=25mm]{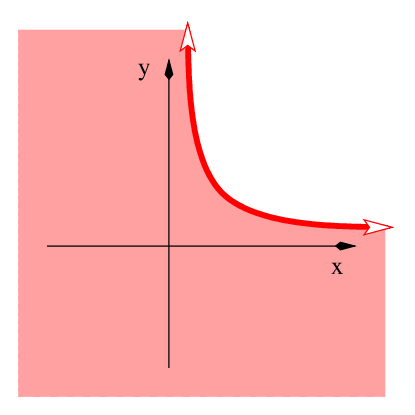}\end{array}
\ =\
  \begin{array}{@{}c@{}}\includegraphics[height=25mm]{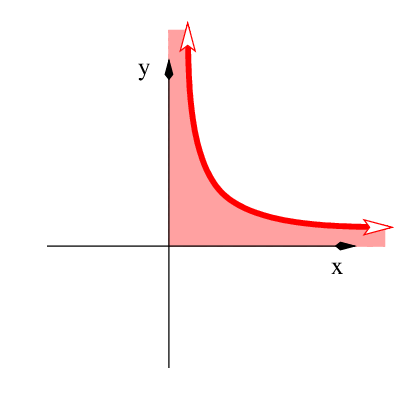}\end{array}
\cup\,
  \begin{array}{@{}c@{}}\includegraphics[height=25mm]{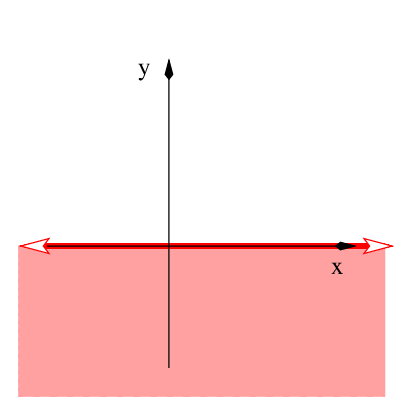}\end{array}
\cup\,
  \begin{array}{@{}c@{}}\includegraphics[height=25mm]{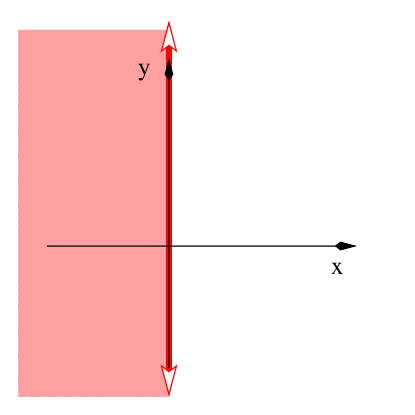}\end{array}
$$\vspace{-.6ex}%
of its subsets that locally die pure deaths of some type: every red
point in the
\begin{itemize}
\item%
leftmost subset on the right dies when pushed over to the right or up far enough;
\item%
middle subset dies in the \emph{localization} of~$D$ along the
$x$-axis (Definition~\ref{d:PF} or Definition~\ref{d:support}) when
pushed up far enough; and
\item%
rightmost subset dies locally along the $y$-axis when pushed over far
enough.
\end{itemize}
\end{example}

Isolating all coprimary elements functorially requires localization
that does not alter the ambient poset, after which local support
functors (Definition~\ref{d:local-support}) do the job, as in ordinary
commutative algebra and algebraic geometry.

What does primary decomposition do?  It expresses a given module~$\cM$
as a submodule of a direct sum of coprimary modules.  Consequently,
this decomposition tells the fortune of every element: its death types
are teased apart as the ``pure death types'' of the coprimary summands
where the element lands with nonzero image.

\begin{example}\label{e:hyperbola-PD}
The union in Example~\ref{e:hyperbola-GD} results in a canonical
primary decomposition
$$%
\psfrag{x}{\tiny$x$}
\psfrag{y}{\tiny$y$}
  \kk\!
  \left[
  \begin{array}{@{}c@{}}\includegraphics[height=25mm]{hyperbola}\end{array}
  \right]
\:\into\ 
  \kk\!
  \left[
  \begin{array}{@{}c@{}}\includegraphics[height=25mm]{hyperbola}\end{array}
  \right]
\oplus\,
  \kk\!
  \left[
  \begin{array}{@{}c@{}}\includegraphics[height=25mm]{x-component}\end{array}
  \right]
\oplus\,
  \kk\!
  \left[
  \begin{array}{@{}c@{}}\includegraphics[height=25mm]{y-component}\end{array}
  \right]
$$
of the downset module $\kk[D]$ over~$\RR^2$ (Corollary~\ref{c:PF}).
Elements in the lower-left quadrant locally die any type of death.
\end{example}

It bears emphasizing that primary decomposition of downset modules, or
equivalently, expressions of downsets as unions of coprimary downsets
(cogenerated by the $\tau$-coprimary elements for some face~$\tau$;
see Definition~\ref{d:primDecomp}), is canonical by Theorem~\ref{t:PF}
and Corollary~\ref{c:PF}, generalizing the canonical primary
decomposition of monomial ideals in ordinary polynomial rings.
However, notably lacking from primary decomposition theory over
arbitrary polyhedral partially ordered abelian groups is a notion of
minimality---alas, a lack that is intrinsic.

\begin{example}\label{e:minimality}
Although three pure death types occur in~$D$ in
Example~\ref{e:hyperbola-GD}, and hence in the union there, the final
two summands in the primary decomposition of~$\kk[D]$ in
Example~\ref{e:hyperbola-PD} are redundant.  One can, of course,
simply omit the redundant summands, but for arbitrary polyhedral
partially ordered groups no criterion is known for detecting a~priori
which summands should be omitted.
\end{example}

The failure of minimality here stems from geometry that can only occur
in partially ordered groups more general than finitely generated free
ones.  More specifically, $D$ contains elements that die pure deaths
of type ``$x$-axis'' but the boundary of~$D$ fails to contain an
actual translate of the face of~$\RR^2_+$ that is the positive
$x$-axis.  This can be seen as a certain failure of localization to
commute with taking homomorphisms into~$\kk[D]$; it is the source of
much of the subtlety in the theory developed in the sequel
\cite{essential-real} to this paper, whose purpose is partly to
rectify, for real multiparameter persistence, the failure of
minimality in Example~\ref{e:hyperbola-PD}.

The view toward algorithmic computation draws the focus to the case
where $\cQ$ is \emph{polyhedral}, meaning that it has only finitely
many faces (Definition~\ref{d:face}).  This notion is apparently new
for arbitrary partially ordered abelian groups.  Its role here is to
guarantee finiteness of primary decomposition of finitely encoded
modules (Theorem~\ref{t:primDecomp}).

\subsection{Bar codes}\label{sub:bar-codes}

Tame modules over the totally ordered set of integers or real numbers
are, up to isomorphism, the same as ``bar codes'': finite multisets of
intervals.  The most general form of this bijection between algebraic
objects and essentially combinatorial objects over totally ordered
sets is due to Crawley-Boevey \cite{crawley-boevey2015}.  At its root
this bijection is a manifestation of the tame representation theory of
the type~$A$ quiver; that is the context in which bar codes were
invented by Abeasis and Del Fra, who called them ``diagrams of boxes''
\cite{abeasis-delFra1980,abeasis-delFra-kraft1981}.  Subsequent
terminology for objects equivalent to these diagrams of boxes include
bar codes themselves (see \cite{ghrist2008}) and planar depictions
discovered effectively simultaneously in topological data analysis,
where they are called persistence diagrams
\cite{edelsbrunner-letscher-zomorodian2002} (see
\cite{cohenSteiner-edelsbrunner-harer2007} for attribution) and
combinatorial algebraic geometry, where they are called lace arrays
\cite{quivers}.

No combinatorial analogue of the bar code can classify modules over an
arbitrary poset because there are too many indecomposable modules,
even over seemingly well behaved posets like~$\ZZ^n$
\cite{multiparamPH}: the indecomposables come in families of positive
dimension.  Over arbitrary posets, every tame module does still admit
a decomposition of the Krull--Schmidt sort, namely as a direct sum of
indecomposables \cite{botnan-crawley-boevey2018}, but again, there are
too many indecomposables for this to be useful in general.

This paper makes no attempt to define bar codes or persistence
diagrams for modules over arbitrary posets.  Instead of decomposing
modules as direct sums of elemental pieces, which could in effect be
arbitrarily complicated, the commutative algebra view advocates
expressing poset modules in terms of decidedly simpler modules,
especially indicator modules for upsets and downsets, by way of less
rigid constructions like fringe presentation, primary decomposition,
or resolution.  This relaxes the direct sum in a $K$-theoretic way,
allowing arbitrary complexes instead of split short exact sequences.

Consequently, various aspects of bar codes are reflected in the
equivalent concepts of tameness.  The finitely many regions of
constancy are seen in topological tameness by constant subdivision.
The matching between left and right endpoints is seen in algebraic
tameness by fringe presentation, where the left endpoints form lower
borders of birth upsets and the right endpoints form upper borders of
death downsets.  The expressions of modules in terms of bars is seen,
in its relaxed form, in homological tameness, where modules become
``virtual'' sums, in the sense of being formal alternating
combinations rather than direct sums.  Primary decomposition isolates
elements that would, in a bar code, lie in bars unbounded in fixed
sets~of~directions~(see~also~\cite{harrington-otter-schenck-tillmann2019}).

Implicit in the notion of bar code is some concept of minimality: left
endpoints must correspond to \emph{minimal} generators, and right
endpoints to \emph{minimal} cogenerators.  These are not available
over arbitrary posets and are subtle to define and handle properly
even for partially ordered real vector spaces \cite{essential-real}.
When minimality is available, instead of a bijection (perfect
matching) from a multiset of births to a multiset of deaths, the best
one can settle for is a linear map from a filtration of the birth
multiset to a filtration of the death multiset
\cite{functorial-multiPH}.  The linear map in the case of a perfect
matching from left endpoints to right endpoints is represented by an
identity matrix, assuming that the left and right endpoints are
ordered consistently.

\subsection{Derived applications}\label{sub:dervied}

The syzygy theorem for poset modules (Theorem~\ref{t:syzygy}) is
intentionally stated in the most elementary language possible, without
sheaves, functors, or derived categories.  But its content has deep
interpretations in these enriched contexts.  Section~\ref{s:derived}
demonstrates this power by proving two conjectures made by Kashiwara
and Schapira.  The first concerns the relationship between, on one
hand, constructibility of sheaves on real vector spaces in the derived
category with microsupport restricted to a cone, and on the other
hand, stratification of the vector space in a manner compatible with
the cone \cite[Conjecture~3.17]{kashiwara-schapira2017}.%
	\footnote{Bibliographic note: this conjecture appears in~v3
	(the version cited here) and earlier versions of the cited
	arXiv preprint.  It does not appear in the published version
	\cite{kashiwara-schapira2018}, which is~v6 on the arXiv.  The
	published version is cited where it is possible to do so, and
	v3 \cite{kashiwara-schapira2017} is cited otherwise.}
The second concerns the case of piecewise linear (PL) objects in this
setting, particularly the existence of polyhedrally structured
resolutions that, in principle, lend themselves to explicit or
algorithmic computation
\mbox{\cite[Conjecture~3.20]{kashiwara-schapira2019}}.  Both
conjectures follow from Theorem~\ref{t:res}, which is essentially a
translation of the relevant real-vector-space special cases of the
syzygy theorem for complexes of poset modules
(Theorem~\ref{t:syzygy-complexes}) into the language of derived
categories of constructible sheaves with conic
microsupport~or~\mbox{under}~a~%
conic~topology.

The theory in Sections~\ref{s:tame}--\ref{s:syzygy} was developed
simultaneously and independently from that in
\cite{kashiwara-schapira2018,kashiwara-schapira2019}, cf.~\S2--5 in
\cite{qr-codes}.  Having made the connection between these approaches,
it is worth comparing them in more detail.

The syzygy theorem here (Theorem~\ref{t:syzygy}) and its combinatorial
underpinnings in Section~\ref{s:encoding} hold over arbitrary posets.
When the poset is a real vector space, the constructibility
encapsulated by topological tameness (Definition~\ref{d:tame}) has no
subanalytic, algebraic, or piecewise-linear hypothesis, although these
additional structures are preserved by the syzygy theorem transitions.
For example, the upper boundary of a downset in the plane with the
usual componentwise partial order could be the graph of any continuous
weakly decreasing function, among other things, and could be present
(i.e.,~the downset is closed) or absent (i.e.,~the downset is open),
or somewhere in between (e.g., a Cantor set could be missing).  The
conic topology in \cite{kashiwara-schapira2018} or
\cite{kashiwara-schapira2019} specializes at the outset to the case of
a partially ordered real vector space, and it allows only subanalytic
or polyhedral regions, respectively, with upsets having closed lower
boundaries and downsets having open upper boundaries.  The
constructibility in
\cite{kashiwara-schapira2018,kashiwara-schapira2019} is otherwise
essentially the same as tameness here, except that tameness requires
constant subdivisions to be finite, whereas constructibility in the
derived category allows constant subdivisions to be locally finite.
That said, this agreement of constructibility with locally finite
tameness that is subanalytic or PL, more or less up to boundary
considerations, is visible in \cite{kashiwara-schapira2017} or
\cite{kashiwara-schapira2019} only via conjectures, namely the ones
proved in Section~\ref{s:derived} using the general poset methods
here.

The theory of primary decomposition in Section~\ref{s:decomp} requires
the poset to be a polyhedral group (Definition~\ref{d:face}): a
partially ordered group whose positive cone has finitely many faces.
Polyhedral groups can be integer or real or something in between, but
the finiteness is essential for primary decomposition in any of these
settings; see Example~\ref{e:circular-cone}.  The local finiteness
allowed by the usual constructibility in
\cite{kashiwara-schapira2018} does not provide a remedy, although it
is possible that the PL hypothesis in \cite{kashiwara-schapira2019}
does.  Note that, in either the integer case or the real case,
detailed understanding of the topology results in a stronger theory of
primary decomposition than over an aribtrary polyhedral group, with
much more complete supporting commutative algebra \cite{essential-real}.

Most of the remaining differences between the developments here and
those in \cite{kashiwara-schapira2018,kashiwara-schapira2019}, beyond
the types of allowed functions and the shapes of allowed regions, is
the behavior allowed on the boundaries of regions.  That difference is
accounted for by the transition between the conic topology and the
Alexandrov topology, the distinction being that the Alexandrov
topology has for its open sets all upsets, whereas the conic topology
has only the upsets that are open in the usual topology.  This
distinction is explored in detail by Berkouk and Petit
\cite{berkouk-petit2019}.  It is intriguing that ephemeral modules are
undetectable metrically \cite[Theorem~4.22]{berkouk-petit2019} but
their presence here brings indispensable insight into homological
behavior in the conic~topology.

\section{Tame poset modules}\label{s:tame}

\subsection{Modules over posets and persistence}\label{sub:persistence}

\begin{defn}\label{d:poset-module}
Let $\cQ$ be a partially ordered set (\emph{poset}) and~$\preceq$ its
partial order.  A \emph{module over~$\cQ$} (or a \emph{$\cQ$-module})
is
\begin{itemize}
\item%
a $Q$-graded vector space $\cM = \bigoplus_{q\in Q} \cM_q$ with
\item%
a homomorphism $\cM_q \to \cM_{q'}$ whenever $q \preceq q'$ in~$Q$
such that
\item%
$\cM_q \to \cM_{q''}$ equals the composite $\cM_q \to \cM_{q'} \to
\cM_{q''}$ whenever $q \preceq q' \preceq q''$.
\end{itemize}
A \emph{homomorphism} $\cM \to \cN$ of $\cQ$-modules is a
degree-preserving linear map, or equivalently a collection of vector
space homomorphisms $\cM_q \to \cN_q$, that commute with the structure
homomorphisms $\cM_q \to \cM_{q'}$ and $\cN_q \to \cN_{q'}$.
\end{defn}

The last bulleted item is \emph{commutativity}: it reflects that
inclusions of subspaces induce functorial maps on homology in the
motivating examples of $\cQ$-modules
(\mbox{Example}~\ref{e:multifiltration}).

\begin{remark}\label{r:poset-module}
Definition~\ref{d:poset-module} is same as the concept of commutative
module over a directed acyclic graph
\cite[Definition~2.2]{chambers-letscher2018}.  (The language of posets
is equivalent as noted before
\cite[Definition~2.1]{chambers-letscher2018}.)

When the poset~$\cQ$ is locally finite (every interval is finite), a
$\cQ$-module is the same a module over the incidence algebra of~$\cQ$
\cite{doubilet-rota-stanley1972}.  But local finiteness, or a
similarly restrictive finitary hypothesis on the module
\cite{khripchenko-novikov2009}, excludes the basic, motivating
examples of real multifiltrations.
\end{remark}

\begin{example}\label{e:multifiltration}
Let $X$ be a topological space and $\cQ$ a poset.
\begin{enumerate}
\item\label{i:filtration}%
A \emph{filtration of~$X$ indexed by~$\cQ$} is a choice of subspace
$X_q \subseteq X$ for each $q \in \cQ$ such that $X_q \subseteq
X_{q'}$ whenever $q \preceq q'$.

\item\label{i:PH}%
The \emph{$i^\mathrm{th}$ persistent homology} of the \emph{filtered
space}~$X$ is the associated homology module, meaning the $\cQ$-module
$\bigoplus_{q \in \cQ} H_i X_q$.
\end{enumerate}
\end{example}

\begin{remark}\label{r:curry}
There are a number of abstract, equivalent ways to phrase
Example~\ref{e:multifiltration}.  For example, a filtration is a
functor from $\cP$ to the category $\mathcal{S}$ of subspaces of~$X$
or an $\mathcal{S}$-valued sheaf on the topological space~$\cP$, where
a base for its \emph{Alexandrov topology} is the set of principal
upsets (i.e., principal dual order~ideals).  For background on and
applications of many of these perspectives, see Curry's dissertation
\cite{curry-thesis}, particularly \S4.2 there.  Some of these
perspectives arise in full force in Section~\ref{s:derived}.
\end{remark}

\begin{conv}\label{c:kk}
The homology here could be taken over an arbitrary ring, but for
simplicity it is assumed throughout that homology is taken with
coefficients in a field~$\kk$.
\end{conv}

\begin{example}\label{e:RR+}
A \emph{real multifiltration} of~$X$ is a filtration indexed
by~$\RR^n$, with its partial order by coordinatewise comparison.
Example~\ref{e:fly-wing-filtration} is a real multifiltration of $X =
\RR^2$ with $n = 2$.  The monoid $\RR_+^n \subset \RR^n$ of
nonnegative real vectors under addition has monoid algebra
$\kk[\RR_+^n]$ over the field~$\kk$, a ``polynomial'' ring whose
elements are (finite) linear combinations of monomials $\xx^\aa$ with
real, nonnegative exponent vectors $\aa = (a_1,\dots,a_n) \in
\RR_+^n$.  It contains the usual polynomial ring $\kk[\NN^n]$ as a
$\kk$-subalgebra.  The persistent homology of a real $n$-filtered
space~$X$ is a \emph{multipersistence module}: an $\RR^n$-graded
module over\/~$\kk[\RR_+^n]$, which is the same thing as an
$\RR^n$-module \cite[\S2.1]{lesnick-interleav2015}.
\end{example}

\subsection{Constant subdivisions}\label{sub:constant}

\begin{defn}\label{d:constant-subdivision}
Fix a $\cQ$-module~$\cM$.  A \emph{constant subdivision} of~$\cQ$
\emph{subordinate} to~$\cM$ is a partition of~$\cQ$ into
\emph{constant regions} such that for each constant region~$I$ there
is a single vector space~$M_I$ with an isomorphism $M_I \to M_\ii$ for
all $\ii \in I$ that \emph{has no monodromy}: if $J$ is some (perhaps
different) constant region, then all comparable pairs $\ii \preceq
\jj$ with $\ii \in I$ and $\jj \in J$ induce the same composite
homomorphism $M_I \to M_\ii \to M_\jj \to M_J$.
\end{defn}

\begin{example}\label{e:puuska-nonconstant-isotypic}
Consider the poset module (kindly provided by Ville Puuska
\cite{puuska18})
$$%
\psfrag{1}{\tiny$1$}
\psfrag{2}{\tiny$2$}
\psfrag{kk}{\tiny$\kk$}
\begin{array}{@{}c@{}}\includegraphics[height=30mm]{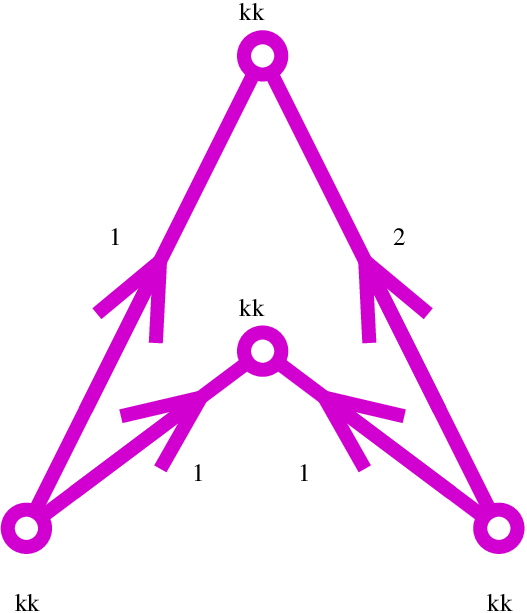}\end{array}
$$
in which the structure morphisms $M_\aa \to M_\bb$ are all identity
maps on~$\kk$, except for the rightmost one.  This example
demonstrates that module structures need not be recoverable from their
\emph{isotypic subdivision}, in which elements of~$\cQ$ lie in the
same region when their vector spaces are isomorphic via a poset
relation.  In cases like this, refining the isotypic subdivision
appropriately yields a constant subdivision.  Here, the two minimum
elements must lie in different constant regions
and the two maximum elements must lie in different constant regions.
Any partition accomplishing these separations---that is, any
refinement of a partition that has a region consisting of precisely
one maximum and one minimum---is a constant subdivision.  Of course, a
finite poset always admits a constant subdivision with finitely many
regions, since~the~\mbox{partition}~into~\mbox{singletons}~works.
\end{example}

\begin{example}\label{e:antidiagonal}
Constant subdivisions need not refine the isotypic subdivision in
Example~\ref{e:puuska-nonconstant-isotypic}, one reason being that a
single constant region can contain two or more incomparable isotypic
regions.  For a concrete instance with a single constant region
comprised of uncountably many incomparable isotypic regions, let $\cM$
be the $\RR^2$-module that has $\cM_\aa = 0$ for all $\aa \in \RR^2$
except for those on the antidiagonal line spanned by $\left[\twoline
1{-1}\right] \in \RR^2$, where $\cM_\aa = \kk$.  There is only one
such $\RR^2$-module because all of the degrees of nonzero graded
pieces of~$\cM$ are incomparable, so all of the structure
homomorphisms $\cM_\aa \to \cM_\bb$ with $\aa \neq \bb$ are zero.
Every point on the line is a singleton isotypic region.
\end{example}

The direction of the line in Example~\ref{e:antidiagonal} is
important: antidiagonal lines, whose points form an antichain
in~$\RR^2$, behave radically differently than diagonal lines.

\begin{example}\label{e:diagonal}
Let $\cM$ be an $\RR^2$-module with $\cM_\aa = \kk$ whenever $\aa$
lies in the closed diagonal strip between the lines of slope~$1$
passing through any pair of points.  The structure homomorphisms
$\cM_\aa \to \cM_\bb$ could all be zero, for instance, or some of them
could be nonzero.  But the length $|\aa - \bb|$ of any nonzero such
homomorphism must in any case be bounded above by the Manhattan (i.e.,
$\ell^\infty$) distance between the two points, since every longer
structure homomorphism factors through a sequence that exits and
re-enters the strip.
$$%
\psfrag{0}{\tiny$0$}
\begin{array}{@{}c@{}}\includegraphics[height=45mm]{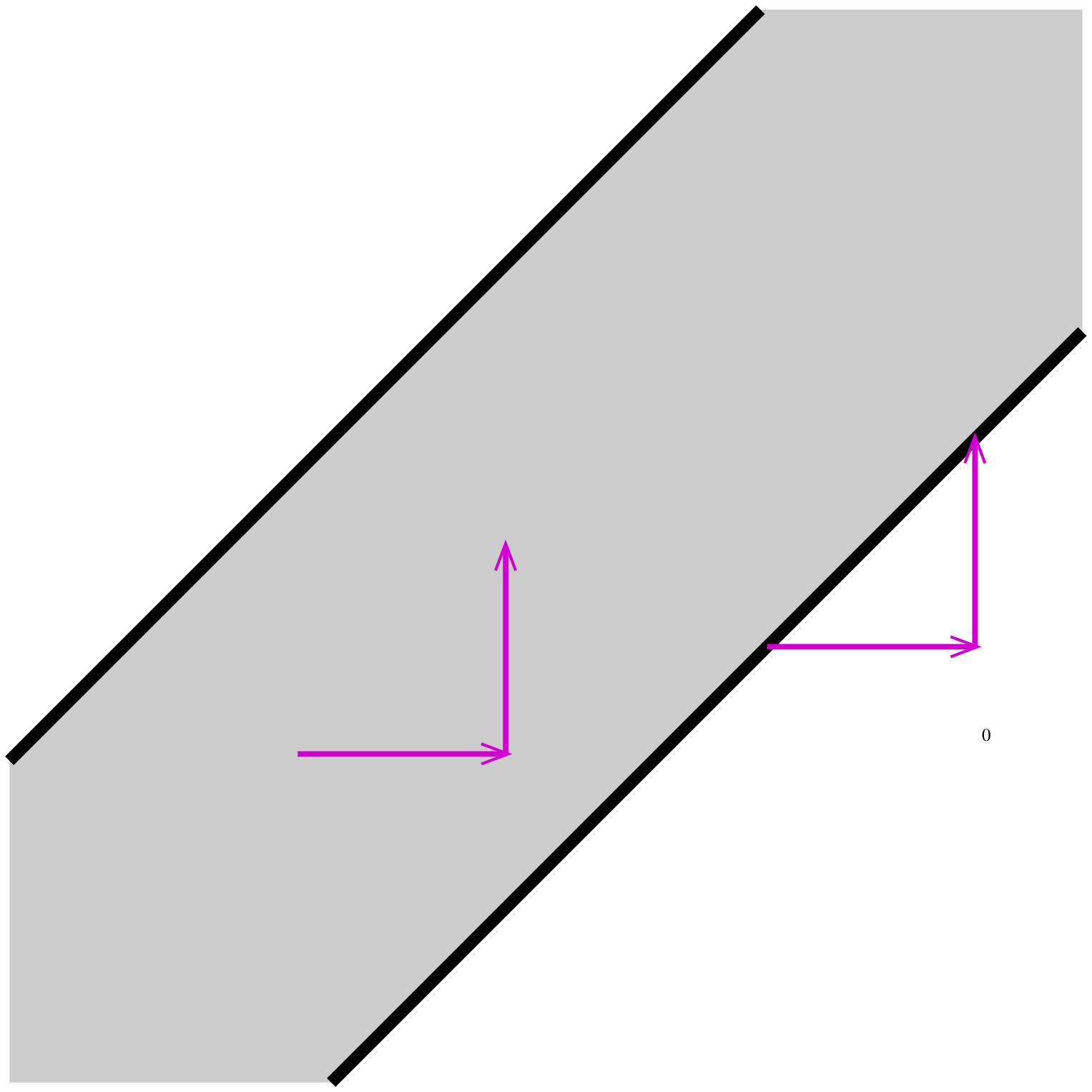}\end{array}
$$
In particular, the structure homomorphism between any pair of points
on the upper boundary line of the strip is zero because it factors
through a homomorphism that points upward first; therefore such pairs
of points lie in distinct regions of any constant subdivision.  The
same conclusion holds for pairs of points on the lower boundary line
of the strip.  When the strip has width~$0$, so the upper and lower
boundary coincide, the module is supported along a diagonal line whose
uncountably many points must all lie in distinct constant regions.
\end{example}

The reference to ``no monodromy'' in
Definition~\ref{d:constant-subdivision} agrees with the usual notion.

\begin{lemma}\label{l:no-monodromy}
Fix a constant region~$I$ subordinate to a poset module~$\cM$.  The
composite isomorphism $M_I \to M_\ii \to \dots \to M_{\ii'} \to M_I$
is independent of the path from~$\ii$ to~$\ii'$ through~$I$, if one
exists.  In particular, when $\ii = \ii'$ the composite is the
identity on~$M_I$.
\end{lemma}
\begin{proof}
The second claim follows from the first.  When the path has
length~$0$, the claim is that $M_I \to M_\ii \to M_I$ is the identity
on~$M_I$, which follows by definition.  For longer paths the result is
proved by induction on path length.
\end{proof}

Now it is time to introduce the central finiteness concept of the
paper.

\begin{defn}\label{d:tame}
Fix a poset~$\cQ$ and a $\cQ$-module~$M$.
\mbox{}
\begin{enumerate}
\item\label{i:finite-constant-subdiv}%
A constant subdivision of~$\cQ$ is \emph{finite} if it has finitely
many constant regions.

\item\label{i:Q-finite}%
The $\cQ$-module~$M$ is \emph{$\cQ$-finite} if its components $M_q$
have finite dimension over~$\kk$.

\item\label{i:tame}%
The $\cQ$-module~$M$ is \emph{tame} if it is $\cQ$-finite and $\cQ$
admits a finite constant subdivision subordinate to~$M$.
\end{enumerate}
\end{defn}

\pagebreak
\begin{remark}\label{r:tame}
\mbox{}
\begin{enumerate}
\item%
The tameness condition here includes but is much less rigid than the
compact tameness condition in
\cite{scolamiero-chacholski-lundman-ramanujam-oberg16}, the latter
meaning more or less that the module is finitely generated over a
scalar multiple of~$\ZZ^n$ in~$\QQ^n$.

\item%
Some literature calls Definition~\ref{d:tame}.\ref{i:Q-finite}
\emph{pointwise finite dimensional (PFD)}.  The terminology here
agrees with that in \cite{alexdual}, on which Section~\ref{s:ZZn} here
is based.
\end{enumerate}
\end{remark}

\begin{remark}\label{r:isotypic-expected}
Data analysis should always produce tame persistence modules.  Indeed,
from limited experience, in data analysis the isotypic subdivision
(Example~\ref{e:puuska-nonconstant-isotypic}) appears consistently to
be a finite constant subdivision.  It is unclear what kind of data
situation might produce a nonconstant isotypic subdivision, but it
seems likely that in such cases a constant subdivision can always be
obtained by subdividing---into contractible pieces---isotypic
components that have nontrivial topology and then gathering
incomparable contractible pieces into single constant regions.
Exploring these assertions is left open.
\end{remark}

\begin{lemma}\label{l:subdiv}
Any refinement of a constant subdivision subordinate to a
$\cQ$-module~$\cM$ is a constant subdivision subordinate to~$\cM$.
\end{lemma}
\begin{proof}
Choosing the same vector space~$M_I$ for every region of the
refinement contained in the constant region~$I$, the lemma is
immediate from Definition~\ref{d:constant-subdivision}.
\end{proof}

\subsection{Auxiliary hypotheses}\label{sub:auxiliary}\mbox{}

\noindent
Effectively computing with real multifiltered spaces requires keeping
track of the shapes of various regions, such as constant regions.  (In
later sections, other regions along these lines include upsets,
downsets, and fibers of poset morphisms.)  The fact that applications
of persistent homology often arise from metric considerations, which
are semialgebraic in nature, or are approximated by piecewise linear
structures, suggests the following auxiliary hypotheses for
algorithmic developments.  The subanalytic hypothesis is singled out
for the theoretical purposes surrounding conjectures of Kashiwara and
Schapira in Section~\ref{s:derived}.

\begin{defn}\label{d:auxiliary-hypotheses}
Fix a subposet~$\cQ$ of a partially ordered real vector space of
finite dimension (see Definition~\ref{d:pogroup}, or take $\cQ =
\RR^n$ for now).  A partition of~$\cQ$ into subsets~is
\begin{enumerate}
\item\label{i:semialgebraic}%
\emph{semialgebraic} if the subsets are real semialgebraic varieties;

\item\label{i:PL}%
\emph{piecewise linear (PL)} if the subsets are finite unions of
convex polyhedra, where a \emph{convex polyhedron} is an intersection
of finitely many closed or open half-spaces;

\item\label{i:subanalytic}%
\emph{subanalytic} if the subsets are subanalytic varieties;

\item\label{i:classX}%
\emph{of class~$\mathfrak X$} if the subsets lie in a
family~$\mathfrak X$ of subsets of~$\cQ$ that is closed under
complements, finite intersections, negation, and Minkowski sum with
the \emph{positive cone} $\cQ_+$, namely the set of vectors $\qq \in
\cQ$ such that $\0 \preceq \qq$.
\end{enumerate}
A module over~$\cQ$ is \emph{semialgebraic}, or \emph{PL},
\emph{subanalytic}, or \emph{of class~$\mathfrak X$} if it is tamed by
a subordinate finite constant subdivision of the corresponding type.
\end{defn}

\begin{remark}\label{r:auxiliary-hypotheses}
Subposets of partially ordered real vector spaces are allowed in
Definition~\ref{d:auxiliary-hypotheses} to be able to speak of, for
example, piecewise linear sets in rational vector spaces, or
semialgebraic subsets of~$\ZZ^n$ (see Example~\ref{e:circular-cone}
for an instance of the latter).  When $\cQ$ is properly contained in
the ambient real vector space, subsets of~$\cQ$ are semialgebraic, PL,
or subanalytic when they are intersections with~$\cQ$ of the
corresponding type of subset of the ambient real vector space.
\end{remark}

\begin{prop}\label{p:auxiliary-hypotheses}
Fix a partially ordered real vector space~$\cQ$.
\begin{enumerate}
\item\label{i:classes}%
The classes of semialgebraic, PL, and subanalytic subsets of~$\cQ$ are
each closed under complements, finite intersections, and negation.

\item\label{i:sum-semialg}%
The Minkowski sum $S + \cQ_+$ of a semialgebraic set~$S$ with the
positive cone is semialgebraic if~$\cQ_+$ is semialgebraic.

\item\label{i:sum-PL}%
The Minkowski sum $S + \cQ_+$ of a PL set with the positive cone is
semialgebraic if~$\cQ_+$ is polyhedral.

\item\label{i:sum-subanalytic}%
The Minkowski sum $S + \cQ_+$ of a bounded subanalytic set~$S$ with
the positive cone is subanalytic if~$\cQ_+$ is subanalytic.
\end{enumerate}
\end{prop}
\begin{proof}
See \cite{shiota97} (for example) to treat the semialgebraic and
subanalytic cases of item~\ref{i:classes}.  The PL case reduces easily
to checking that the complement of a single polyhedron is PL, which in
turn follows because a real vector space is the union of the
(relatively open) faces in any finite hyperplane arrangement, so
removing a single one of these faces leaves a PL set remaining.

For item~\ref{i:sum-semialg}, use that the image of a semialgebraic
set under linear projection is a semialgebraic set, and then express
$S + \cQ_+$ as the image of $S \times \cQ_+$ under the projection $\cQ
\times \cQ \to \cQ$ that acts by $(\qq,\qq') \mapsto \qq + \qq'$.  The
same argument works for item~\ref{i:sum-PL}.  The same argument also
works for item~\ref{i:sum-subanalytic} but requires that the
restriction of the projection to the closure of $S \times \cQ_+$ be a
proper map, which always occurs when $S$ is bounded.
\end{proof}

\section{Fringe presentation by upsets and downsets}\label{s:fringe}

Presentations are the core structures of computational ring theory.
However, the default notions of free or projective presentation---as
well as more generally and dually the notions of flat or injective
presentation---are too restrictive for modules over arbitrary posets.
These notions all still make formal sense, but it is too much to ask
for presentations that are finite direct sums of indecomposables in
such generality, as demonstrated by the formidable infinitude of such
objects in the case of $\RR^n$-modules like fly wing modules (see
Example~\ref{e:toy-model-fly-wing}).  The idea here, both for
theoretical and computational purposes, is to allow arbitrary upset
and downset modules instead of only flat and injective ones.  The use
of indicator modules instead of free or injective modules gathers the
generators and cogenerators, known as births and deaths in persistent
homology, into finitely many groups (Definition~\ref{d:fringe}),
paving the way for effective data structures in the form of monomial
matrices (Definition~\ref{d:monomial-matrix-fr}).

\subsection{Upsets and downsets}\label{sub:indicator}

\begin{defn}\label{d:indicator}
The vector space $\kk[\cQ] = \bigoplus_{q\in\cQ} \kk$ that assigns
$\kk$ to every point of the poset~$\cQ$ is a $\cQ$-module with
identity maps on~$\kk$.  More generally,
\begin{enumerate}
\item\label{i:upset}%
an \emph{upset} (also called a \emph{dual order ideal}) $U \subseteq
\cQ$, meaning a subset closed under going upward in~$\cQ$ (so $U +
\RR_+^n = U$, when $\cQ = \RR^n$) determines an \emph{indicator
submodule} or \emph{upset module} $\kk[U] \subseteq \kk[\cQ]$; and
\item\label{i:downset}%
dually, a \emph{downset} (also called an \emph{order ideal}) $D
\subseteq \cQ$, meaning a subset closed under going downward in~$\cQ$
(so $D - \RR_+^n = D$, when $\cQ = \RR^n$) determines an
\emph{indicator quotient module} or \emph{downset module} $\kk[\cQ]
\onto \kk[D]$.
\end{enumerate}
When $\cQ$ is a subposet of a partially ordered real vector space, an
indicator module of either sort is semialgebraic, PL, subanalytic, or
of class~$\mathfrak X$ if the corresponding upset or downset is of the
same type (Definition~\ref{d:auxiliary-hypotheses}).
\end{defn}

\begin{remark}\label{r:indicator}
Indicator submodules $\kk[U]$ and quotient modules $\kk[D]$ are
$\cQ$-modules, not merely $U$-modules or $D$-modules, by setting the
graded components indexed by elements outside of the ideals to~$0$.
It is only by viewing indicator modules as $\cQ$-modules that they are
forced to be submodules or quotients, respectively.  For relations
between these notions and those in Remark~\ref{r:curry}, again see
Curry's thesis \cite{curry-thesis}.  For example, upsets form the open
sets in the topology from Remark~\ref{r:curry}.
\end{remark}

\begin{example}\label{e:melting}
Ising crystals at zero temperature, with polygonal boundary conditions
and fixed mesh size, are semialgebraic upsets in~$\RR^n$.  That much
is by definition: fixing a mesh size means that the crystals in
question are (staircase surfaces of finitely generated) monomial
ideals in $n$ variables.  Remarkably, such crystals remain
semialgebraic in the limit of zero mesh size; see \cite{okounkov16}
for an exposition and references.
\end{example}

\begin{example}\label{e:asw}
Monomial ideals in polynomial rings with real exponents, which
correspond to upsets in $\RR^n_+$, are considered in
\cite{andersen--sather-wagstaff2015}, including aspects of primality,
irreducible decomposition, and Krull dimension.  Upsets in $\RR^n$ are
also considered in \cite{madden-mcguire2015}, where the combinatorics
of their lower boundaries, and topology of related simplicial
complexes, are investigated in cases with locally finite generating
sets.
\end{example}

\begin{defn}\label{d:connected-poset}
A poset~$\cQ$ is
\begin{enumerate}
\item\label{i:connected}%
\emph{connected} if every pair of elements $q,q' \in \cQ$ is joined by
a \emph{path} in~$\cQ$: a sequence $q = q_0 \preceq q'_0 \succeq q_1
\preceq q'_1 \succeq \dots \succeq q_k \preceq q'_k = q'$ in~$\cQ$;

\item%
\emph{upper-connected} if every pair of elements in~$\cQ$ has an
upper bound in~$\cQ$;

\item%
\emph{lower-connected} if every pair of elements in~$\cQ$ has a
lower bound in~$\cQ$; and

\item%
\emph{strongly connected} if $\cQ$ is upper-connected and
lower-connected.
\end{enumerate}
\end{defn}

\begin{example}\label{e:connected-poset}
$\RR^n$ is strongly connected.  The same is true of any partially
ordered abelian group (see Section~\ref{sub:polyhedral} for basic
theory of those posets).
\end{example}

\begin{example}\label{e:bounded-poset}
A poset~$\cQ$ is upper-connected if (but not only if,
cf.~Example~\ref{e:connected-poset}) it has a maximum element---one
that is preceded by every element of~$\cQ$.  Similarly, $\cQ$ is
lower-connected if it has a minimum element---one that precedes every
element of~$\cQ$.
\end{example}

\begin{remark}\label{r:pi0}
The relation $q \sim q'$ defined by the existence of a path joining
$q$ to~$q'$ as in Definition~\ref{d:connected-poset}.\ref{i:connected}
is an equivalence relation.
\end{remark}

\begin{defn}\label{d:pi0}
Fix a poset~$\cQ$.  For any subset $S \subseteq \cQ$, write $\pi_0 S$
for the set of connected components of~$S$: the maximal connected
subsets of~$S$, or equivalently the classes under the relation from
Remark~\ref{r:pi0}.
\end{defn}

\begin{prop}\label{p:U->D}
Fix a poset~$\cQ$.
\begin{enumerate}
\item\label{i:U->D}%
For an upset~$U$ and a downset~$D$,
$$%
  \Hom_\cQ(\kk[U], \kk[D]) = \kk^{\pi_0(U \cap D)},
$$
a product of copies of\/~$\kk$, one for each connected component of $U
\cap D$.

\item\label{i:U'->U}%
For upsets $U$ and~$U'$,
$$%
  \Hom_\cQ(\kk[U'], \kk[U]) = \kk^{\{S\in\pi_0 U' \,\mid\, S \subseteq U\}},
$$
a product of copies of\/~$\kk$, one for each connected component
of~$U'$ contained in~$U$.

\item\label{i:D->D'}%
For downsets $D$ and~$D'$,
$$%
  \Hom_\cQ(\kk[D], \kk[D']) = \kk^{\{S\in\pi_0 D' \,\mid\, S \subseteq D\}},
$$
a product of copies of\/~$\kk$, one for each connected component
of~$D'$ contained in~$D$.
\end{enumerate}
\end{prop}
\begin{proof}
For the first claim, the action~$\phi_q$ of $\phi: \kk[U] \to \kk[D]$
on the copy of~$\kk$ in any degree $q \in U \minus D$ is~$0$ because
$\kk[D]_q = 0$, so assume $q \in U \cap D$.  Then $\phi_q = \phi_{q'}:
\kk \to \kk$ if $q \preceq q' \in U \cap D$ because $\kk[U]_q \to
\kk[U]_{q'}$ and $\kk[D]_q \to \kk[D]_{q'}$ are identity maps
on~$\kk$.  Similarly, $\phi_q = \phi_{q'}$ if $q \succeq q' \in U \cap
D$.  Induction on the length of the path in
Definition~\ref{d:connected-poset}.\ref{i:connected} shows that
$\phi_q = \phi_{q'}$ if $q$ and~$q'$ lie in the same connected
component of $U \cap D$.  Thus $\Hom_\cQ(\kk[U], \kk[D]) \subseteq
\kk^{\pi_0(U \cap D)}$.  On the other hand, specifying for each
component $S \in \pi_0(U \cap D)$ a scalar $\alpha_S \in \kk$ yields a
homomorphism $\phi: \kk[U] \to \kk[D]$, if $\phi_q$ is defined to be
multiplication by~$\alpha_S$ on the copies of $\kk = \kk[U]_q$ indexed
by $q \in S$ and $0$ for $q \in U \minus D$; that $\phi$ is indeed a
$\cQ$-module homomorphism follows because $\kk[D]_{q'} = 0$ (that is,
$q' \not\in D$) whenever $q' \succeq q \in D$ but $q'$ does not lie in
the connected component of~$U \cap D$ containing~$q$.  Said another
way, pairs of elements of $U \cap D$ either lie in the same connected
component of $U \cap D$ or they are incomparable.

The proofs of the last two claims are similar (and dual to one
another), particularly when it comes to showing that a homomorphism of
indicator modules of the same type---that is, source and target both
upset or both downset---is constant on the relevant connected
components.  The only point not already covered is that if $U'$ is a
connected upset and $U' \not\subseteq U$ then every homomorphism
$\kk[U'] \to \kk[U]$ is~$0$ because $q' \in U' \minus U$ implies
$\kk[U']_{q'} \to 0 = \kk[U]_{q'}$.
\end{proof}

The cases of interest in this paper and its sequels
\cite{essential-real,functorial-multiPH}, particulary real and discrete
polyhedral partially ordered groups (Example~\ref{e:real-pogroup},
Example~\ref{e:discrete-pogroup}, and Definition~\ref{d:face}) such
as~$\RR^n$ and~$\ZZ^n$, have strong connectivity properties, thereby
simplifying the conclusions of Proposition~\ref{p:U->D}.  First, here
is a convenient notation.

\begin{cor}\label{c:U->D}
Fix a poset~$\cQ$ with upsets $U,U'$ and downsets $D,D'$.
\begin{enumerate}
\item\label{i:kk}%
$\Hom_\cQ(\kk[U], \kk[D]) = \kk$ if $U \cap D \neq \nothing$ and
either $U$ is lower-connected as a subposet of~$\cQ$ or $D$ is
upper-connected as a subposet of~$\cQ$.

\item\label{i:U}%
If $U$ and~$U'$ are upsets and $\cQ$ is upper-connected, then
$\Hom_\cQ(\kk[U'],\kk[U]) = \kk$ if $U' \subseteq U$ and\/~$0$
otherwise.

\item\label{i:D}%
If $D$ and~$D'$ are downsets and $\cQ$ is lower-connected, then
$\Hom_\cQ(\kk[D],\kk[D']) = \kk$ if $D \supseteq D'$ and\/~$0$
otherwise.\qed
\end{enumerate}
\end{cor}

\begin{example}\label{e:disconnected-homomorphism}
Consider the poset $\NN^2$, the upset $U = \NN^2 \minus \{\0\}$, and
the downset $D$ consisting of the origin and the two standard basis
vectors.  Then $\kk[U] = \mm = \<x,y\>$ is the graded maximal ideal of
$\kk[\NN^2] = \kk[x,y]$ and $\kk[D] = \kk[\NN^2]/\mm^2$.  Now
calculate
$$%
  \Hom_{\NN^2}(\kk[U],\kk[D])
  =
  \Hom_{\NN^2}(\mm,\kk[\NN^2]/\mm^2)
  =
  \kk^2,
$$
a vector space of dimension~$2$: one basis vector preserves the
monomial~$x$ while killing the monomial~$y$, and the other basis
vector preserves~$y$ while killing~$x$.
\end{example}

\begin{example}\label{e:totally-disconnected}
For an extreme case, consider the poset $\cQ = \RR^2$ with $U$ the
closed half-plane above the antidiagonal line $y = -x$ and $D = -U$,
so that $U \cap D$ is totally disconnected: $\pi_0(U \cap D) = U \cap
D$.  Then $\Hom_\cQ(\kk[U],\kk[D]) = \kk^\RR$ is a vector space of
beyond continuum dimension, the copy of~$\RR$ in the exponent being
the antidiagonal~line.
\end{example}

The proliferation of homomorphisms in
Examples~\ref{e:disconnected-homomorphism}
and~\ref{e:totally-disconnected} is undesirable for both computational
and theoretical purposes; it motivates the following concept.

\begin{defn}\label{d:connected-homomorphism}
Let each of $S$ and $S'$ be a nonempty intersection of an upset in a
poset~$\cQ$ with a downset in~$\cQ$, so $\kk[S]$ and $\kk[S']$ are
subquotients of~$\kk[\cQ]$.  A homomorphism $\phi: \kk[S] \to \kk[S']$
is \emph{connected} if there is a scalar $\lambda \in \kk$ such that
$\phi$ acts as multiplication by~$\lambda$ on the copy of~$\kk$ in
degree $q$ for all $q \in S \cap S'$.
\end{defn}

The cases of interest in the rest of this paper concern three
situations: both $S$ and~$S'$ are upsets, or both are downsets, or $S
= U$ is an upset and $S' = D$ is downset with $U \cap D \neq
\nothing$.  However, the full generality of
Definition~\ref{d:connected-homomorphism} is required in the sequel to
this work \cite{essential-real}.

\begin{remark}\label{r:U->D}
Corollary~\ref{c:U->D} says that homomorphisms among indicator modules
are automatically connected in the presence of appropriate upper- or
lower-connectedness.
\end{remark}

\subsection{Fringe presentations}\label{sub:fringe}\mbox{}

\begin{defn}\label{d:fringe}
Fix any poset~$\cQ$.  A \emph{fringe presentation} of a
$\cQ$-module~$\cM$ is
\begin{itemize}
\item%
a direct sum~$F$ of upset modules~$\kk[U]$,
\item%
a direct sum~$E$ of downset modules~$\kk[D]$, and
\item%
a homomorphism $F \to E$ of $\cQ$-modules with
\begin{itemize}
  \item%
  image isomorphic to~$\cM$ and
  \item%
  components $\kk[U] \to \kk[D]$ that are connected
  (Definition~\ref{d:connected-homomorphism}).
\end{itemize}
\end{itemize}
A fringe presentation
\begin{enumerate}
\item%
is \emph{finite} if the direct sums are finite;

\item\label{i:dominate}%
\emph{dominates} a constant subdivision of~$\cM$ if the subdivision is
subordinate to each summand $\kk[U]$ of~$F$ and~$\kk[D]$ of~$E$; and

\item\label{i:auxiliary-fringe}%
is \emph{semialgebraic}, \emph{PL}, \emph{subanalytic}, or \emph{of
class~$\mathfrak X$} if $\cQ$ is a subposet of a partially ordered
real vector space of finite dimension and the fringe presentation
dominates a constant subdivision of the corresponding type
(Definition~\ref{d:auxiliary-hypotheses}).
\end{enumerate}
\end{defn}

Fringe presentations are effective data structures via the following
notational trick.  Topologically, it highlights that births occur
along the lower boundaries of the upsets and deaths occur along the
upper boundaries of the downsets, with a linear map over the ground
field to relate them.

\begin{defn}\label{d:monomial-matrix-fr}
Fix a finite fringe presentation $\phi: \bigoplus_p \kk[U_p] = F \to E
= \bigoplus_q \kk[D_q]$.  A \emph{monomial matrix} for $\phi$ is an
array of \emph{scalar entries}~$\phi_{pq}$ whose columns are labeled
by the \emph{birth upsets}~$U_p$ and whose rows are labeled by the
\emph{death downsets}~$D_q$:
$$%
\begin{array}{ccc}
  &
  \monomialmatrix
	{U_1\\\vdots\ \\U_k\\}
	{\begin{array}{ccc}
		   D_1    & \cdots &    D_\ell   \\
		\phi_{11} & \cdots & \phi_{1\ell}\\
		\vdots    & \ddots &   \vdots    \\
		\phi_{k1} & \cdots & \phi_{k\ell}\\
	 \end{array}}
	{\\\\\\}
\\
  \kk[U_1] \oplus \dots \oplus \kk[U_k] = F
  & \fillrightmap
  & E = \kk[D_1] \oplus \dots \oplus \kk[D_\ell].
\end{array}
$$
\end{defn}

\begin{prop}\label{p:scalars}
With notation as in Definition~\ref{d:monomial-matrix-fr}, $\phi_{pq}
= 0$ unless $U_p \cap D_q \neq \nothing$.  Conversely, if an array of
scalars $\phi_{pq} \in \kk$ with rows labeled by upsets and columns
label\-ed by downsets has $\phi_{pq} = 0$ unless $U_p \cap D_q \neq
\nothing$, then it represents a fringe~\mbox{presentation}.
\end{prop}
\begin{proof}
Proposition~\ref{p:U->D}.\ref{i:U->D} and
Definition~\ref{d:connected-homomorphism}.
\end{proof}

\begin{example}\label{e:one-param-fringe}
Fringe presentation in one parameter reflects the usual matching
between left endpoints and right endpoints of a module, once it has
been decomposed as a direct sum of bars.  A single bar, say an
interval $[a,b)$ that is closed on the left and open on the right, has
fringe presentation
$$%
\hspace{8ex}
\psfrag{a}{\footnotesize\raisebox{-.2ex}{$a$}}
\psfrag{b}{\footnotesize\raisebox{-.2ex}{$b$}}
\psfrag{vert-to}{\small\raisebox{-.2ex}{$\downarrow$}}
\psfrag{vert-into}{\small\raisebox{-.2ex}{$\lhookdownarrow$}}
\psfrag{vert-onto}{\small\raisebox{-.2ex}{$\twoheaddownarrow$}}
\psfrag{has image}{}
\begin{array}{@{}l@{}}
\includegraphics[height=20mm]{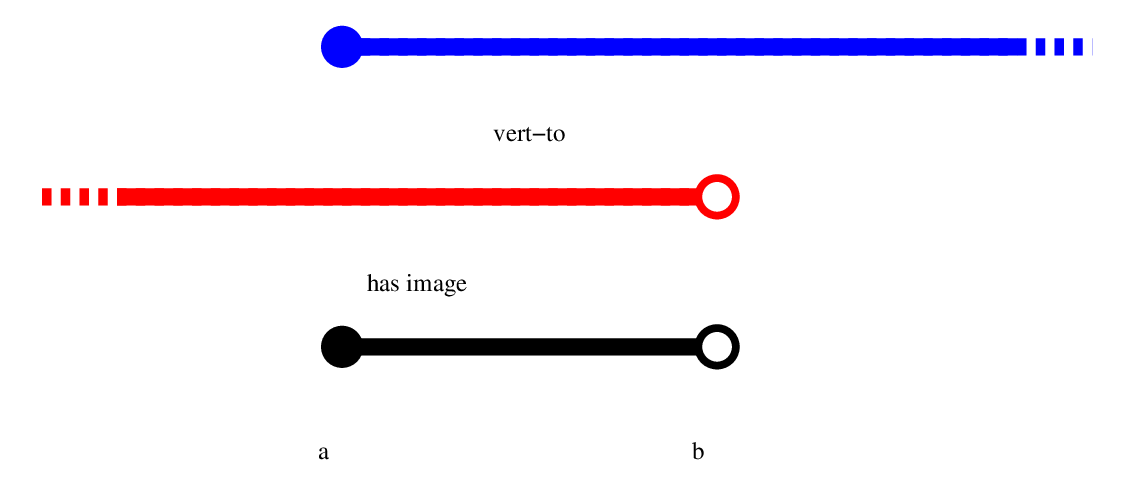}\\[-5ex]
\llap{with image\hspace{-1.5ex}}\\[2ex]
\end{array}
$$
in which a subset $S \subseteq \RR$ is drawn instead of writing
$\kk[S]$.  With multiple bars, the bijection from left to right
endpoints yields a monomial matrix whose scalar entries form the
identity, with rows labeled by positive rays having the specified left
endpoints (the ray is the whole real line when the left endpoint
is~$-\infty$) and columns labeled by negative rays having the
corresponding right endpoints (again, the whole line when the right
endpoint~is~$+\infty$).  In practical terms, the rows and columns can
be labeled simply by the endpoints themselves, with (say) a bar over a
closed endpoint and a circle over an open one.  Thus a standard bar
code, in monomial matrix notation, has the form
$$%
\begin{array}{c}
\monomialmatrix
	{ \blu \ol a_1 \\ \blu\vdots\,\ \\ \blu \ol a_k \\}
	{\begin{array}{ccc}
		\red \ob_1 & \red\cdots & \red \ob_k \\
		    1    &            &          \\
		         &  \ddots    &          \\
		         &            &     1    \\
	 \end{array}}
	{\\\\\\}.
\end{array}
$$
\end{example}

\begin{example}\label{e:two-param-fringe}
Although there are many opinions about what the multiparameter
analogue of a bar code should be, the analogue of a single bar is
generally accepted to be some kind of interval in the parameter
poset---that is, $\kk[U \cap D]$, where $U$ is an upset and $D$ is a
downset---sometimes with restrictions on the shape of the interval,
depending on context.  This case of a single bar explains the
terminology ``birth upset'' and ``death downset''.  For instance, a
fringe presentation of the yellow interval%
$$%
\hspace{-2.2ex}
\ \,\
\begin{array}{@{}r@{\hspace{-5.5pt}}|@{}l@{}}
\raisebox{-.2mm}{\includegraphics[height=25mm]{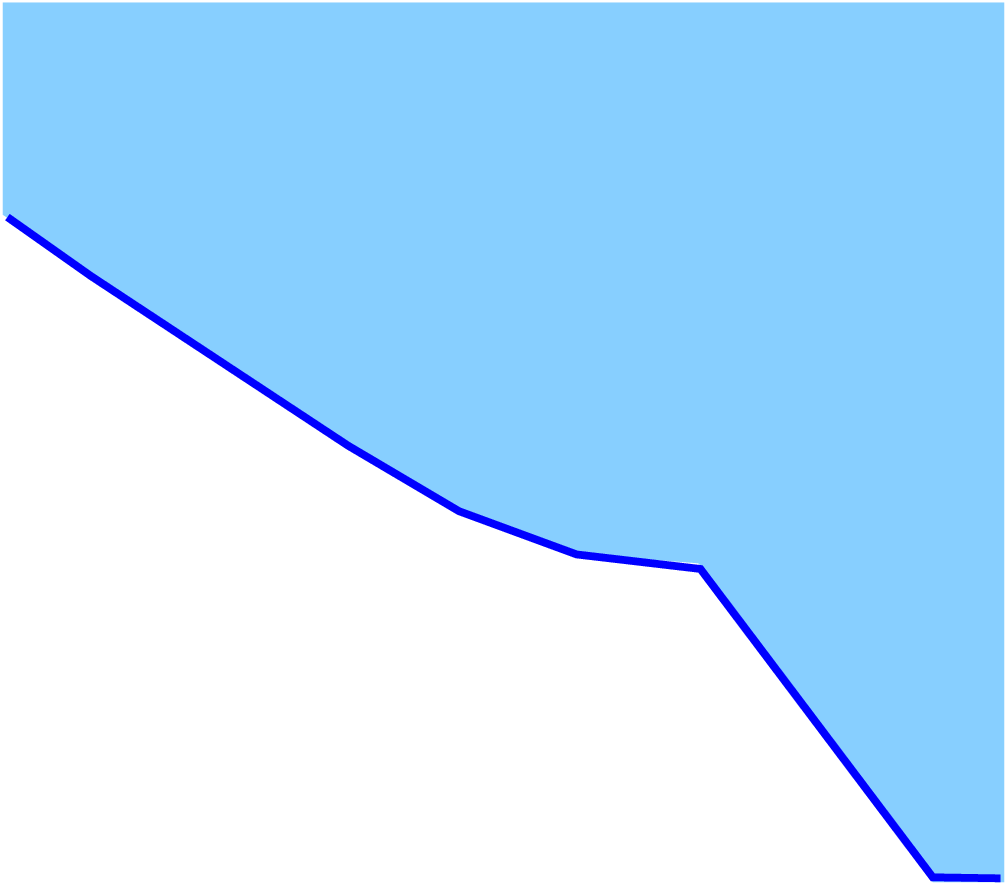}}
&\ \,\hspace{-.2pt}\\[-4.1pt]\hline
\end{array}
\ \to\
\begin{array}{@{}r@{\hspace{-.3pt}}|@{}l@{}}
\raisebox{-5mm}{\includegraphics[height=27mm]{red-downset}}
&\ \,\\[-6.3pt]\hline
\end{array}
\hspace{1.6ex}
\begin{array}{@{}c@{}}
\text{with image}\\[0ex]
\end{array}
\hspace{.7ex}
\hspace{.1pt}
\begin{array}{c}
\\[-1.5ex]
\begin{array}{@{}r@{\hspace{-.4pt}}|@{}l@{}}
\includegraphics[height=25mm]{semialgebraic}&\ \,\hspace{-.3pt}\\[-4.2pt]\hline
\end{array}
\\[-1.5ex]\mbox{}
\end{array}
\hspace{-6pt}
$$
locates the births along the lower boundary of the blue upset and the
deaths along the upper boundary of the red downset.  The scalar
entries relate the births to the deaths.  In this special case of one
bar, the monomial matrix is $1 \times 1$ with a single nonzero scalar
entry; choosing bases appropriately, this nonzero entry might as well
be~$1$.
\end{example}

\pagebreak

\begin{example}\label{e:flange-switching}
Consider an $\NN^2$-filtration of the following simplicial complex.
$$%
  \includegraphics[width=50mm]{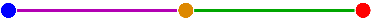}
$$
Each simplex is present above the correspondingly colored rectangular
curve in the following diagram, which theoretically should extend
infinitely far up and to the~right.
$$%
  \includegraphics[width=65mm]{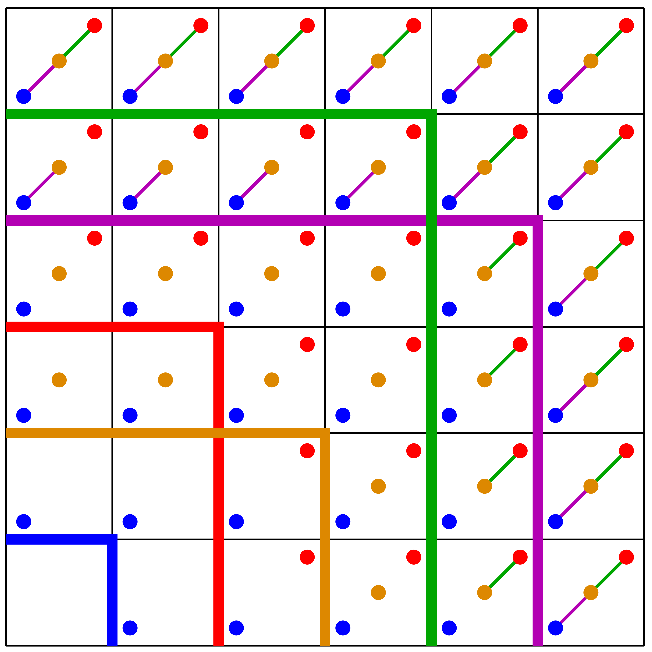}
$$
Each little square depicts the simplicial complex that is present at
the parameter occupying its lower-left corner.  Taking zeroth homology
yields an $\NN^2$-module that replaces the simplicial complex in each
box with the vector space spanned by its connected components.  A
fringe presentation for this $\NN^2$-module is
$$%
\begin{array}{c}
\\
\monomialmatrix
	{\includegraphics[width=10mm]{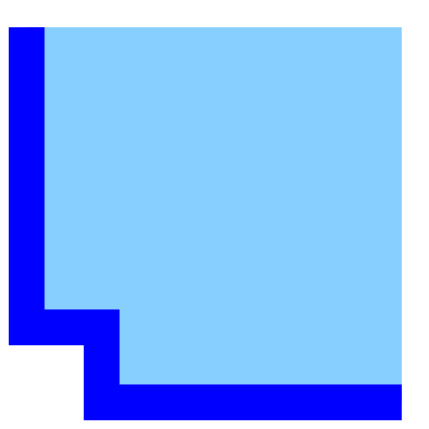}\\
	 \includegraphics[width=10mm]{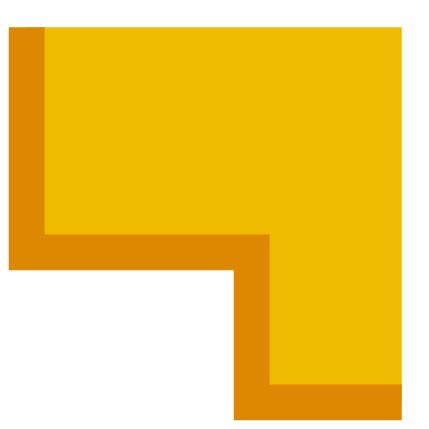}\\
	 \includegraphics[width=10mm]{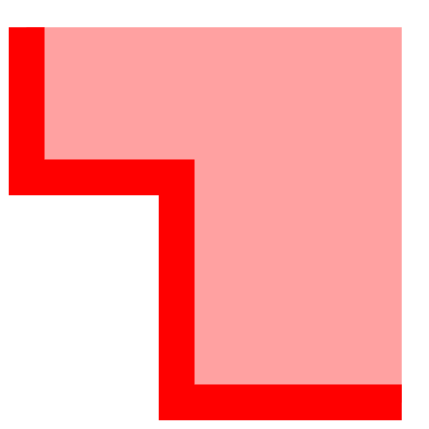}\\}
	{\begin{array}{ccc}
	\\[-8ex]
	\includegraphics[width=10mm]{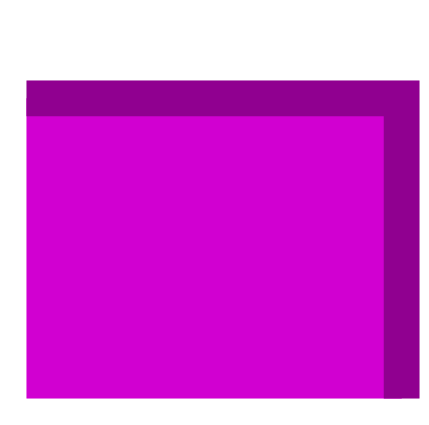}&
	\includegraphics[width=10mm]{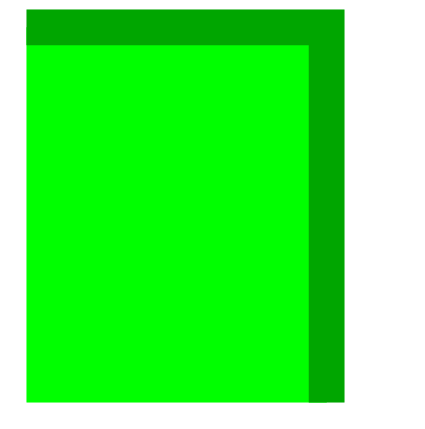}&
	\includegraphics[width=10mm]{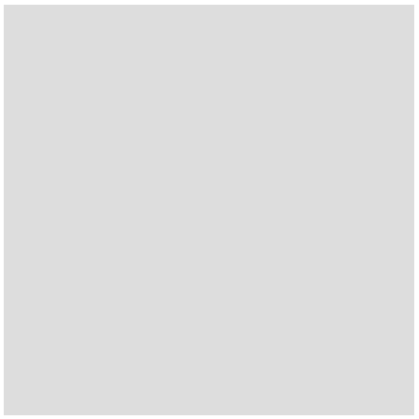}\\
		\\[-2ex] 1 & 0 &1 \\[2.5ex]
		\\[-2ex]-1 & 1 &1 \\[2.5ex]
		\\[-2ex] 0 &-1 &1 \\
	 \end{array}}
	{\\\\\\\\\\\\\\},
\end{array}
$$
where the grey square atop the third column represents the downset
that is all of~$\NN^2$.  This fringe presentation means that, for
example, the connected component that is the blue endpoint of the
simplicial complex is born along the union of the axes with the origin
removed but the point $\bigl[\twoline 11\bigr]$ appended.  The purple
downset, corresponding to the left edge, records the death---along the
upper purple boundary---of the homology class represented by the
difference of the blue (left) and gold (middle) vertices.
Computations and figures for this example were kindly provided by
Ashleigh Thomas.
\end{example}

\begin{remark}\label{r:portmanteau-fr}
The term ``fringe'' is a portmanteau of ``free'' and ``injective''
(that is, ``frinj''), the point being that it combines aspects of free
and injective resolutions while also conveying that the data structure
captures trailing topological features at both the birth and death
ends.
\end{remark}

\section{Encoding poset modules}\label{s:encoding}

Sections~\ref{s:tame} and~\ref{s:fringe} introduce two finiteness
conditions: a topological one (tameness, Definition~\ref{d:tame}),
which is the intuitive control of topological variation in a
filtration, and an algebraic one (fringe presentation,
Definition~\ref{d:fringe}), which provides effective data structures.
To interpolate between them, a third finiteness condition, this one
combinatorial in nature (finite encoding,
Definition~\ref{d:encoding}), serves as a theoretical tool whose
functorial essence supports much of the development in this paper.

\subsection{Finite encoding}\label{sub:encoding}\mbox{}

\noindent
The main result of Section~\ref{s:encoding}, namely
Theorem~\ref{t:tame}, says that tame $\cQ$-modules can be encoded
in the following manner.

\begin{defn}\label{d:encoding}
Fix a poset~$\cQ$.  An \emph{encoding} of a $\cQ$-module $\cM$ by a
poset~$\cP$ is a poset morphism $\pi: \cQ \to \cP$ together with a
$\cP$-module $\cH$ such that $\cM \cong \pi^*\cH =
\bigoplus_{q\in\cQ}H_{\pi(q)}$, the \emph{pullback of~$\cH$
along~$\pi$}, which is naturally a $\cQ$-module.  The encoding is
\emph{finite} if
\begin{enumerate}
\item%
the poset $\cP$ is finite, and
\item%
the vector space $H_p$ has finite dimension for all $p \in \cP$.
\end{enumerate}
\end{defn}

\begin{example}\label{e:wing-subdivision}
Example~\ref{e:toy-model-fly-wing} shows a constant isotypic
subdivision of $\RR^2$ which happens to form a poset and therefore
produces an encoding.
\end{example}

\begin{example}\label{e:flange-switching'}
A finite encoding of the module in Example~\ref{e:flange-switching} is
as follows.
$$%
\psfrag{0-1}{\tiny$\left[\begin{array}{@{}c@{}}
	0\\[-.5ex]1\end{array}\right]$}
\psfrag{10-00-01}{\tiny$\left[\begin{array}{@{}c@{\ }c@{}}
	1&0\\[-.5ex]0&0\\[-.5ex]0&1\end{array}\right]$}
\psfrag{00-10-01}{\tiny$\left[\begin{array}{@{}c@{\ }c@{}}
	0&0\\[-.5ex]1&0\\[-.5ex]0&1\end{array}\right]$}
\psfrag{110-001}{\tiny$\left[\begin{array}{@{}c@{\ }c@{\ }c@{}}
	1&1&0\\[-.5ex]0&0&1\end{array}\right]$}
\psfrag{100-011}{\tiny$\left[\begin{array}{@{}c@{\ }c@{\ }c@{}}
	1&0&0\\[-.5ex]0&1&1\end{array}\right]$}
\psfrag{11}{\tiny$\left[\begin{array}{@{}c@{\ }c@{}}
	1&1\end{array}\right]$}
\psfrag{kk}{\footnotesize$\kk$}
\psfrag{kk2}{\footnotesize$\kk^2$}
\psfrag{kk3}{\footnotesize$\kk^3$}
\begin{array}{@{}c@{}}\includegraphics[height=63mm]{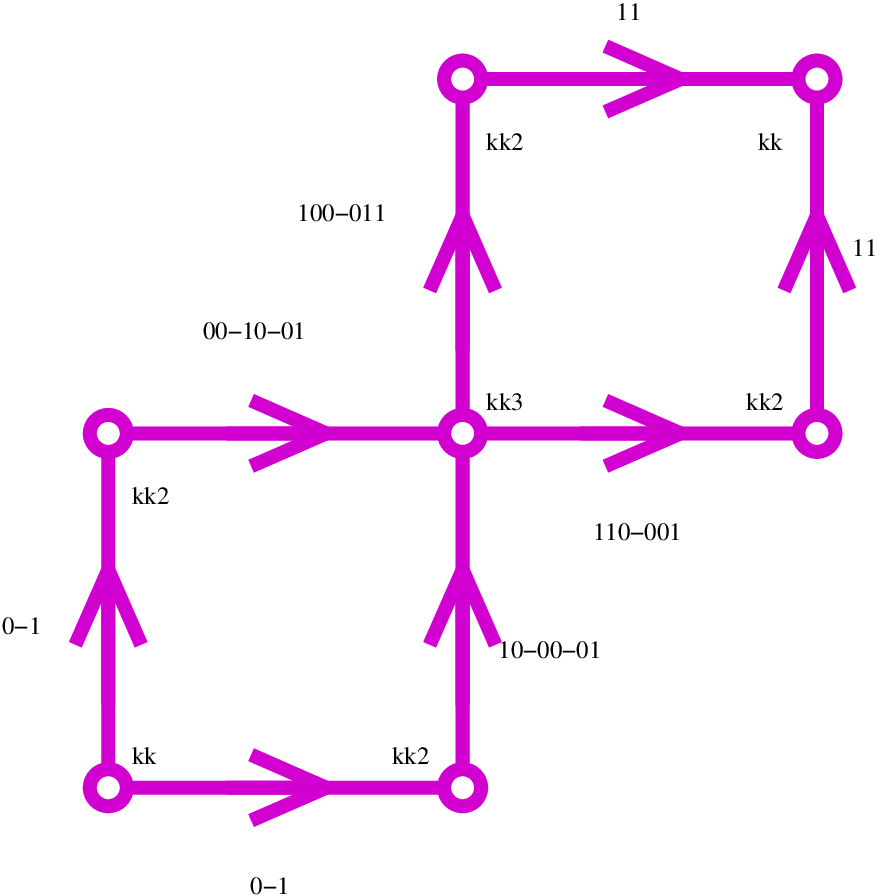}\end{array}
$$
\end{example}

\begin{example}\label{e:subdivide}
There is no natural way to impose a poset structure on the set of
regions in a constant subdivision.  Take, for example, $\cQ = \RR^2$
and $\cM = \kk_\0 \oplus \kk[\RR^2]$, where $\kk_\0$ is the
$\RR^2$-module whose only nonzero component is at the origin, where it
is a vector space of dimension~$1$.  This module~$\cM$ induces only
two isotypic regions, namely the origin and its complement, and they
constitute a constant subdivision.
$$%
\includegraphics[width=10mm]{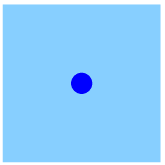}
\begin{array}{c}=\\[4ex]\end{array}
\includegraphics[width=10mm]{isotypic0}
\hspace{-1.25mm}\begin{array}{c}\cup\\[4ex]\end{array}\hspace{1.8mm}
\includegraphics[width=10mm]{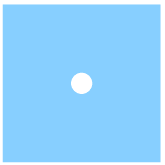}
$$
Neither of the two regions has a stronger claim to precede the other,
but at the same time it would be difficult to justify forcing the
regions to be incomparable.
\end{example}

\begin{example}\label{e:convex-projection}
Take $\cQ = \ZZ^n$ and $\cP = \NN^n$.  The \emph{convex projection}
$\ZZ^n \to \NN^n$ sets to~$0$ every negative coordinate.  The pullback
under convex projection is the \v Cech hull
\cite[Definition~2.7]{alexdual}.  More generally, suppose $\aa \preceq
\bb$ in~$\ZZ^n$.  The interval $[\aa,\bb] \subseteq \ZZ^n$ is a box
(rectangular parallelepiped) with lower corner at~$\aa$ and upper
corner at~$\bb$.  The \emph{convex projection} $\pi: \ZZ^n \to
[\aa,\bb]$ takes every point in~$\ZZ^n$ to its closest point in the
box.  A $\ZZ^n$-module is \emph{finitely determined} if it is finitely
encoded by~$\pi$.
\end{example}

\begin{example}\label{e:indicator}
The indicator module $\kk[\cQ]$ is encoded by the morphism from $\cQ$
to the one-point poset with vector space $\cH = \kk$.  This
generalizes to other indicator modules.
\begin{enumerate}
\item%
Any upset module $\kk[U] \subseteq \kk[\cQ]$ is encoded by a morphism
from~$\cQ$ to the chain~$P$ of length~$1$, consisting of two points $0
< 1$, that sends $U$ to~$1$ and the complement $\ol U$ to~$0$.  The
$P$-module~$H$ that pulls back to~$\kk[U]$ has $H_0 =\nolinebreak
0$~and~$H_1 = \kk$.
\item%
Dually, any downset module $\kk[D]$ is also encoded by a morphism
from~$\cQ$ to the chain $P$ of length~$1$, but this one sends $D$
to~$0$ and the complement $\ol D$ to~$1$, and the $P$-module~$H$ that
pulls back to~$\kk[D]$ has $H_0 = \kk$~and~$H_1 = 0$.
\end{enumerate}
\end{example}

\begin{defn}\label{d:subordinate-encoding}
Fix a poset~$\cQ$ and a $\cQ$-module~$\cM$.
\begin{enumerate}
\item\label{i:morphism}%
A poset morphism $\pi: \cQ \to \cP$ or an encoding of a $\cQ$-module
(perhaps different from~$\cM$) is \emph{subordinate} to~$\cM$ if there
is a $\cP$-module~$\cH$ such that~$\cM \cong\nolinebreak \pi^*\cH$.

\item\label{i:auxiliary-encoding}%
When $\cQ$ is a subposet of a partially ordered real vector space, an
encoding of~$\cM$ is \emph{semialgebraic}, \emph{PL},
\emph{subanalytic}, or \emph{of class~$\mathfrak X$} if the partition
of~$\cQ$ formed by the fibers of~$\pi$ is of the corresponding type
(Definition~\ref{d:auxiliary-hypotheses}).
\end{enumerate}
\end{defn}

\begin{example}\label{e:antidiagonal'}
The ``antidiagonal'' $\RR^2$-module $\cM$ in
Example~\ref{e:antidiagonal} has a semialgebraic poset encoding by the
chain with three elements, where the fiber over the middle element is
the antidiagonal line, and the fibers over the top and bottom elements
are the open half-spaces above and below the line, respectively.
In contrast, using the diagonal line spanned by $\left[\twoline
11\right] \in \RR^2$ instead of the antidiagonal line yields a module
with no finite encoding; see Example~\ref{e:diagonal}.
\end{example}

\begin{lemma}\label{l:constant}
An indicator module is constant on every fiber of a poset morphism
$\pi: \cQ \to \cP$ if and only if the module is the pullback
along~$\pi$ of an indicator $\cP$-module.
\end{lemma}
\begin{proof}
The ``if'' direction is by definition.  For the ``only if'' direction,
observe that if $U \subseteq \cQ$ is an upset that is a union of
fibers of~$\cP$, then the image $\pi(U) \subseteq \cP$ is an upset
whose preimage equals~$U$.  The same argument works for downsets.
\end{proof}

\begin{example}[Pullbacks of flat and injective modules]\label{e:pullback}
An indecomposable flat $\ZZ^n$-module $\kk[\bb + \ZZ\tau + \NN^n]$ is
an upset module for the poset~$\ZZ^n$.  Pulling back to any poset
under a poset map to~$\ZZ^n$ therefore yields an upset module for the
given poset.  The dual statement holds for any indecomposable
injective module $\kk[\bb + \ZZ\tau - \NN^n]$: its pullback is a
downset module.
\end{example}

Pullbacks have particularly transparent monomial matrix
interpretations.

\begin{prop}\label{p:pullback-monomial-matrix}
Fix a poset~$\cQ$ and an encoding of a $\cQ$-module~$\cM$ via a poset
morphism $\pi: \cQ \to \cP$ and $\cP$-module~$\cH$.  Any monomial
matrix for a fringe presentation of~$\cH$ pulls back to a monomial
matrix for a fringe presentation that dominates the encoding by
replacing the row labels $U_1,\dots,U_k$ and column labels
$D_1,\dots,D_\ell$ with their preimages, namely $\pi^{-1}(U_1), \dots,
\pi^{-1}(U_k)$ and $\pi^{-1}(D_1), \dots, \pi^{-1}(D_\ell)$.
\hfill$\square$
\end{prop}

\subsection{Uptight posets}\label{sub:uptight}\mbox{}

\noindent
Constructing encodings from constant subdivisions uses general poset
combinatorics.

\begin{defn}\label{d:uptight-region}
Fix a poset~$\cQ$ and a set $\Upsilon$ of upsets.  For each poset
element $\aa \in \cQ$, let $\Upsilon_\aa \subseteq \Upsilon$ be the
set of upsets from~$\Upsilon$ that contain~$\aa$.  Two poset elements
$\aa,\bb \in \cQ$ lie in the same \emph{uptight region} if
$\Upsilon_\aa = \Upsilon_\bb$.
\end{defn}

\begin{remark}\label{r:common-refinement}
The partition of $\cQ$ into uptight regions in
Definition~\ref{d:uptight-region} is the common refinement of the
partitions $\cQ = U \cupdot (\cQ \minus U)$ for $U \in \Upsilon$.
\end{remark}

\begin{remark}\label{r:iso-uptight}
Every uptight region is the intersection of a single upset (not
necessarily one of the ones in~$\Upsilon$) with a single downset.
Indeed, the intersection of any family of upsets is an upset, the
complement of an upset is a downset, and the intersection of any
family of downsets is a downset.  Hence the uptight region
containing~$\aa$ equals $\bigl(\bigcap_{U \in \Upsilon_\aa} U\bigr)
\cap \bigl(\bigcap_{U \not\in \Upsilon_\aa} \ol U\bigr)$, the first
intersection being an upset and the second a~downset.
\end{remark}

\begin{prop}\label{p:posetQuotient}
In the situation of Definition~\ref{d:uptight-region}, the relation on
uptight regions given by $A \preceq B$ whenever $\aa \preceq \bb$ for
some $\aa \in A$ and $\bb \in B$ is reflexive and acyclic.
\end{prop}
\begin{proof}
The stipulated relation on the set of uptight regions is
\begin{itemize}
\item%
reflexive because $\aa \preceq \aa$ for any element~$\aa$ in any
uptight region~$A$; and

\item%
acyclic because going up from $\aa \in \cQ$ causes the set
$\Upsilon_\aa$ in Definition~\ref{d:uptight-region} to (weakly)
increase, so a directed cycle can only occur with a constant sequence
of sets~$\Upsilon_\aa$.\qedhere
\end{itemize}
\end{proof}

\begin{example}\label{e:puuska-nontransitive}
The relation in Proposition~\ref{p:posetQuotient} makes the set of
uptight regions into a directed acyclic graph, but the relation need
not be transitive.  An example in the poset $\cQ = \NN^2$, kindly
provided by Ville Puuska \cite{puuska18}, is as follows.
Notationally, it is easier to work with monomial ideals in $\kk[x,y] =
\kk[\NN^2]$, which correspond to upsets (see
Example~\ref{e:disconnected-homomorphism}).  Let $\Upsilon =
\{U_1,\dots,U_4\}$ consist of the upsets with indicator modules
$$%
  \kk[U_1] = \<x^2,y\>,\quad
  \kk[U_2] = \<x^3,y\>,\quad
  \kk[U_3] = \<xy\>,\quad
  \kk[U_4] = \<x^2y\>.
$$
Identifying each monomial $x^ay^b$ with the corresponding pair $(a,b)
\in \NN^2$, it follows that $\Upsilon_{\!x^2} = \{U_1\}$, and
$\Upsilon_{\!x^3} = \Upsilon_{\!y} = \{U_1,U_2\}$, and
$\Upsilon_{\!xy} = \{U_1,U_2,U_3\}$ represent three distinct uptight
regions; call them $A$, $B$, and $C$.  They satisfy $A \prec B \prec
C$ because $x^2 \prec x^3$ and $y \prec xy$.  However, $A \not \preceq
C$ because $A = \{x^2\}$ while $U_4$ forces $C = xy\kk[y]$ to consist
of all lattice points in a vertical ray starting at $xy$.
\end{example}

\begin{defn}\label{d:uptight-poset}
In the situation of Definition~\ref{d:uptight-region}, the
\emph{uptight poset} is the transitive closure $\cP_\Upsilon$ of the
directed acyclic graph of uptight regions in
Proposition~\ref{p:posetQuotient}.
\end{defn}

\subsection{Constant upsets}\label{sub:upsets}

\begin{defn}\label{d:uptight-constant}
Fix a constant subdivision of~$\cQ$ subordinate to~$\cM$.  A
\emph{constant upset} of~$\cQ$ is either
\begin{enumerate}
\item%
an upset $U_I$ generated by a constant region~$I$ or
\item%
the complement of a downset $D_I$ cogenerated by a constant
region~$I$.
\end{enumerate}
\end{defn}

\begin{thm}\label{t:constant-uptight}
Let $\Upsilon$ be the set of constant upsets from a~constant
subdivision of~$\cQ$ subordinate to~$\cM$.  The partition of~$\cQ$
into uptight regions for $\Upsilon$ forms another constant subdivision
subordinate to~$\cM$.
\end{thm}
\begin{proof}
Suppose that $A$ is an uptight region that contains points from
constant regions $I$ and~$J$.  Any point in $I \cap A$ witnesses the
containments $A \subseteq D_I$ and $A \subseteq U_I$ of~$A$ inside the
constant upset and downset generated and cogenerated by~$I$.  Any
point $\jj \in J \cap A$ is therefore sandwiched between elements
$\ii, \ii' \in I$, so $\ii \preceq \jj \preceq \ii'$, because $\jj \in
U_I$ (for~$\ii$) and $\jj \in D_I$ (for~$\ii'$).  By symmetry,
switching $I$ and~$J$, there exists $\jj' \in J$ with $\ii' \preceq
\jj'$.  The sequence
$$%
  M_I \to \cM_\ii \to \cM_\jj \to \cM_{\ii'} \to \cM_{\jj'} \to M_J,
$$
where the first and last isomorphisms come from
Definition~\ref{d:constant-subdivision} and the homomorphisms in
between are $\cQ$-module structure homomorphisms, induces isomorphisms
$\cM_\ii \to \cM_{\ii'}$ and $\cM_\jj \to \cM_{\jj'}$ by definition of
constant region.  Elementary homological algebra implies that $\cM_\ii
\to \cM_\jj$ is an isomorphism.  The induced isomorphism $M_I \to M_J$
is independent of the choices of $\ii$, $\jj$, $\ii'$, and~$\jj'$ (in
fact, merely considering independence of the choices of~$\ii$
and~$\jj'$ would suffice) because constant subdivisions
have~no~monodromy.

The previous paragraph need not imply that $I = J$, but it does imply
that all of the vector spaces $M_J$ for constant regions~$J$ that
intersect~$A$ are---viewing the data of the original constant
subdivision as given---canonically isomorphic to~$M_I$, thereby
allowing the choice of $M_A = M_I$.  This, plus the lack of monodromy
in constant subdivisions, ensures that $M_A \to M_\aa \to M_\bb \to
M_B$ is independent of the choices of $\aa \in A$ and $\bb \in B$ with
$\aa \preceq \bb$.  Thus the uptight subdivision is also constant
subordinate~to~$\cM$.
\end{proof}

\begin{example}\label{e:iso-uptight}
Theorem~\ref{t:constant-uptight} does not claim that $I = U_I \cap
D_I$, and in fact that claim is often not true, even if the isotypic
subdivision (Example~\ref{e:puuska-nonconstant-isotypic}) is already
constant.  Consider $\cQ = \RR^2$ and $\cM = \kk_\0 \oplus
\kk[\RR^2]$, as in Example~\ref{e:subdivide}, and take $I = \RR^2
\minus \{\0\}$.  Then $U_I = D_I = \RR^2$, so $U_I \cap D_I$ contains
the other isotypic region $J = \{\0\}$.  The uptight poset $\cP_\cM$
has precisely four elements:
\begin{enumerate}
\item%
the origin $\{\0\} = U_J \cap D_J$;
\item%
the complement $U_J \minus \{\0\}$ of the origin in $U_J$;
\item%
the complement $D_J \minus \{\0\}$ of the origin in $D_J$; and
\item%
the points $\RR^2 \minus (U_J \cup D_J)$ lying only in~$I$ and in
neither $U_J$ nor~$D_J$.
\end{enumerate}
Oddly, uptight region~4 has two connected components, the second and
fourth quadrants $A$ and~$B$, that are incomparable: any chain of
relations from Definition~\ref{d:constant-subdivision} that realizes
the equivalence $\aa \sim \bb$ for $\aa \in A$ and~$\bb \in B$ must
pass through the positive quadrant or the negative quadrant, each of
which accidentally becomes comparable to the other isotypic region~$J$
and hence lies in a different uptight region.
\end{example}

\subsection{Finite encoding from constant subdivisions}\label{sub:existence}

\begin{defn}\label{d:compactly-supported}
If $\cQ$ is a subposet of a partially ordered real vector space, then
a $\cQ$-module~$\cM$ has \emph{compact support} if~$\cM$ has nonzero
components~$\cM_q$ only in a bounded set of degrees $q \in \cQ$.  A
constant subdivision subordinate to such a module is \emph{compact} if
it has exactly one unbounded constant region (namely those $q \in \cQ$
for which $\cM_q = 0$).
\end{defn}

\begin{thm}\label{t:tame}
Fix a $\cQ$-finite module~$\cM$ over a poset~$\cQ$.
\begin{enumerate}
\item\label{i:admits-finite-encoding}%
$\cM$ admits a finite encoding if and only if there exists a finite
constant subdivision of~$\cQ$ subordinate to~$\cM$.  More precisely,

\item\label{i:uptight-encoding}%
the uptight poset of the set of constant upsets from any constant
subdivision yields an \emph{uptight encoding} of~$\cM$ that is finite
if the constant subdivision is finite.

\item\label{i:auxiliary-uptight}%
If $\cQ$ is a subposet of a partially ordered real vector space and
the constant subdivision in the previous item is
\begin{itemize}
\item%
semialgebraic, with $\cQ_{+\!}$ also semialgebraic; or

\item%
PL, with $\cQ_{+\!}$ also polyhedral; or

\item%
compact and subanalytic, with $\cQ_{+\!}$ also subanalytic; or

\item%
of class~$\mathfrak X$,
\end{itemize}
then the relevant uptight encoding is semialgebraic, PL, subanalytic,
or class~$\mathfrak X$.
\end{enumerate}
\end{thm}
\begin{proof}
One direction of item~\ref{i:admits-finite-encoding} is easy: a finite
encoding induces a constant subdivision almost by definition.  Indeed,
if $\pi: \cQ \to \cP$ is a poset encoding of~$\cM$ by a
$\cP$-module~$\cH$, then each fiber~$I$ of~$\pi$ is a constant region
with $M_I = H_{\pi(I)}$.  If $\ii \preceq \jj$ with $\ii \in I$ and
$\jj \in J$, then the composite homomorphism $M_I \to M_\ii \to M_\jj
\to M_J$ is merely the structure morphism $H_{\pi(I)} \to H_{\pi(J)}$
of the $\cP$-module~$\cH$.

The hard direction is producing a finite encoding from a constant
subdivision.  For that, it suffices to prove
item~\ref{i:uptight-encoding}.  Let $\Upsilon$ be the set of constant
upsets from a~constant subdivision of~$\cQ$ subordinate to~$\cM$.
Consider the quotient map $\cQ \to \cP_\Upsilon$ of sets that sends
each element of~$\cQ$ to the uptight region containing it.
Proposition~\ref{p:posetQuotient} and Definition~\ref{d:uptight-poset}
imply that this map of sets is a morphism of posets.  By
Definition~\ref{d:constant-subdivision} the vector spaces $M_A$
indexed by the uptight regions $A \in \cP_\Upsilon$ constitute a
$\cP_\Upsilon$-module~$H$ that is well defined by
Theorem~\ref{t:constant-uptight}.  The pullback of~$H$ to~$\cQ$ is
isomorphic to~$\cM$ by construction.  The claim about finiteness
follows because the number of uptight regions is bounded above by
$2^{2r}$, where $r$ is the number of constant regions in the original
constant subdivision: every element of~$\cQ$ lies inside or outside of
each constant upset and inside or outside of each constant downset.

For claim~\ref{i:auxiliary-uptight}, every constant upset is a
Minkowski sum $I + \cQ_+$ or the complement of $I - \cQ_+ = - (-I +
\cQ_+)$ by Definition~\ref{d:uptight-constant}.  These are
semialgebraic, PL, subanalytic, or of class~$\mathfrak X$,
respectively, by Proposition~\ref{p:auxiliary-hypotheses} (or
Definition~\ref{d:auxiliary-hypotheses} for class~$\mathfrak X$).
Note that in the compact subanalytic case, the unique unbounded
constant region~$I$ afforded by Definition~\ref{d:compactly-supported}
has $I + \cQ_+ = I - \cQ_+ = \cQ$, which is subanalytic.
\end{proof}

\begin{example}\label{e:antidiagonal''}
For the ``antidiagonal'' $\RR^2$-module $\cM$ in
Examples~\ref{e:antidiagonal} and~\ref{e:antidiagonal'}, every point
on the line is a singleton isotypic region, but these uncountably many
isotypic regions can be gathered together: the finite encoding there
is the uptight poset for the two upsets that are the closed and open
half-spaces bounded below by the antidiagonal.
\end{example}

\begin{example}\label{e:subdivide'}
In any encoding of the ``diagonal strip'' $\RR^2$-module $\cM$ in
Example~\ref{e:subdivide}, the poset must be uncountable by
Theorem~\ref{t:tame}.
\end{example}

\subsection{The category of tame modules}\label{sub:cat}

\begin{example}\label{e:ker}
The kernel of a homomorphism of tame modules need not be tame.  The
upset $U \subseteq \RR^2$ that is the closed half-space above the
antidiagonal line~$L$ given by $a + b = 1$ has interior $U^\circ$,
also an upset.  The quotient module $N = \kk[U]/\kk[U^\circ]$ is the
translate by one unit (up or to the right) of the antidiagonal module
in Examples~\ref{e:antidiagonal}, \ref{e:antidiagonal'},
and~\ref{e:antidiagonal''}.  Both $M = \kk[U] \oplus \kk[U]$ and $N$
are tame.  The surjection $\phi: M \onto N$ that acts in every degree
$\aa = \left[\twoline ab \right]$ along~$L$ by sending the basis
vectors of $M_\aa = \kk^2$ to $b$ and $-a$ in $N_\aa = \kk$ has
kernel~$K = \ker\phi$ that is the submodule of~$M$ with
\begin{itemize}
\item%
$\kk^2$ in every degree from~$U^\circ$, and
\item%
the line in~$\kk^2$ through $\left[\twoline 00 \right]$ and
$\left[\twoline ab \right]$ in every degree from~$L$.
\end{itemize}
$$%
\psfrag{x}{\tiny$x$}
\psfrag{y}{\tiny$y$}
0
\ \to\
\begin{array}{@{}c@{}}\includegraphics[height=30mm]{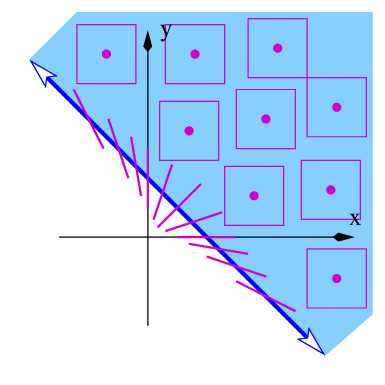}\end{array}
\ \too\
\begin{array}{@{}c@{}}\includegraphics[height=30mm]{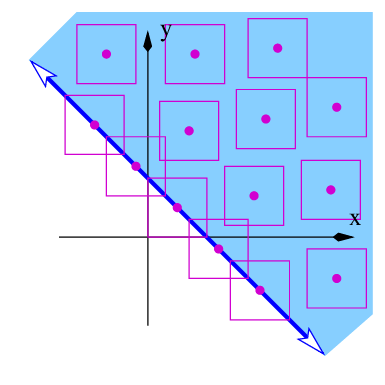}\end{array}
\ \too\
\begin{array}{@{}c@{}}\includegraphics[height=30mm]{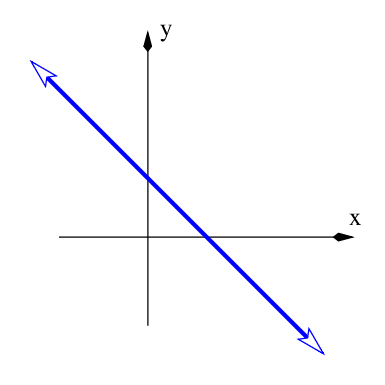}\end{array}
\ \to\
0
$$
That is, $K_\aa$ agrees with $M_\aa$ for degrees~$\aa$ outside of~$L$,
and $K_\aa$ is the line in~$M_\aa$ of slope $b/a$ through the origin
when $\aa$ lies on~$L$.  This kernel~$K$ is not tame.  Indeed, if
$\aa$ and~$\aa'$ are distinct points on~$L$, then the homomorphisms
$K_\aa \to K_{\aa\vee\aa'}$ and $K_{\aa'} \to K_{\aa\vee\aa'}$ have
different images, so $\aa$ and~$\aa'$ are forced to lie in different
constant regions in every constant subdivision of~$\RR^2$ subordinate
to~$K$.  (Note the relation between this example and
Proposition~\ref{p:U->D}.\ref{i:U->D} for $Q = U \subset \RR^2$ and $D
= L \subset Q$.)
\end{example}

\begin{remark}\label{r:lurie}
Encoding of a $\cQ$-module~$\cM$ by a poset morphism is related to
viewing~$\cM$ as a sheaf on~$\cQ$ with its Alexandrov topology that is
constructible in the sense of Lurie \cite[Definitions~A.5.1
and~A.5.2]{lurie2017}.
The difference is that poset encoding requires constancy (in the sense
of Definition~\ref{d:constant-subdivision}) on fibers of the encoding
morphism, whereas Alexandrov constructibility requires only local
constancy in the sense of sheaf theory.
This distinction is decisive for Example~\ref{e:ker}, where the
kernel~$K$ is constructible but not finitely encoded.
\end{remark}

Because of Remark~\ref{r:lurie}, allowing arbitrary homomorphisms
between tame modules would step outside of the tame class.  More
formally, inside the category of $\cQ$-modules, the full subcategory
generated by the tame modules contains modules that are not tame.  To
preserve tameness, it is thus necessary to restrict the allowable
morphisms.

\begin{defn}\label{d:tame-morphism}
A homomorphism $\phi: M \to N$ of $\cQ$-modules is \emph{tame} if
$\cQ$ admits a finite constant subdivision subordinate to both $M$
and~$N$ such that for each constant region~$I$ the composite
isomorphism $M_I \to M_\ii \to N_\ii \to N_I$ does not depend~on~$\ii
\in I$.  The map $\phi$ is semialgebraic, PL, subanalytic, or
class~$\mathfrak X$ if this constant subdivision~is.
\end{defn}

\begin{lemma}\label{l:ker-coker}
The kernel and cokernel of any tame homomorphism of $\cQ$-modules are
tame morphisms of tame modules.  The same is true when tameness is
replaced by semialgebraic, PL, subanalytic, or class~$\mathfrak X$.
\end{lemma}
\begin{proof}
Any constant subdivision as in Definition~\ref{d:tame-morphism} is
subordinate to both the kernel and cokernel of $M \to N$, with the
vector spaces assocated to any constant region~$I$ being $\ker(M_I \to
N_I)$ and $\coker(M_I \to N_I)$.
\end{proof}

\begin{defn}\label{d:abelian-category}
The \emph{category of tame modules} is the subcategory of
$\cQ$-modules whose objects are the tame modules and whose morphisms
are the tame homomorphisms.
\end{defn}

\begin{remark}\label{r:morphism}
To be precise with language, a \emph{morphism} of tame modules is
required to be tame, whereas a \emph{homomorphism} of tame modules is
not.  That is, morphisms in the category of tame modules are called
morphisms, whereas morphisms in the category of $\cQ$-modules are
called homomorphisms.  To avoid confusion, the set of tame morphisms
from a tame module~$M$ to another tame module~$N$ is denoted
$\Mor(M,N)$ instead of $\Hom(M,N)$.
\end{remark}

\begin{prop}\label{p:abelian-category}
Over any poset~$\cQ$, the category of tame $\cQ$-modules is~abelian.
If~$\cQ$ is a subposet of a partially ordered real vector space of
finite dimension, then the category of semialgebraic, PL, subanalytic,
or class~$\mathfrak X$ modules is~abelian.
\end{prop}
\begin{proof}
Over any poset, the category in question is a subcategory of the
category of $\cQ$-modules, which is abelian.  The subcategory is not
full, but $\Mor(M,N)$ is an abelian subgroup of $\Hom(M,N)$; this is
most easily seen via Theorem~\ref{t:tame}, for if $\phi: M \to N$
and $\phi': M \to N'$ are finitely encoded by $\pi: \cQ \to P$ and
$\pi': \cQ \to P'$, respectively, then $\phi + \phi'$ is finitely
encoded by $\pi \times \pi': \cQ \to P \times P'$.  The same
construction, but with the source of~$\pi'$ being a new module~$M'$
instead of~$M$, shows that the ordinary product and direct sum of a
pair of finitely encoded modules serves as a product and coproduct in
the tame category.  Kernels and cokernels of morphisms in the tame
category exist by Lemma~\ref{l:ker-coker}, which also implies that
every monomorphism is a kernel (it is the kernel of its cokernel in
the category of $\cQ$-modules) and every epimorphism is a cokernel (it
is the cokernel of its kernel in the category of $\cQ$-modules).

The semialgebraic, PL, and class~$\mathfrak X$ cases have the same
proof, noting that $\pi \times \pi'$ has fibers of the desired type if
$\pi$ and~$\pi'$ both do.  The subanalytic case only follows from this
argument when restricted to the category of modules whose nonzero
graded pieces lie in a bounded subset of~$\cQ$ (the subset is allowed
to depend on the module).  However, the argument in the previous
paragraph can be done directly with common refinements of pairs of
constant subdivisions, so reducing to Theorem~\ref{t:tame} is not
necessary.
\end{proof}

\section{Primary decomposition over partially ordered groups}\label{s:decomp}

In the context of persistent homology, multiparameter features can die
in many ways, persisting indefinitely as some of the parameters
increase without limit but dying when any of the others increase
sufficiently.  The essence of this phenomenon is captured in
elementary language by Theorem~\ref{t:elementary-coprimary}.  In
the ordinary situation of one parameter, the only distinction being
made here is that a feature can be mortal or immortal.  Beyond the
intrinsic mathematical value, decomposing a module according to these
distinctions has concrete benefits for statistical analysis using
multipersistence \cite{primary-distance}.

\subsection{Polyhedral partially ordered groups}\label{sub:polyhedral}\mbox{}

\noindent
The next definition, along with elementary foundations surrounding it,
can be found in Goodearl's book \cite[Chapter~1]{goodearl86}.

\begin{defn}\label{d:pogroup}
An abelian group~$\cQ$ is \emph{partially ordered} if it is generated
by a submonoid~$\cQ_+$, called the \emph{positive cone}, that has
trivial unit group.  The partial order is: $q \preceq q' \iff q' - q
\in \cQ_+$.  All partially ordered groups in this paper are assumed
abelian.
\end{defn}

\begin{example}\label{e:discrete-pogroup}
The finitely generated free abelian group $\cQ = \ZZ^n$ can be
partially ordered with any positive cone~$\cQ_+$, polyhedral or
otherwise, resulting in a \emph{discrete partially ordered group}.
The free commutative monoid $\cQ_+ = \NN^n$ of integer vectors with
nonnegative coordinates is the most common instance and serves as a
well behaved, well known foundational case (see Section~\ref{s:ZZn})
to which substantial parts of the general theory reduce.  For
notational clarity, $\ZZ_+^n$ always means the nonnegative orthant
in~$\ZZ^n$, which induces the standard componentwise partial order
on~$\ZZ^n$.  Other partial orders can be specified using notation $\cQ
\cong \ZZ^n$ with arbitrary positive cone~$\cQ_+$.
\end{example}

\begin{example}\label{e:real-pogroup}
The group $\cQ = \RR^n$ can be partially ordered with any positive
cone~$\cQ_+$, polyhedral or otherwise, closed, open (away from the
origin~$\0$) or anywhere in between, resulting in a \emph{real
partially ordered group}.  The orthant $\cQ_+ = \RR_+^n$ of vectors
with nonnegative coordinates is most useful for multiparameter
persistence.  For notational clarity, $\RR_+^n$ always means the
nonnegative orthant in~$\RR^n$, which induces the standard
componentwise partial order on~$\RR^n$.  Other partial orders can be
specified using notation $\cQ \cong \RR^n$ with arbitrary positive
cone~$\cQ_+$.
\end{example}

\begin{example}\label{e:torsion}
Definition~\ref{d:pogroup} allows the group to have torsion.  Thus the
submonoid of $\ZZ \oplus \ZZ/2\ZZ$ generated by $\bigl[\twoline
10\bigr]$ and $\bigl[\twoline 11\bigr]$
is a positive cone in the group.  There is a continuous version in
which the resulting partial order is easier to
\begin{figure}[h]
\vspace{-1ex}
$$%
\includegraphics[height=20mm]{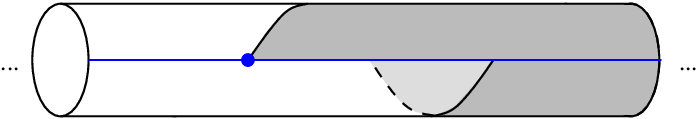}
$$
\vspace{-4ex}
\end{figure}
see geometrically: $\cQ = \RR \times \RR/\ZZ$ with $\cQ_+$ generated
by $\bigl[\twoline 10\bigr]$ and $\bigl[\twoline 11\bigr]$.  In the
figure, the blue center line is the first factor~$\RR$, with
origin~$\0$ at the fat blue dot.  The positive cone~$\cQ_+$ is shaded.
\end{example}

The following allows the free use of the language of either
$\cQ$-modules or $\cQ$-graded $\kk[\cQ_+]$-modules, as appropriate to
the context.

\begin{lemma}\label{l:Q-graded}
A module over a partially ordered abelian group~$\cQ$ is the same
thing as a $\cQ$-graded\/ module over~$\kk[\cQ_+]$.\hfill\qed
\end{lemma}

\begin{example}\label{e:ZZn-graded}
When $\cQ = \ZZ^n$ and $\cQ_+ = \NN^n$, the relevant monoid
algebra is the polynomial ring $\kk[\NN^n] = \kk[\xx]$, where $\xx =
x_1,\dots,x_n$ is a sequence of $n$ commuting~variables.
\end{example}

Primary decomposition of $\cQ$-modules depends on certain finiteness
conditions.  In ordinary commutative algebra, where $\cQ = \ZZ^n$, the
finiteness comes from~$\cQ_+$, which is assumed to be finitely
generated (so it is an \emph{affine semigroup}).  This condition
implies that finitely generated $\cQ$-modules are noetherian: every
increasing chain of submodules stabilizes.  Primary decomposition is
then usually derived as a special case of the theory for finitely
generated modules over noetherian rings.  But the noetherian condition
is stronger than necessary: in the presence of a tame hypothesis, it
suffices for the positive cone to have finitely many faces, in the
following sense.  To the author's knowledge, the notion of polyhedral
partially ordered group is new, and there is no existing literature on
primary decomposition in this setting.

\begin{defn}\label{d:face}
A \emph{face} of the positive cone~$\cQ_+$ of a partially ordered
group~$\cQ$ is a submonoid $\sigma \subseteq \cQ_+$ such that
$\cQ_{+\!} \minus \sigma$ is an ideal of the monoid~$\cQ_+$.
Sometimes it is simpler to say that~$\sigma$ is a \emph{face}
of~$\cQ$.  Call~$\cQ$ \emph{polyhedral} if it has only finitely
many~faces.
\end{defn}

Polyhedrality suffices to prove existence of (finite) primary
decomposition (Theorem~\ref{t:primDecomp}).  However, many of the
ingredients, such as localization along or taking support on a face
(Definition~\ref{d:PF}), make sense also under a different sort
of~\mbox{hypothesis}.

\begin{defn}\label{d:closed}
Let $\cQ$ be a partially ordered group.
\begin{enumerate}
\item%
A \emph{ray} of the positive cone~$\cQ_+$ is a face that is totally
ordered as a partially ordered submonoid of~$\cQ$.

\item%
The partially ordered group $\cQ$ is \emph{closed} if the complement
$\cQ_{+\!} \minus \tau$ of each face~$\tau$ is generated as an upset
(i.e., as an ideal) of~$\cQ_{+\!}$ by $\rho \minus \{\0\}$ for the
rays $\rho \not\subseteq \tau$.
\end{enumerate}
\end{defn}

\begin{example}\label{e:closed}
Any real partially ordered group (Example~\ref{e:real-pogroup})~$\cQ$
whose positive cone~$\cQ_+$ is closed in the usual topology on~$\cQ$
is a closed partially ordered group by the Krein--Milman theorem:
$\cQ_+$ is the set of nonnegative real linear combinations of vectors
on extreme rays of~$\cQ_+$.  For instance, a non-polyhedral closed
partial order on $\cQ = \RR^3$ results by taking $\cQ_+$ to be a cone
over a~disk, such as either half of the cone $x^2 + y^2 \leq z^2$.  In
contrast, if~$\cQ_+$ is an intersection of finitely many closed
half-spaces, then there are only finitely many extreme rays.  (This
case is crucial in applications---see \cite{kashiwara-schapira2019,
essential-real}, for instance.)  Even in the polyhedral case the cone
need not be rational; that is, the vectors that generate it---or the
linear functions defining the closed half-spaces whose intersection
is~$\cQ_+$---need not have rational entries.
\end{example}

\begin{example}\label{e:discrete-polyhedral}
Any discrete partially ordered group
(Example~\ref{e:discrete-pogroup}) whose positive cone is a finitely
generated submonoid is automatically both polyhedral and closed; see
\cite[Lemma~7.12]{cca}.  A discrete partially ordered group can also
have a positive cone that is not a finitely generated submonoid, such
as $\cQ = \ZZ^2$ with $\cQ_+ = C \cap \ZZ^2$ for the cone $C \subseteq
\RR^2$ generated by $\bigl[\twoline 10\bigr]$ and $\bigl[\twoline
1\pi\bigr]$.  This particular irrational cone yields a partially
ordered group that is polyhedral but not closed.  Indeed, there are
fewer than the expected faces, because only some of the faces of~$C$
result in faces of~$\cQ_+$ itself.  The image of~$\cQ$ is not discrete
in the quotient of $\cQ \otimes \RR$ modulo the subgroup spanned by
the irrational real face, which can have unexpected consequences for
the algebra of poset modules under localization along such a face.
\end{example}

\begin{example}\label{e:torsion'}
The cylindrical group~$\cQ$ in Example~\ref{e:torsion} has two faces:
the origin~$\0$ (the fat blue dot) and~$\RR_{+\!}$ (the rightmost half
of the horizontal blue center line).
\end{example}

\subsection{Primary decomposition of downsets}\label{sub:downsets}

\begin{defn}\label{d:PF}
Fix a face~$\tau$ of the positive cone~$\cQ_+$ in a polyhedral or
closed partially ordered group~$\cQ$ and a downset $D \subseteq \cQ$.
Write $\ZZ \tau$ for the subgroup of~$\cQ$ generated~by~$\tau$.
\begin{enumerate}
\item\label{i:localization}%
The \emph{localization} of~$D$ \emph{along~$\tau$} is the subset
$$%
  D_\tau = \{q \in D \mid q + \tau \subseteq D\}.
$$

\item\label{i:globally-supported}%
An element $q \in D$ is \emph{globally supported on~$\tau$} if $q
\not\in D_{\tau'}$ whenever $\tau' \not\subseteq \tau$.

\item%
The part of~$D$ \emph{globally supported on~$\tau$} is
$$%
  \Gamma_{\!\tau} D = \{q \in D \mid q \text{ is globally supported on }\tau\}.
$$

\item%
An element $q \in D$ is \emph{locally supported on~$\tau$} if $q$ is
globally supported on~$\tau$ in~$D_\tau$.

\item%
The \emph{local $\tau$-support} of~$D$ is the subset
$\Gamma_{\!\tau}(D_\tau) \subseteq D$ consisting of elements globally
supported on~$\tau$ in the localization~$D_\tau$.

\item\label{i:primary-component}%
The \emph{$\tau$-primary component} of~$D$ is the downset
$$%
  P_\tau(D) = \Gamma_{\!\tau}(D_\tau) - \cQ_+
$$
cogenerated by the local $\tau$-support of~$D$.
\end{enumerate}
\end{defn}

\begin{example}\label{e:PF}
The local $\tau$-supports of the under-hyperbola downset in
Example~\ref{e:hyperbola-GD} are the subsets depicted on the
right-hand side there, for the faces $\tau = \0$, $x$-axis, and
$y$-axis, respectively.  The corresponding primary components are
depicted in Example~\ref{e:hyperbola-PD}.  In contrast, the global
support on (say) the $y$-axis consists of the part of the local
support that sits strictly above the $x$-axis, and the global support
at~$\0$ is the part of~$D$ strictly in the positive quadrant.

This example demonstrates that the $\tau$-primary component of~$D$ in
Definition~\ref{d:PF} need not be supported on~$\tau$.  Indeed, $D =
P_\0(D)$ here, and points outside of~$\cQ_+$ are not supported at the
origin, being instead locally supported at either the $x$-axis (if the
point is below the $x$-axis) or the $y$-axis (if the point is behind
the $y$-axis).
\end{example}

\begin{remark}\label{r:PF}
Definition~\ref{d:PF} makes formal sense in any partially ordered
group, but extreme
caution is recommended without the closed or polyhedral assumptions.
Indeed, without such assumptions, faces can be virutally present, such
as a missing face in a real polyhedron that is not closed or the
irrational face in Example~\ref{e:discrete-polyhedral}.  In such
cases, aspects of Definition~\ref{d:PF} might produce unintended
output.  That said, the natural generality for the concepts in
Definition~\ref{d:PF}~is~unclear.
\end{remark}

\begin{example}\label{e:coprincipal}
The \emph{coprincipal} downset $\aa + \tau - \cQ_+$ inside of $\cQ =
\ZZ^n$ \emph{cogenerated} by~$\aa$ \emph{along~$\tau$} is globally
supported along~$\tau$.  It also equals its own localization
along~$\tau$, so it equals its local $\tau$-support and is its own
$\tau$-primary component.  Note that when $\cQ_+ = \NN^n$, faces
of~$\cQ_+$ correspond to subsets of~$[n] = \{1,\dots,n\}$, the
correspondence being $\tau \leftrightarrow \chi(\tau)$, where
$\chi(\tau) = \{i \in [n] \mid \ee_i \in \tau\}$ is the
\emph{characteristic subset} of~$\tau$ in~$[n]$.  (The vector~$\ee_i$
is the standard basis vector whose only nonzero entry is $1$ in
slot~$i$.)
\end{example}

\begin{remark}\label{r:freely}
The localization of~$D$ along~$\tau$ is acted on freely by~$\tau$.
Indeed, $D_\tau$ is the union of those cosets of~$\ZZ \tau$ each of
which is already contained in~$D$.  The minor point being made here is
that the coset $q + \ZZ \tau$ is entirely contained in~$D$ as soon as
$q + \tau \subseteq D$ because $D$ is a downset: $q + \ZZ \tau = q +
\tau - \tau \subseteq q + \tau - \cQ_+ \subseteq D$ if~\mbox{$q + \tau
\subseteq D$}.
\end{remark}

\begin{remark}\label{r:monomial-localization}
The localization of~$D$ is defined to reflect localization at the
level of $\cQ$-modules: enforcing invertibility of structure
homomorphisms $\kk[D]_q \to \kk[D]_{q+f}$ for $f \in \tau$ results in
a localized indicator module $\kk[D][\ZZ \tau] = \kk[D_\tau]$; see
Definition~\ref{d:support}.
\end{remark}

\begin{example}\label{e:support}
Fix a downset $D$ in a partially ordered group~$\cQ$ that is closed
(Definition~\ref{d:closed} and subsequent examples).  An element $q
\in D$ is globally supported on~$\tau$ if and only if it lands outside
of~$D$ when pushed far enough up in any direction outside
of~$\tau$---that is, every $f \in \cQ_{+\!} \minus \tau$ has a
nonnegative integer multiple~$\lambda f$~with~$\lambda f +\nolinebreak
q \not\in D$.

One implication is easy: if every $f \in \cQ_{+\!} \minus \tau$ has
$\lambda f + q \not\in D$ for some $\lambda \in \NN$, then any element
$f' \in \tau' \minus \tau$ has a multiple $\lambda f' \in \tau'$ such
that $\lambda f' + q \not\in D$, so $q \not\in D_{\tau'}$.  For the
other direction, use Definition~\ref{d:closed}: $q \in \Gamma_{\!\tau}
D \implies q \not\in D_\rho$ for all rays~$\rho$ of~$\cQ_{+\!}$ that
are not contained in~$\tau$, so along each such ray~$\rho$ there is a
vector~$v_\rho$ with $v_\rho + q \not\in D$.  Given $f \in \cQ_{+\!}
\minus \tau$, choose $\lambda \in \NN$ big enough so that $\lambda f
\succeq v_\rho$ for some~$\rho$.
\end{example}

\begin{defn}\label{d:primDecomp}
Fix a downset~$D$ in a polyhedral partially ordered group~$\cQ$.
\begin{enumerate}
\item%
The downset~$D$ is \emph{coprimary} if $D = P_\tau(D)$ for some
face~$\tau$ of the positive cone~$\cQ_+$.  If $\tau$ needs to
specified then $D$ is called \emph{$\tau$-coprimary}.

\item%
A \emph{primary decomposition} of~$D$ is an expression $D =
\bigcup_{i=1}^r D_i$ of coprimary downsets~$D_i$, called
\emph{components} of the decomposition.
\end{enumerate}
\end{defn}

\begin{thm}\label{t:PF}
Every downset $D$ in a polyhedral partially ordered group~$\cQ$ is the
union $\bigcup_\tau \Gamma_{\!\tau}(D_\tau)$ of its local
$\tau$-supports for all faces $\tau$ of the positive cone.
\end{thm}
\begin{proof}
Given an element $q \in D$, finiteness of the number of faces implies
the existence of a face~$\tau$ that is maximal among those such that
$q \in D_\tau$; note that $q \in D = D_\0$ for the trivial face $\0$
consisting of only the identity of~$\cQ$.  It follows immediately that
$q$ is supported on~$\tau$ in~$D_\tau$.
\end{proof}

\begin{cor}\label{c:PF}
Every downset $D$ in a polyhedral partially ordered group~$\cQ$ has a
canonical primary decomposition $D = \bigcup_\tau P_\tau(D)$, the
union being over all faces~$\tau$ of the positive cone with nonempty
support $\Gamma_{\!\tau}(D_\tau)$.
\end{cor}

\begin{remark}\label{r:disjoint}
The union in Theorem~\ref{t:PF} is not necessarily disjoint.  Nor,
consequently, is the union in Corollary~\ref{c:PF}.  There is a
related union, however, that is disjoint: the sets $(\Gamma_{\!\tau}
D) \cap D_\tau$ do not overlap.  Their union need not be all of~$D$,
however; try Example~\ref{e:PF}, where the negative quadrant
intersects none of the sets~$(\Gamma_{\!\tau} D) \cap D_\tau$.

Algebraically, $(\Gamma_{\!\tau} D) \cap D_\tau$ should be interpreted
as taking the elements of~$D$ globally supported on~$\tau$ and then
taking their images in the localization along~$\tau$, which deletes
the elements that aren't locally supported on~$\tau$.  That is,
$(\Gamma_{\!\tau} D) \cap D_\tau$ is the set of degrees where the
image of $\Gamma_{\!\tau}\kk[D] \to \kk[D]_\tau$ is nonzero.
\end{remark}

\begin{example}\label{e:PF'}
The decomposition in Theorem~\ref{t:PF}---and hence
Corollary~\ref{c:PF}---is not necessarily minimal: it might be that
some of the canonically defined components can be omitted.  This
occurs, for instance, in Example~\ref{e:hyperbola-PD}.  The general
phenomenon, as in this hyperbola example, stems from geometry of the
elements in~$D_\tau$ supported on~$\tau$, which need not be bounded in
any sense, even in the quotient $\cQ/\ZZ \tau$.  In contrast, for
(say) quotients by monomial ideals in the polynomial
ring~$\kk[\NN^n]$, only finitely many elements have support at the
origin, and the downset they cogenerate is consequently~artinian.
\end{example}

\subsection{Localization and support}\label{sub:local-support}\mbox{}

\begin{defn}\label{d:support}
Fix a face~$\tau$ of a partially ordered group~$\cQ$.  The
\emph{localization} of a $\cQ$-module~$\cM$ \emph{along~$\tau$} is the
tensor product
$$%
  \cM_\tau = \cM \otimes_{\kk[\cQ_{+\!}]} \kk[\cQ_{+\!} + \ZZ \tau],
$$
viewing~$\cM$ as a $\cQ$-graded $\kk[\cQ_{+\!}]$-module.  The
submodule of~$\cM\hspace{-1.17pt}$ \emph{globally
supported~on~$\tau$}~is
$$%
  \Gamma_{\!\tau} \cM
  =
  \bigcap_{\tau' \not\subseteq \tau}\bigl(\ker(\cM \to \cM_{\tau'})\bigr)
  =
  \ker \bigl(\cM \to \prod_{\tau' \not\subseteq \tau} \cM_{\tau'}\bigr).
$$
\end{defn}

\begin{example}\label{e:Gamma}
Definition~\ref{d:PF}.2 says that $1_q \in \kk[D]_q = \kk$ lies
in~$\Gamma_{\!\tau} \kk[D]$ if and only if $q \in \Gamma_{\!\tau} D$,
because $q \not\in D_{\tau'}$ if and only if $1_q \mapsto 0$ under
localization of~$\kk[D]$ along~$\tau'$.
\end{example}

\begin{example}\label{e:global-support}
The global supports of the indicator subquotient for the interval
$$%
\psfrag{x}{\tiny$x$}
\psfrag{y}{\tiny$y$}
  \begin{array}{@{}c@{}l@{}}
	{\red .}\qquad\quad\ \\[-1.7ex]
	{\red .}\qquad\quad\ \\[-1.7ex]
	{\red .}\qquad\quad\ \\[-1ex]
	\includegraphics[height=29mm]{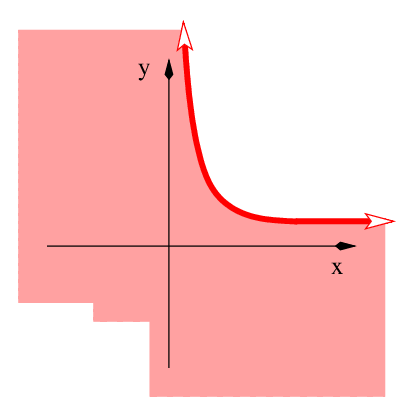}
	&\raisebox{3ex}{\red$\!\cdot\!\cdot\!\cdot$}
	\end{array}\!\!
\quad\ \ \goesto\quad\ \
  \begin{array}{@{}c@{}}
	{\red .}\quad\,\\[-1.7ex]
	{\red .}\quad\,\\[-1.7ex]
	\makebox[0pt][l]{\quad$\tau = \nothing$}
	{\red .}\quad\,\\[-1ex]
	\includegraphics[height=29mm]{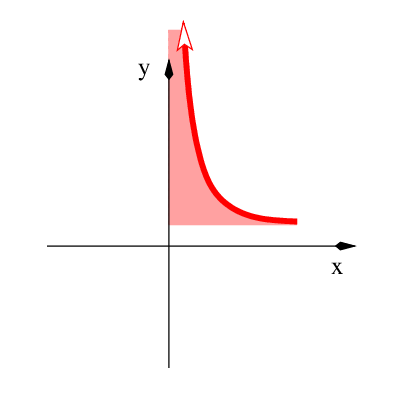}
	\end{array}
\ ,\,\quad
  \begin{array}{@{}c@{}l@{}}
	\phantom{.}\qquad\quad\ \ \\[-1.7ex]
	\phantom{.}\qquad\quad\ \ \\[-1.7ex]
	\makebox[0pt][l]{\qquad$\tau = \{x\}$}
	\phantom{.}\qquad\quad\ \ \\[-1ex]
	\includegraphics[height=29mm]{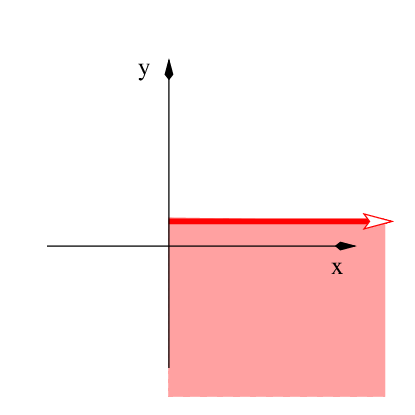}
	&\raisebox{3ex}{\red$\!\cdot\!\cdot\!\cdot$}
	\end{array}
\ ,\,\quad
  \begin{array}{@{}c@{}}
	{\red .}\qquad\quad\ \ \\[-1.7ex]
	{\red .}\qquad\quad\ \ \\[-1.7ex]
	\makebox[0pt][l]{\qquad$\tau = \{y\}$}
	{\red .}\qquad\quad\ \ \\[-1ex]
	\includegraphics[height=29mm]{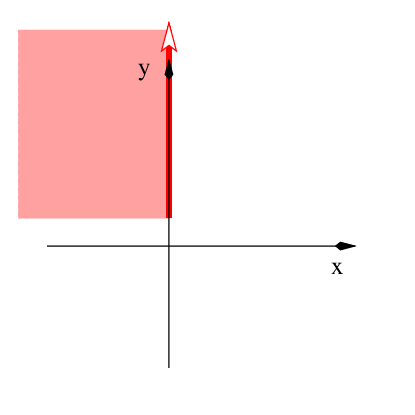}
	\end{array}
$$
in~$\RR^2$ on the left-hand side of this display are the indicator
subquotients for the intervals on the right-hand side, each labeled by
the relevant face~$\tau$.  Caution: this example is not to be confused
with Examples~\ref{e:hyperbola-GD}, \ref{e:hyperbola-PD}, \ref{e:PF},
and~\ref{e:PF'}, where the curve is a hyperbola whose asymptotes are
the two axes.  In contrast, here upper boundary of the interval has
the vertical axis as an asymptote, whereas the horizontal axis is
exactly parallel to the positive end of the upper boundary.
\end{example}

\begin{lemma}\label{l:left-exact}
The kernel of any natural transformation between two exact covariant
functors is left-exact.  In more detail, if $\alpha$ and $\beta$ are
two exact covariant functors $\cA \to \cB$ for abelian categories
$\cA$ and~$\cB$, and $\gamma_X: \alpha(X) \to \beta(X)$ naturally for
all objects~$X$ of~$\cA$, then the association $X \mapsto \ker
\gamma_X$ is a left-exact covariant functor~$\cA \to \cB$.
\end{lemma}
\begin{proof}
This can be checked by diagram chase or spectral sequence.
\end{proof}

\begin{prop}\label{p:support-left-exact}
The global support functor\/ $\Gamma_{\!\tau\!}$ is left-exact.
\end{prop}
\begin{proof}
Use Lemma~\ref{l:left-exact}: global support is the kernel of the
natural transformation from the identity to a direct product of
localizations.
\end{proof}

\begin{prop}\label{p:support-localizes}
For modules over a polyhedral partially ordered group, localization
commutes with taking support: $(\Gamma_{\!\tau'} \cM)_\tau =
\Gamma_{\!\tau'}(\cM_\tau)$, and both sides are~$0$ unless~$\tau'
\supseteq \tau$.
\end{prop}
\begin{proof}
Localization along~$\tau$ is exact, so
$$%
  \ker(\cM \to \cM_{\tau''})_\tau
  =
  \ker\bigl(\cM_\tau \to (\cM_{\tau''})_\tau\bigr)
  =
  \ker\bigl(\cM_\tau \to (\cM_\tau)_{\tau''}\bigr).
$$
Since localization along~$\tau$ commutes with finite intersections of
submodules, $(\Gamma_{\!\tau'} \cM)_\tau$ is the intersection of the
leftmost of these modules over the faces $\tau'' \not\subseteq \tau'$,
of which there are only finitely many by the polyhedral hypothesis.
But $\Gamma_{\!\tau'}(\cM_\tau)$ equals the same intersection of the
rightmost of these modules by definition.  And if $\tau' \not\supseteq
\tau$ then one of these $\tau''$ equals~$\tau$, so $\cM_\tau \to
(\cM_\tau)_{\tau''} = \cM_\tau$ is the identity map, whose
kernel~is~$0$.
\end{proof}

\begin{defn}\label{d:local-support}
Fix a $\cQ$-module $\cM$ for a polyhedral partially ordered
group~$\cQ$.  The \emph{local $\tau$-support} of~$\cM$ is the module
$\Gamma_{\!\tau} \cM_\tau$ of elements globally supported on~$\tau$ in
the localization~$\cM_\tau$, or equivalently (by
Proposition~\ref{p:support-localizes}) the localization along~$\tau$
of the submodule of~$\cM$ globally supported on~$\tau$.
\end{defn}

\begin{defn}\label{d:coprimary}%
A module $\cM$ over a polyhedral partially ordered group is
\emph{coprimary} if for some face~$\tau$, the localization map $\cM
\into \cM_\tau$ is injective and $\Gamma_{\!\tau} \cM_\tau$ is an
essential submodule of~$\cM_\tau$: every nonzero submodule
of~$\cM_\tau$ intersects $\Gamma_{\!\tau} \cM_\tau$ nontrivially.
\end{defn}

\begin{remark}\label{r:unique-face}
It is easy to check that over any polyhedral partially ordered group,
if a module $E$ is coprimary then it is $\tau$-coprimary for a unique
face~$\tau$ of~$\cQ$.
\end{remark}

\begin{remark}\label{r:ZZn-coprimary}
It is an interesting exercise to check that every element of a
coprimary module is coprimary when the polyhedral partially ordered
group is discrete (Example~\ref{e:discrete-polyhedral}) and closed
(Definition~\ref{d:closed}).
\end{remark}

The coprimary concept has an elementary, intuitive formulation in the
language of persistence, when the ambient partially ordered group is
polyhedral and closed.

\begin{defn}\label{d:elementary-coprimary}
Fix a face~$\tau$ of the positive cone~$\cQ_+$ in a partially ordered
group~$\cQ$.  A homogeneous element $y \in M_q$ in a $Q$-module~$M$ is
\begin{enumerate}
\item%
\emph{$\tau$-persistent} if it has nonzero image in $M_{q'}$ for all
$q' \in q + \tau$;

\item%
\emph{$\ol\tau$-transient} if, for each $f \in \cQ_+ \minus \tau$, the
image of~$y$ vanishes in $M_{q'}$ whenever $q' = q + \lambda f$ for
$\lambda \gg 0$;

\item%
\emph{$\tau$-coprimary} if it is $\tau$-persistent and
$\ol\tau$-transient.
\end{enumerate}
\end{defn}

\begin{thm}\label{t:elementary-coprimary}
Fix a $\cQ$-module~$\cM$ and a face~$\tau$ of the positive
cone~$\cQ_+$ in a closed polyhedral partially ordered group~$\cQ$.
The module~$\cM$ is $\tau$-coprimary if and only if every homogeneous
element divides a $\tau$-coprimary element, where $y \in M_q$
\emph{divides} $y' \in M_{q'}$ if $q \preceq q'$ and $y$ has
image~$y'$ under the structure morphism $M_q \to M_{q'}$.
\end{thm}
\begin{proof}
If $\cM$ is $\tau$-coprimary and $y \in \cM_q$ is a nonzero
homogeneous element, then $y$ is $\tau$-persistent because $\cM$ is a
submodule of~$\cM_\tau$ on which $\kk[\ZZ\tau]$ acts freely.  On the
other hand, $y$ divides a $\ol\tau$-transient element because
$\Gamma_{\!\tau} \cM_\tau$ is an essential submodule of~$\cM_\tau$:
the submodule of $\cM_\tau$ generated by~$y$ intersects
$\Gamma_{\!\tau} \cM_\tau$ nontrivially.
The closed hypothesis on~$\cQ$ implies that an element supported
on~$\tau$ is $\ol\tau$-transient, as in Example~\ref{e:support}.

The other direction does not require the closed hypothesis.  Assume
that every homogeneous element of~$\cM$ divides a $\tau$-coprimary
element.  The graded component of the localization~$\cM_\tau$ in
degree $q \in \cQ$ is the direct limit of~$\cM_q'$ over $q' \in q +
\tau$.  If $y \in \cM_q$ lies in $\ker(\cM \to \cM_\tau)$, then the
image of~$y$ must vanish in some~$\cM_{q'}$, whence $y = 0$ to begin
with, by $\tau$-persistence.  On the other hand, that $\Gamma_{\!\tau}
\cM_\tau$ is an essential submodule of~$\cM_\tau$ follows because
every $\ol\tau$-transient element is supported on~$\tau$.
\end{proof}

\subsection{Primary decomposition of modules}\label{sub:prim-decomp}

\begin{defn}\label{d:primDecomp'}
Fix a $\cQ$-module $\cM$ for a polyhedral partially ordered
group~$\cQ$.  A \emph{primary decomposition} of~$\cM$ is an injection
$\cM \into \bigoplus_{i=1}^r \cM/\cM_i$ into a direct sum of coprimary
quotients $\cM/\cM_i$, called \emph{components} of the decomposition.
\end{defn}

\begin{remark}\label{r:primary}
Primary decomposition is usually phrased in terms of \emph{primary
submodules} $\cM_i \subseteq \cM$, which by definition have coprimary
quotients~$\cM/\cM_i$, satisfying $\bigcap_{i=1}^r \cM_i = 0$
in~$\cM$.  This is equivalent to Definition~\ref{d:primDecomp'}.
\end{remark}

\begin{example}\label{e:prim-decomp-downset}
A primary decomposition $D = \bigcup_{i=1}^r D_i$ of a downset~$D$
yields a primary decomposition of the corresponding indicator
quotient, namely the injection $ \kk[D] \into \bigoplus_{i=1}^r
\kk[D_i] $ induced by the surjections $\kk[D] \onto \kk[D_i]$.  See,
e.g., Example~\ref{e:hyperbola-PD}.
\end{example}

\begin{example}\label{e:global-support'}
The interval module in Example~\ref{e:global-support} has a primary
decomposition
$$%
\psfrag{x}{\tiny$x$}
\psfrag{y}{\tiny$y$}
  \kk\!
  \left[
  \begin{array}{@{\!}c@{\!\!}}\includegraphics[height=29mm]{decomp}\end{array}
  \right]
\:\into\ 
  \kk\!
  \left[
  \begin{array}{@{\!}c@{\!\!}}\includegraphics[height=29mm]{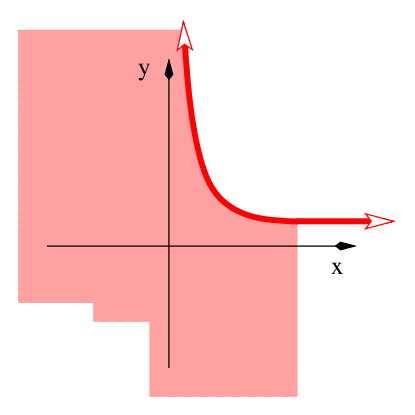}\end{array}
  \right]
\oplus\,
  \kk\!
  \left[
  \begin{array}{@{\!}c@{\!\!}}\includegraphics[height=29mm]{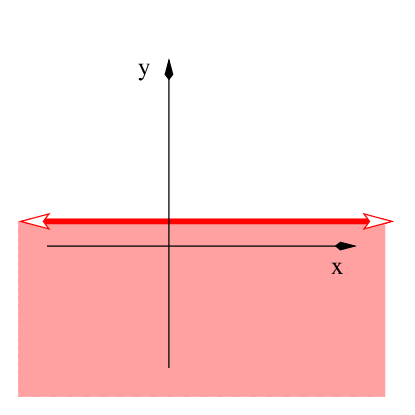}\end{array}
  \right]
\oplus\,
  \kk\!
  \left[
  \begin{array}{@{\!}c@{\!\!\!\!}}\includegraphics[height=29mm]{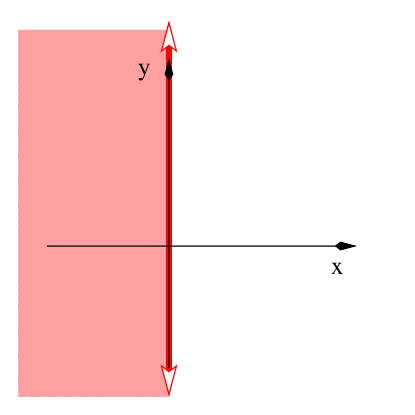}\end{array}
  \right]
$$
in which the global support along each face is extended downward so as
to become a quotient instead of a submodule of the original interval
module.
\end{example}

The existence of primary decomposition in Theorem~\ref{t:primDecomp}
is intended for tame modules, but because it deals with essential
submodules and not generators, it only requires the downset half of a
fringe presentation.

\begin{defn}\label{d:downset-hull}
A \emph{downset hull} of a module~$\cM$ over an arbitrary poset is an
injection $\cM \into \bigoplus_{j \in J} E_j$ with each $E_j$ being a
downset module The hull is \emph{finite} if $J$ is~finite.  The
module~$\cM$ is \emph{downset-finite} if it admits a finite downset
hull.
\end{defn}

\begin{thm}\label{t:primDecomp}
Every downset-finite module over a polyhedral partially ordered group
admits a primary decomposition.
\end{thm}
\begin{proof}
If $\cM \into \bigoplus_{j=1}^k E_j$ is a downset hull of the
module~$\cM$, and $E_j \into \bigoplus_{i=1}^\ell E_{ij}$ is a primary
decomposition for each~$j$ afforded by Corollary~\ref{c:PF} and
Example~\ref{e:prim-decomp-downset}, then let $E^\tau$ be the direct
sum of the downset modules $E_{ij}$ that are $\tau$-coprimary.  Set
$M^\tau = \ker(\cM \to E^\tau)$.  Then $\cM/\cM^\tau$ is coprimary,
being a submodule of a coprimary module.  Moreover, $\cM \to
\bigoplus_{\tau} \cM/\cM^\tau$ is injective because its kernel is the
same as the kernel of $\cM \to \bigoplus_{ij} E_{ij}$, which is a
composite of two injections and hence injective by construction.
Therefore $\cM \to \bigoplus_{\tau} \cM/\cM^\tau$ is a primary
decomposition.
\end{proof}

\begin{example}\label{e:circular-cone}
The finiteness of primary decomposition depends on the polyhedral
condition that posits finiteness of the number of faces of the
positive cone (Definition~\ref{d:face}).  When the positive cone has
infinitely many faces, such as the positive half~$Q_+$ of the right
circular cone $x^2 + y^2 \leq z^2$ in~$Q = \RR^3$, the $\cQ$-module
$$%
  \kk[\del Q_+] = \kk[Q_+]/\kk[Q_+^\circ]
$$
does not admit a finite primary decomposition.  The module $M =
\kk[\del Q_+]$ has a vector space of dimension~$1$ on the boundary of
the positive cone and~$0$ elsewhere.  Every face of the positive cone
must get its own summand $M/M_i$ in Definition~\ref{d:primDecomp'} for
the homomorphism $M \to \bigoplus_{i=1}^r \cM/\cM_i$ there to be
injective, and in that case the infinite number of faces would force
the direct sum to become a direct product.  This particular example,
with the right circular cone, works as well in the discrete partially
ordered group~$\ZZ^3$ because the circle has infinitely many rational
points.
\end{example}

\section{Finitely determined \texorpdfstring{$\ZZ^n$}{Zn}-modules}\label{s:ZZn}

Unless otherwise stated, this section is presented over the discrete
polyhedral partially ordered group $\cQ = \ZZ^n$ with $\cQ_+ = \NN^n$.
It begins by reviewing the structure of finitely determined
$\ZZ^n$-mod\-ules, including (minimal) injective and flat resolutions.
These are the foundation underlying the syzygy theorem for tame
modules (Section~\ref{sub:syzygy}), including the existence of fringe
presentations.  They also serve as models for the concepts of socle,
cogenerator, and downset hull over real polyhedral groups, covered in
the sequel to this work \cite{essential-real}, as well as their dual
notions of top, generator, and upset covers.

The main references for $\ZZ^n$-modules used here are
\cite{alexdual,cca}.  The development of homological theory for
injective and flat resolutions in the context of finitely determined
modules is functorially equivalent to the development for finitely
generated modules, by \cite[Theorem~2.11]{alexdual}, but it is
convenient to have on hand the statements in the finitely determined
case directly.  The characterization of finitely determined modules in
Proposition~\ref{p:determined} and (hence)
Theorem~\ref{t:finitely-determined} is apparently new.

\subsection{Definitions}\label{sub:def-finitely-determined}\mbox{}

\noindent
The essence of the finiteness here is that all of the relevant
information about the relevant modules should be recoverable from what
happens in a bounded box in~$\ZZ^n$.

\begin{defn}\label{d:determined}
A $\ZZ^n$-finite module~$\cN$ is \emph{finitely determined} if for
each $i = 1,\dots,n$ the multiplication map $\cdot x_i: N_\bb \to
N_{\bb+\ee_i}$ (see Example~\ref{e:ZZn-graded} for notation) is an
isomorphism whenever $b_i$ lies outside of some bounded interval.
\end{defn}

\begin{remark}\label{r:determined}
This notion of finitely determined is the same notion as in
Example~\ref{e:convex-projection}.  A module is finitely determined if
and only if, after perhaps translating its $\ZZ^n$-grading, it is
\emph{$\aa$-determined} for some $\aa \in \NN^n$, as defined in
\cite[Definition~2.1]{alexdual}.
\end{remark}

\begin{remark}\label{r:fg/ZZ^n}
For $\ZZ^n$-modules, the finitely determined condition is
weaker---that is, more inclusive---than finitely generated, but it is
much stronger than tame or (equivalently, by Theorem~\ref{t:tame})
finitely encoded.  The reason is essentially
Example~\ref{e:convex-projection}, where the encoding has a very
special nature.  For a generic sort of example, the restriction
to~$\ZZ^n$ of any $\RR^n$-finite $\RR^n$-module with finitely many
constant regions of sufficient width is a tame $\ZZ^n$-module, and
there is simply no reason why the constant regions should be
commensurable with the coordinate directions in~$\ZZ^n$.  Already the
toy-model fly wing modules in Examples~\ref{e:toy-model-fly-wing}
and~\ref{e:encoding} yield infinitely generated but tame
$\ZZ^n$-modules, and this remains true when the discretization $\ZZ^n$
of~$\RR^n$ is rescaled by any factor.
\end{remark}

\begin{example}\label{e:local-cohomology}
The local cohomology of an affine semigroup rings is tame but usually
not finitely determined; see \cite{injAlg} and \cite[Chapter~13]{cca},
particularly Theorem~13.20, Example~13.17, and Example~13.4 in the
latter.
\end{example}

\subsection{Injective hulls and resolutions}\label{sub:inj}\mbox{}

\begin{remark}\label{r:injective}
Every $\ZZ^n$-finite module that is injective in the category of
$\ZZ^n$-modules is a direct sum of downset modules $\kk[D]$ for
downsets $D$ cogenerated (Example~\ref{e:coprincipal}) by vectors
along faces.  This statement holds over any polyhedral discrete
partially ordered group (Definition~\ref{d:face} and
Example~\ref{e:discrete-pogroup}) by \cite[Theorem~11.30]{cca}.
\end{remark}

Minimal injective resolutions work for finitely determined modules
just as they do for finitely generated modules.  The standard
definitions are as follows.

\begin{defn}\label{d:inj}
Fix a $\ZZ^n$-module~$\cN$.
\begin{enumerate}
\item%
An \emph{injective hull} of~$\cN$ is an injective homomorphism $\cN
\to E$ in which $E$ is an injective $\ZZ^n$-module (see
Remark~\ref{r:injective}).  This injective hull is
\begin{itemize}
\item%
\emph{finite} if $E$ has finitely many indecomposable summands and
\item%
\emph{minimal} if the number of such summands is minimal.
\end{itemize}
\item%
An \emph{injective resolution} of~$\cN$ is a complex~$E^\spot$ of
injective $\ZZ^n$-modules whose differential $E^i \to E^{i+1}$ for $i
\geq 0$ has only one nonzero homology $H^0(E^\spot) \cong\nolinebreak
\cN$ (so $\cN \into E^0$ and $\coker(E^{i-1} \to E^i) \into E^{i+1}$
are injective hulls for all $i \geq 1$).  \nolinebreak$E^\spot$
\begin{itemize}
\item%
has \emph{length~$\ell$} if $E^i = 0$ for $i > \ell$ and $E^\ell \neq
0$;
\item%
is \emph{finite} if $E^\spot = \bigoplus_i E^i$ has finitely many
indecomposable summands; and
\item%
is \emph{minimal} if $\cN \into E^0$ and $\coker(E^{i-1} \to E^i)
\into E^{i+1}$ are minimal injective hulls for all $i \geq 1$.
\end{itemize}
\end{enumerate}
\end{defn}

\begin{prop}\label{p:determined}
The following are equivalent for a $\ZZ^n$-module~$\cN$.
\begin{enumerate}
\item%
$\cN$ is finitely determined.
\item%
$\cN$ admits a finite injective resolution.
\item%
$\cN$ admits a finite minimal injective resolution.
\end{enumerate}
Any finite minimal resolution is unique up to isomorphism and has
length~at~most~$n$.
\end{prop}
\begin{proof}
The proof is based on existence of finite minimal injective hulls and
resolutions for finitely generated $\ZZ^n$-modules, along with
uniqueness and length~$n$ given minimality, as proved by Goto and
Watanabe \cite{GWii}.

First assume $\cN$ is finitely determined.  Translating the
$\ZZ^n$-grading affects nothing about existence of a finite injective
resolution.  Therefore, using Remark~\ref{r:determined}, assume that
$\cN$ is $\aa$-determined.  Truncate by taking the $\NN^n$-graded part
of~$\cN$ to get a positively $\aa$-determined---and hence finitely
generated---module~$\cN_{\succeq\0}$; see
\cite[Definition~2.1]{alexdual}.  Take any minimal injective
resolution $\cN_{\succeq\0} \to E^\spot$.  Extend backward using the
\v Cech hull \cite[Definition~2.7]{alexdual}, which is exact
\cite[Lemma~2.9]{alexdual}, to get a finite minimal injective
resolution $\vC(\cN_{\succeq\0} \to E^\spot) = (\cN \to \vC E^\spot)$,
noting that $\vC$ fixes indecomposable injective modules whose
$\NN^n$-graded parts are nonzero and is zero on all other
indecomposable injective modules \cite[Lemma~4.25]{alexdual}.  This
proves 1~$\implies$~3.

That 3 $\implies$~2 is trivial.  The remaining implication,
2~$\implies$~1, follows because every indecomposable injective is
finitely determined and the category of finitely determined modules is
abelian.  (The category of $\ZZ^n$-modules each of which is nonzero
only in a bounded set of degrees is abelian, and constructions such as
kernels, cokernels, or direct sums in the category of finitely
determined modules are pulled back from there.)
\end{proof}

\subsection{Flat covers and resolutions}\label{sub:flat}\mbox{}

\noindent
Minimal flat resolutions are not commonplace, but the notion is Matlis
dual to that of minimal injective resolution.  In the context of
finitely determined modules, flat resolutions work as well as
injective resolutions.  The definitions are as follows.

\begin{defn}\label{d:flat}
Fix a $\ZZ^n$-module~$\cN$.
\begin{enumerate}
\item%
A \emph{flat cover} of~$\cN$ is a surjective homomorphism $F \to \cN$
in which $F$ is a flat $\ZZ^n$-module (see Remark~\ref{r:flat}).  This
flat cover is
\begin{itemize}
\item%
\emph{finite} if $F$ has finitely many indecomposable summands and
\item%
\emph{minimal} if the number of such summands is minimal.
\end{itemize}
\item%
A \emph{flat resolution} of~$\cN$ is a complex~$F_\spot$ of flat
$\ZZ^n$-modules whose differential $F_{i+1} \to F_i$ for $i \geq 0$
has only one nonzero homology $H_0(F_\spot) \cong \cN$ (so $F_0 \onto
\cN$ and $F_{i+1} \onto \ker(F_i \to F_{i-1})$ are flat covers for all
$i \geq 1$).  The flat resolution~$F_\spot$
\begin{itemize}
\item%
has \emph{length~$\ell$} if $F_i = 0$ for $i > \ell$ and $F_\ell \neq
0$;
\item%
is \emph{finite} if $F_\spot = \bigoplus_i F_i$ has finitely many
indecomposable summands; and
\item%
is \emph{minimal} if $F_0 \onto \cN$ and $F_{i+1} \onto \ker(F_i \to
F_{i-1})$ are minimal flat covers for all $i \geq 1$.
\end{itemize}
\end{enumerate}
\end{defn}

\begin{defn}\label{d:matlis}
The \emph{Matlis dual} of a $\ZZ^n$-module~$\cM$ is the
$\ZZ^n$-module~$\cM^\vee$ defined by
$$%
  (\cM^\vee)_\aa = \Hom_\kk(\cM_{-\aa},\kk),
$$
so the homomorphism $(\cM^\vee)_\aa \to (M^\vee)_\bb$ is transpose to
$\cM_{-\bb} \to \cM_{-\aa}$.
\end{defn}

\begin{lemma}\label{l:vee-vee}
$(\cM^\vee)^\vee\!$ is canonically isomorphic to~$\cM$ for any
$\ZZ^n$-finite module~$\cM$.~\hspace{1ex}$\square$
\end{lemma}

\begin{remark}\label{r:flat}
By the adjunction between Hom and $\otimes$, a module is flat if and
only its Matlis dual is injective (see \cite[\S1.2]{alexdual}, for
example).  The Matlis dual of Remark~\ref{r:injective} therefore says
that every $\cQ$-finite flat module over a discrete polyhedral
partially ordered group~$\cQ$ is isomorphic to a direct sum of upset
modules~$\kk[U]$ for upsets of the form $U = \bb + \ZZ\tau + \cQ_+$.
These upset modules are the graded translates of localizations
of~$\kk[\cQ_+]$ along faces.
\end{remark}

\subsection{Flange presentations}\label{sub:flange}\mbox{}

\begin{defn}\label{d:flange}
Fix a $\ZZ^n$-module~$\cN$.
\begin{enumerate}
\item%
A \emph{flange presentation} of~$\cN$ is a $\ZZ^n$-module morphism
$\phi: F \to E$, with image isomorphic to~$\cN$, where $F$ is flat and
$E$ is injective in the category of \mbox{$\ZZ^n$-modules}.
\item%
If $F$ and~$E$ are expressed as direct sums of indecomposables, then
$\phi$ is \emph{based}.
\item%
If $F$ and~$E$ are finite direct sums of indecomposables, then $\phi$
is \emph{finite}.
\item%
If the number of indecomposable summands of~$F$ and~$E$ are
simultaneously minimized then $\phi$ is \emph{minimal}.
\end{enumerate}
\end{defn}

\begin{remark}\label{r:portmanteau-fl}
The term \emph{flange} is a portmanteau of \emph{flat} and
\emph{injective} (i.e., ``flainj'') because a flange presentation is
the composite of a flat cover and an injective hull.
\end{remark}

The same notational trick to make fringe presentations effective data
structures (Definition~\ref{d:monomial-matrix-fr}) works on flange
presentations.

\begin{defn}\label{d:monomial-matrix-fl}
Fix a based finite flange presentation $\phi:
\bigoplus_p\hspace{-.2pt} F_p = F \to E =
\nolinebreak\bigoplus_q\hspace{-.2pt} E_q$.  A \emph{monomial matrix}
for $\phi$ is an array of \emph{scalar entries}~$\phi_{qp}$ whose
columns are labeled by the indecomposable flat summands~$F_p$ and
whose rows are labeled by the indecomposable injective summands~$E_q$:
$$%
\begin{array}{ccc}
  &
  \monomialmatrix
	{F_1\\\vdots\ \\F_k\\}
	{\begin{array}{ccc}
		   E_1    & \cdots &    E_\ell   \\
		\phi_{11} & \cdots & \phi_{1\ell}\\
		\vdots    & \ddots &   \vdots    \\
		\phi_{k1} & \cdots & \phi_{k\ell}\\
	 \end{array}}
	{\\\\\\}
\\
  F_1 \oplus \dots \oplus F_k = F
  & \fillrightmap
  & E = E_1 \oplus \dots \oplus E_\ell.
\end{array}
$$
\end{defn}

The entries of the matrix $\phi_{\spot\spot}$ correspond to
homomorphisms $F_p \to E_q$.

\begin{lemma}\label{l:F->E}
If $F = \kk[\aa + \ZZ\tau' + \NN^n]$ is an indecomposable flat
$\ZZ^n$-module and $E = \kk[\bb + \ZZ\tau - \NN^n]$ is an
indecomposable injective $\ZZ^n$-module, then $\Hom_{\ZZ^n}(F, E) = 0$
unless $(\aa + \ZZ\tau' + \NN^n) \cap (\bb + \ZZ\tau - \NN^n) \neq
\nothing$, in which case $\Hom_{\ZZ^n}(F, E) = \kk$.
\end{lemma}
\begin{proof}
Corollary~\ref{c:U->D}.\ref{i:kk}.
\end{proof}

\begin{defn}\label{d:F<E}
In the situation of Lemma~\ref{l:F->E}, write $F \preceq E$ if their
degree sets have nonempty intersection: $(\aa + \ZZ\tau' + \NN^n) \cap
(\bb + \ZZ\tau - \NN^n) \neq \nothing$.
\end{defn}

\begin{prop}\label{p:scalars-fl}
With notation as in Definition~\ref{d:monomial-matrix-fl}, $\phi_{pq} =
0$ unless $F_p \preceq E_q$.  Conversely, if an array of scalars
$\phi_{qp} \in \kk$ with rows labeled by indecomposable flat modules
and columns labeled by indecomposable injectives has $\phi_{pq} = 0$
unless $F_q \preceq E_q$, then it represents a flange presentation.
\end{prop}
\begin{proof}
Lemma~\ref{l:F->E} and Definition~\ref{d:F<E}.
\end{proof}

The unnatural hypothesis that a persistence module be finitely
generated results in data types and structure theory that are
asymmetric regarding births as opposed to deaths.  In contrast, the
notion of flange presentation is self-dual: their duality interchanges
the roles of births~($F$) and deaths~($E$).

\begin{prop}\label{p:duality}
A $\ZZ^n$-module $\cN$ has a finite flange presentation $F \to E$ if
and only if the Matlis dual $E^\vee \to F^\vee$ is a finite flange
presentation of the Matlis dual $\cN^\vee$.
\end{prop}
\begin{proof}
Matlis duality is an exact, contravariant functor on~$\ZZ^n$-modules
that takes the subcategory of finitely determined $\ZZ^n$-modules to
itself (these properties are immediate from the definitions),
interchanges flat and injective objects therein, and has the property
that the natural map $(\cN^\vee)^\vee \to \cN$ is an isomorphism for
finitely determined~$\cN$ (Lemma~\ref{l:vee-vee}); see
\cite[\S1.2]{alexdual} for a discussion of these properties.
\end{proof}

\subsection{Syzygy theorem for \texorpdfstring{$\ZZ^n$}{Zn}-modules}\label{sub:Zsyzygy}

\begin{thm}\label{t:finitely-determined}
A $\ZZ^n$-module is finitely determined if and only if it admits one,
and hence all, of the following:
\begin{enumerate}
\item\label{i:flange}%
a finite flange presentation; or
\item\label{i:flat-presentation}%
a finite flat presentation; or
\item\label{i:injective-copresentation}%
a finite injective copresentation; or
\item\label{i:flat-res}%
a finite flat resolution; or
\item\label{i:injective-res}%
a finite injective resolution; or
\item\label{i:minimal}%
a minimal one of any of the above.
\end{enumerate}
Any minimal one of these objects is unique up to noncanonical
isomorphism, and the resolutions have length at most~$n$.
\end{thm}
\begin{proof}
The hard work is done by Proposition~\ref{p:determined}.  It implies
that $\cN$ is finitely determined $\iff \cN^\vee$ has a minimal
injective resolution $\iff \cN$ has a minimal flat resolution of
length at most~$n$, since the Matlis dual of any finitely determined
module~$\cN$ is finitely determined.  Having both a minimal injective
resolution and a minimal flat resolution is stronger than having any
of items~\ref{i:flange}--\ref{i:injective-copresentation}, minimal or
otherwise, so it suffices to show that $\cN$ is finitely determined if
$\cN$ has any of
items~\ref{i:flange}--\ref{i:injective-copresentation}.  This follows,
using that the category of finitely determined modules~is~abelian as
in the proof of Proposition~\ref{p:determined}, from the fact that
every indecomposable injective or flat $\ZZ^n$-module is finitely
determined.
\end{proof}

\begin{remark}\label{r:finitely-determined}
Conditions~\ref{i:flange}--\ref{i:minimal} in
Theorem~\ref{t:finitely-determined} remain equivalent for
$\RR^n$-modules, with the standard positive cone $\RR^n_+$, assuming
that the finite flat and injective modules in question are finite
direct sums of localizations of~$\RR^n$ along faces and their Matlis
duals.  (The equivalence, including minimality, is a consequence of
the generator and cogenerator theory over real polyhedral groups
\cite{essential-real}.)  The equivalent conditions do not characterize
$\RR^n$-modules that are pulled back under convex projection from
arbitrary modules over an interval in~$\RR^n$, though, because all
sorts of infinite things can can happen inside of a box, such as
having generators~along~a~curve.
\end{remark}

\section{Homological algebra of poset modules}\label{s:syzygy}

\subsection{Indicator resolutions}\label{sub:indicator-res}

\begin{defn}\label{d:resolutions}
Fix any poset~$\cQ$ and a $\cQ$-module~$\cM$.
\begin{enumerate}
\item%
An \emph{upset resolution} of~$\cM$ is a complex~$F_\spot$ of
$\cQ$-modules, each a direct sum of upset submodules of~$\kk[\cQ]$,
whose differential $F_i \to F_{i-1}$ decreases homological degrees,
has components $\kk[U] \to \kk[U']$ that are connected
(Definition~\ref{d:connected-homomorphism}), and has only one nonzero
homology $H_0(F_\spot) \cong \cM$.

\item%
A \emph{downset resolution} of~$\cM$ is a complex~$E^\spot$ of
$\cQ$-modules, each a direct sum of downset quotient modules
of~$\kk[\cQ]$, whose differential $E^i \to E^{i+1}$ increases
cohomological degrees, has components $\kk[D'] \to \kk[D]$ that are
connected, and has only one nonzero homology $H^0(E^\spot)
\cong\nolinebreak \cM$.
\end{enumerate}\setcounter{separated}{\value{enumi}}
An upset or downset resolution is called an \emph{indicator
resolution} if the up- or down- nature is unspecified.  The
\emph{length} of an indicator resolution is the largest
(co)homological degree in which the complex is nonzero.  An indicator
resolution
\begin{enumerate}\setcounter{enumi}{\value{separated}}
\item%
is \emph{finite} if the number of indicator module summands is finite,

\item%
\emph{dominates} a constant subdivision or encoding of~$\cM$ if the
subdivision or encoding is subordinate to each indicator summand, and

\item\label{i:auxiliary-resolution}%
is \emph{semialgebraic}, \emph{PL}, \emph{subanalytic}, or \emph{of
class~$\mathfrak X$} if $\cQ$ is a subposet of a real partially
ordered group and the resolution dominates a constant subdivision or
encoding of the corresponding type.
\end{enumerate}
\end{defn}

\begin{defn}\label{d:resolution-monomial-matrix}
Monomial matrices for indicator resolutions are defined similarly to
those for fringe presentations in
Definition~\ref{d:monomial-matrix-fr}, except that for the
cohomological case the row and column labels are source and target
downsets, respectively, while in the homological case the row and
column labels are target and source upsets, respectively:
$$%
\begin{array}{ccc}
  &
  \monomialmatrix
	{\vdots\ \\D^i_p\\\vdots\ }
	{\begin{array}{ccc}
		 \cdots & D^{i+1}_q & \cdots \\
		        &           &        \\
		        & \phi_{pq} &        \\
		        &           &        \\
	 \end{array}}
	{\\\\\\}
\\
    E^i
  & \fillrightmap
  & E^{i+1}
\end{array}
\qquad\text{and}\qquad
\begin{array}{ccc}
  &
  \monomialmatrix
	{\vdots\ \\U_i^p\\\vdots\ }
	{\begin{array}{ccc}
		 \cdots & U_{i+1}^q & \cdots \\
		        &           &        \\
		        & \phi_{pq} &        \\
		        &           &        \\
	 \end{array}}
	{\\\\\\}
\\
    F_i
  & \filleftmap
  & F_{i+1}.
\end{array}
$$
(Note the switch of source and target from cohomological to
homological, so the map goes from right to left in the homological
case, with decreasing homological indices.)
\end{defn}

As in Proposition~\ref{p:pullback-monomial-matrix}, pullbacks have
transparent monomial matrix interpretations.

\begin{prop}\label{p:syzygy-monomial-matrix}
Fix a poset~$\cQ$ and an encoding of a $\cQ$-module~$\cM$ by a poset
morphism $\pi: \cQ \to \cP$ and $\cP$-module~$\cH$.  Monomial matrices
for any indicator resolution of~$\cH$ pull back to monomial matrices
for an indicator resolution of~$\cM$ that dominates the encoding by
replacing the row and column labels with their preimages under~$\pi$.
\hfill$\square$
\end{prop}

\begin{defn}\label{d:indicator-(co)presentation}
Fix any poset~$\cQ$ and a $\cQ$-module~$\cM$.
\begin{enumerate}
\item%
An \emph{upset presentation} of~$\cM$ is an expression of~$\cM$ as the
cokernel of a homomorphism $F_1 \to F_0$ such that each $F_i$ is a
direct sum of upset modules and every component $\kk[U'] \to \kk[U]$
of the homomorphism is connected
(Definition~\ref{d:connected-homomorphism}).

\item%
A \emph{downset copresentation} of~$\cM$ is an expression of~$\cM$ as
the kernel of a homomorphism $E^0 \to E^1$ such that each $E^i$ is a
direct sum of downset modules and every component $\kk[D] \to \kk[D']$
of the homomorphism is connected.
\end{enumerate}\setcounter{separated}{\value{enumi}}
These \emph{indicator presentations} are \emph{finite}, or
\emph{dominate} a constant subdivision or encoding of~$\cM$, or are
\emph{semialgebraic}, \emph{PL}, \emph{subanalytic}, or \emph{of
class~$\mathfrak X$} as in Definition~\ref{d:resolutions}.
\end{defn}

\begin{example}\label{e:one-param-upset}
In one parameter, the bar $[a,b)$ in Example~\ref{e:one-param-fringe},
has upset presentation
$$%
\hspace{8ex}
\psfrag{a}{\footnotesize\raisebox{-.2ex}{$a$}}
\psfrag{b}{\footnotesize\raisebox{-.2ex}{$b$}}
\psfrag{vert-to}{\small\raisebox{-.2ex}{$\downarrow$}}
\psfrag{vert-into}{\small\raisebox{-.2ex}{$\lhookdownarrow$}}
\psfrag{vert-onto}{\small\raisebox{-.2ex}{$\twoheaddownarrow$}}
\psfrag{has image}{}
\begin{array}{@{}c@{}}
\includegraphics[height=20mm]{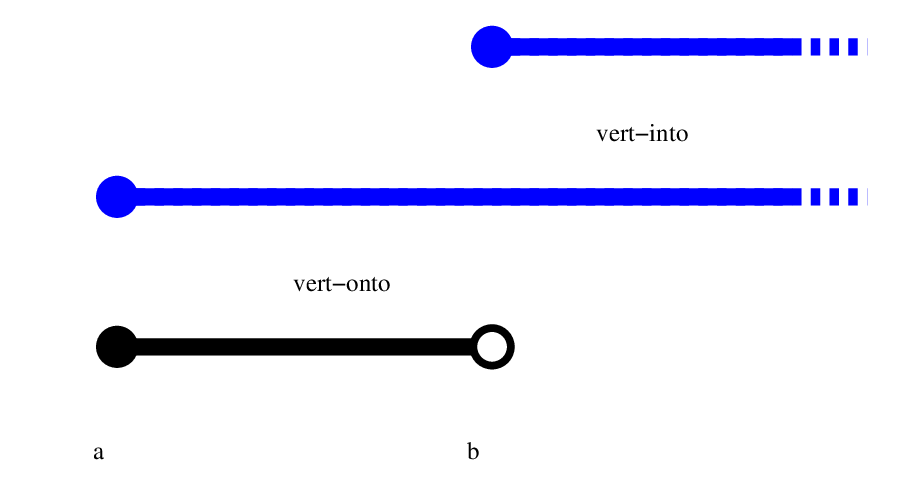}\\[-5ex]
\llap{with cokernel\hspace{10ex}}\\[2ex]
\end{array}
$$
isomorphic to the single bar.  When there are multiple bars, the
bijection from left to right endpoints yields a monomial matrix whose
scalar entries again form an identity matrix, with rows labeled by
positive rays having the specified left endpoints (the ray is the
whole real line when the left endpoint is~$-\infty$) and columns
labeled by positive rays having the corresponding right
endpoints---but with their open or closed nature reversed---as left
endpoints (the ray is empty when the specified right
endpoint~is~$+\infty$).
\end{example}

\begin{example}\label{e:two-param-upset}
$$%
\begin{array}{c}
\\[-1.48ex]
\begin{array}{@{}r@{\hspace{-.4pt}}|@{}l@{}}
\includegraphics[height=25mm]{semialgebraic}&\ \,\hspace{-.3pt}\\[-4.2pt]\hline
\end{array}
\\[-1.52ex]\mbox{}
\end{array}
\ \
\begin{array}{@{}c@{}}
\text{is the cokernel of}\\[2ex]
\end{array}
\qquad
\begin{array}{@{}r@{\hspace{-6.3pt}}|@{}l@{}}
\raisebox{-.2mm}{\includegraphics[height=25mm]{upset-blue}}
&\ \,\\[-4pt]\hline
\end{array}
\,\otni\!
\begin{array}{@{}r@{\hspace{-5.6pt}}|@{}l@{}}
\raisebox{-.4mm}{\includegraphics[height=25mm]{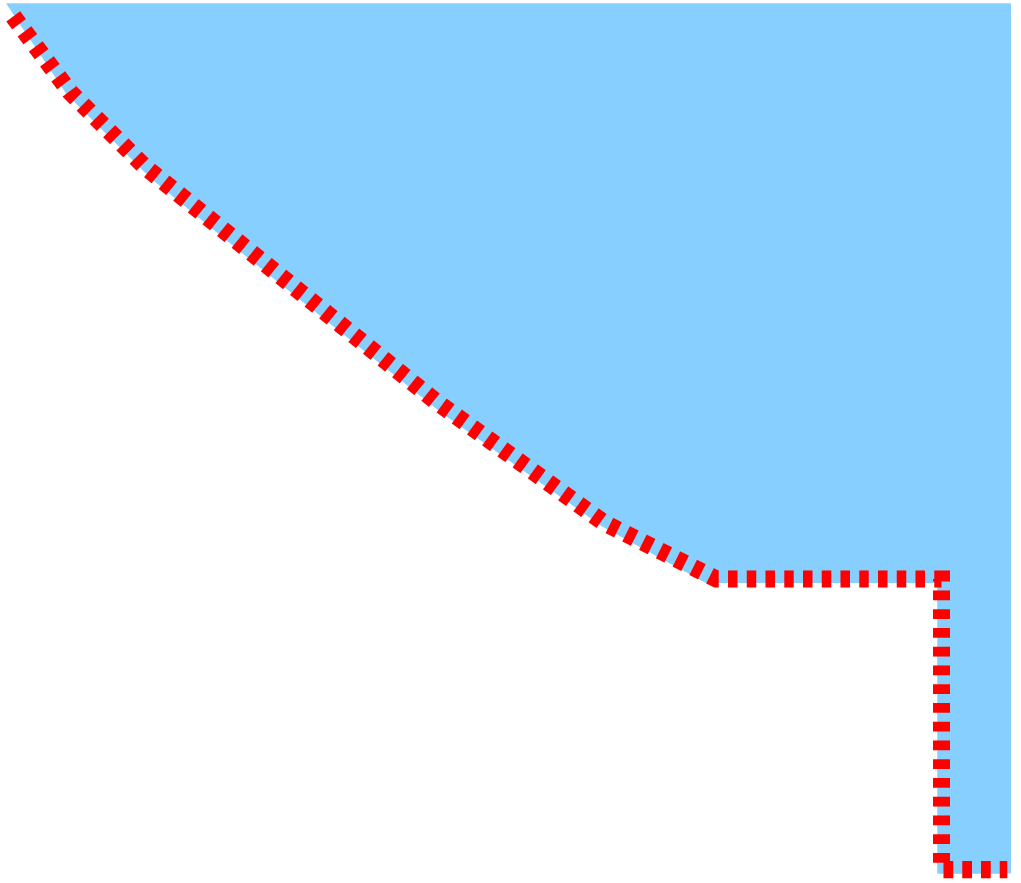}}
&\ \,\hspace{-.2pt}\\[-4pt]\hline
\end{array}
$$
\end{example}

\begin{lemma}\label{l:morphisms}
The homomorphisms in indicator presentations and resolutions are tame,
so their kernels and cokernels are tame.  If the indicator modules in
question are semialgebraic, PL, subanalytic, or of class~$\mathfrak X$
then the morphisms are, as well.
\end{lemma}
\begin{proof}
Any connected homomorphism among indicator modules is tame---and
satisfies one of the auxiliary hypotheses, if the source and target
do---by Definition~\ref{d:tame-morphism}, so the conclusion follows
from Proposition~\ref{p:abelian-category}.
\end{proof}

\begin{example}\label{e:puuska-nonconstant-isotypic'}
The poset module in Example~\ref{e:puuska-nonconstant-isotypic} has an
upset presentation
$$%
\begin{array}{ccc}
  &
  \monomialmatrix
	{L\\R\\}
	{\begin{array}{rr}
		 T &  B \\
		 2 &  1 \\
		-1 & -1 \\
	 \end{array}}
	{\\\\}
\\
  \kk[L] \oplus \kk[R]
  & \filleftmap
  & \kk[T]
\end{array}
$$
in which the monomial matrix has row and column labels
\begin{itemize}
\item%
$L$, the upset generated by the leftmost element;
\item%
$R$, the upset generated by the rightmost element;
\item%
$T$, the upset consisting solely of the maximal element depicted on
top; and
\item%
$B$, the upset consisting solely of the maximal element depicted on
the bottom.
\end{itemize}
Although the disjoint union of $T$ and~$B$ is an upset, and there is a
homomorphism $\phi: \kk[T \cup B] \to \kk[L] \oplus \kk[R]$ whose
cokernel is the desired poset module, there is no way to arrange for
the homomorphism~$\phi$ to be connected.
\end{example}

\begin{remark}\label{r:augmentation}
It is tempting to think that a fringe presentation is nothing more
than the concatenation of the augmentation map of an upset resolution
(that is, the surjection at the end) with the augmentation map of a
downset resolution (that is, the injection at the beginning), but
there is no guarantee that the components $F_i \to E_j$ of the
homomorphism thus produced are connected
(Definition~\ref{d:connected-homomorphism}).  In contrast, a flange
presentation (Definition~\ref{d:flange}) is in fact nothing more than
the concatenation of the augmentation maps of a flat resolution and an
injective resolution, since connected homomorphisms are forced by
Lemma~\ref{l:F->E}.
\end{remark}

\subsection{Syzygy theorem for modules over posets}\label{sub:syzygy}

\begin{prop}\label{p:pushforward}
For any inclusion $\iota: \cP \to \cZ$ of posets and
$\cP$-module~$\cH$ there is a $\cZ$-module $\iota_*\cH$, the
\emph{pushforward to~$\cZ$}, whose restriction to~$\iota(\cP)$
is~$\cH$ and is universally repelling: $\iota_*\cH$ has a canonical
map to every $\cZ$-module whose restriction to~$\iota(\cP)$ is~$\cH$.
\end{prop}
\begin{proof}
At $z \in \cZ$ the pushforward $\iota_*\cH$ places the colimit
$\dirlim\cH_{\preceq z}$ of the diagram of vector spaces indexed by
the elements of~$\cP$ whose images precede~$z$.  The universal
pro\-perty of colimits implies that $\iota_*\cH$ is a $\cZ$-module
with the desired universal property.
\end{proof}

\begin{remark}\label{r:kan-extension}
With perspectives as in Remark~\ref{r:curry}, the pushforward is a
left Kan extension \cite[Remark~4.2.9]{curry-thesis}.  This instance
is a special case of~\cite[Example~4.4]{curry2019}.
\end{remark}

\pagebreak

\begin{thm}[Syzygy theorem]\label{t:syzygy}
A module~$\cM$ over a poset~$\cQ$ is tame if and only if it admits
one, and hence all, of the following:
\begin{enumerate}
\item\label{i:syzygy-tame}%
a finite constant subdivision of~$\cQ$ subordinate to~$\cM$; or

\item\label{i:syzygy-encoding}%
a finite poset encoding subordinate to~$\cM$; or

\item\label{i:fringe}%
a finite fringe presentation; or

\item\label{i:upset-presentation}%
a finite upset presentation; or

\item\label{i:downset-copresentation}%
a finite downset copresentation; or

\item\label{i:upset-res}%
a finite upset resolution; or

\item\label{i:downset-res}%
a finite downset resolution; or

\item\label{i:dominating}%
any of the above dominating any given finite encoding; or

\item\label{i:subordinate-encoding}%
a finite encoding subordinate to any given one of
items~\ref{i:syzygy-tame}--\ref{i:downset-res}; or

\item\label{i:subordinate-constant}%
a finite constant subdivision subordinate to any given one of
items~\ref{i:syzygy-tame}--\ref{i:downset-res}.
\end{enumerate}
The statement remains true over any subposet of a real partially
ordered group if ``tame'' and all occurrences of ``finite'' are
replaced by ``semialgebraic'', ``PL'', or ``class~$\mathfrak X$''.
Moreover, any tame or semialgebraic, PL, or class~$\mathfrak X$
morphism $\cM \to \cM'$ lifts to a similarly well behaved morphism of
presentations or resolutions as in
parts~\ref{i:fringe}--\ref{i:downset-res}.  All of these results
except item~\ref{i:subordinate-encoding} hold in the subanalytic case
if~$\cM$ has compact support.
\end{thm}
\begin{proof}
Tame is equivalent to item~\ref{i:syzygy-tame} without auxiliary
hypotheses by Definition~\ref{d:tame} and with auxiliary hypotheses by
Definition~\ref{d:auxiliary-hypotheses}.  Tame is equivalent to
item~\ref{i:syzygy-encoding} by
Theorem~\ref{t:tame}.\ref{i:admits-finite-encoding}.  With auxiliary
hypotheses, \ref{i:syzygy-tame} $\implies$ \ref{i:syzygy-encoding} by
Theorem~\ref{t:tame}.\ref{i:auxiliary-uptight}; to apply that result
in the subanalytic case starting from an arbitrary subanalytic finite
constant subdivision subordinate to~$\cM$, construct a compact such
subdivision by keeping the bounded constant regions as they are and
taking the union of all unbounded constant regions to get a single
unbounded one.  The implication \ref{i:syzygy-encoding} $\implies$
\ref{i:syzygy-tame} holds because the fibers of the encoding poset
morphism form a constant subdivision of the relevant type.

The necessity to construct an auxiliary compact subdivision from the
given one is the reason to exclude item~\ref{i:subordinate-encoding}
from the subanalytic case, as the upcoming argument produces constant
subdivisions, not directly encodings.  For all of the other cases,
item~\ref{i:subordinate-encoding} proceeds via
item~\ref{i:subordinate-constant}, given the uptight constructions in
the previous paragraph.  For item~\ref{i:subordinate-constant}, to
produce a subordinate finite constant subdivision given a finite
fringe presentation, take the common refinement of the canonical
constant subdivision subordinate to each of its indicator summands.
The same construction works if indicator presentations or resolutions
are given, and it preserves auxiliary hypotheses by
Proposition~\ref{p:auxiliary-hypotheses}.\ref{i:classes}.

What remains is item~\ref{i:dominating}: a finitely encoded
$\cQ$-module~$\cM$ has finite upset and downset resolutions and
(co)presentations, as well as a finite fringe presentation, all
dominating the given encoding.  (As noted in the first paragraph, the
fibers of the encoding morphism are already a constant subdivision of
the relevant type.)  The domination takes care of the cases with
auxiliary hypotheses by
Definitions~\ref{d:fringe}.\ref{i:auxiliary-fringe},
\ref{d:subordinate-encoding}.\ref{i:auxiliary-encoding},
\ref{d:resolutions}.\ref{i:auxiliary-resolution},
and~\ref{d:indicator-(co)presentation}.

Fix a $\cQ$-module~$\cM$ finitely encoded by a poset morphism $\pi:
\cQ \hspace{-.2ex}\to\hspace{-.2ex} \cP$ and
\mbox{$\cP$-module}~$\cH$.  The finite poset~$\cP$ has order
dimension~$n$ for some positive integer~$n$; as such $\cP$ has an
embedding $\iota: \cP \into \ZZ^n$.  The pushforward $\iota_*\cH$
(Proposition~\ref{p:pushforward}) is finitely determined
(Definition~\ref{d:determined}; see also
Example~\ref{e:convex-projection}) as it is pulled back from any box
containing~$\iota(\cP)$.  The desired presentation or resolution is
pulled back to~$\cQ$ (via $\iota \circ \pi: \cQ \to \ZZ^n$) from the
corresponding flange, flat, or injective presentation or resolution
of~$\iota_*\cH$ afforded by Theorem~\ref{t:finitely-determined}.
These pullbacks are finite indicator resolutions of~$\cM$
dominating~$\pi$ by Example~\ref{e:pullback} and
Lemma~\ref{l:constant}.  The component homomorphisms are connected
because, by Corollary~\ref{c:U->D} and Example~\ref{e:connected-poset}
(see Definition~\ref{d:connected-poset}), components of flange
presentations, flat resolutions, and injective resolutions
over~$\ZZ^n$ are~\mbox{automatically}~\mbox{connected}.

The preceding argument proves the claim about a morphism $\cM \to
\cM'$, as well, since
\begin{itemize}
\item%
only one poset morphism is required to encode the morphism $\cM \to
\cM'$;
\item%
the push-pull constructions are functorial; and
\item%
morphisms of finitely determined modules can be lifted to the relevant
presentations and resolutions, since the relevant covers,
presentations, and resolutions are free or injective in the category
of finitely determined modules.\qedhere
\end{itemize}
\end{proof}


\begin{remark}
Comparing Theorems~\ref{t:syzygy} and~\ref{t:finitely-determined},
what happened to minimality?  It is not clear in what generality
minimality can be characterized.  The sequel \cite{essential-real} to
this paper can be seen as a case study for posets arising from abelian
groups that are either finitely generated and free (the closed
discrete polyhedral case) or real vector spaces of finite dimension
(the closed real polyhedral case).  The answer is much more nuanced in
the real case, obscuring how minimality might generalize beyond
these~cases.
\end{remark}

\begin{remark}\label{r:pullback-twice}
In the situation of the proof of Theorem~\ref{t:syzygy}, composing two
applications of Proposition~\ref{p:pullback-monomial-matrix}---one for
the encoding $\pi: \cQ \to \cP$ and one for the embedding $\iota: \cP
\into \ZZ^n$---yields a monomial matrix for a fringe presentation
of~$\cM$ directly from a monomial matrix for a flange presentation.
\end{remark}

\begin{remark}\label{r:RRn-mod}
Lesnick and Wright consider \mbox{$\RR^n$-modules}
\cite[\S2]{lesnick-wright2015} in finitely presented cases.  As
they indicate,
homological algebra of such $\RR^n$-modules is no different than
finitely generated \mbox{$\ZZ^n$-modules}.  This can be seen by finite
encoding: any finite poset in~$\RR^n$ is embeddable in~$\ZZ^n$,
because a product of finite chains is all that~is~needed.
\end{remark}

\subsection{Syzygy theorem for complexes of modules}\label{sub:complexes}\mbox{}

\noindent
Theorem~\ref{t:syzygy} is stated for individual modules, but the proof
works just as well for complexes, in a sense recorded here for
reference during the proof of Theorem~\ref{t:res}.

\begin{defn}\label{d:tame-complex}
Fix a complex $M^\spot$ of modules over a poset~$\cQ$.
\begin{enumerate}
\item%
$\cM^\spot$ is \emph{tame} if its modules and morphisms are tame
(Definitions~\ref{d:tame} and~\ref{d:tame-morphism}).

\item%
A constant subdivision or poset encoding is \emph{subordinate}
to~$\cM^\spot$ if it is subordinate to all of the modules and
morphisms therein, and in that case $\cM^\spot$ is said to
\emph{dominate} the subdivision or encoding.

\item%
An \emph{upset resolution} of~$\cM^\spot$ is a complex of
$\cQ$-modules in which each $F_i$ is a direct sum of upset modules and
the components $\kk[U] \to \kk[U']$ are connected, with a homomorphism
$F^\spot \to \cM^\spot$ of complexes inducing an isomorphism on
homology.

\item%
A \emph{downset resolution} of~$\cM^\spot$ is a complex of
$\cQ$-modules in which each $E_i$ is a direct sum of downset modules
and the components $\kk[D] \to \kk[D']$ are connected, with a
homomorphism $\cM^\spot \to E^\spot$ of complexes inducing an
isomorphism on homology.
\end{enumerate}
These resolutions are \emph{finite}, or \emph{dominate} a constant
subdivision or encoding, or are \emph{semialgebraic}, \emph{PL},
\emph{subanalytic}, or \emph{of class~$\mathfrak X$} as in
Definition~\ref{d:resolutions}.
\end{defn}

\begin{thm}[Syzygy theorem for complexes]\label{t:syzygy-complexes}
Theorem~\ref{t:syzygy} holds verbatim for a bounded complex $M^\spot$
in place of the module~$M$ as long as items~\ref{i:fringe},
\ref{i:upset-presentation}, and \ref{i:downset-copresentation} are
ignored.
\end{thm}
\begin{proof}
As already noted, the proof is the same.  It bears mentioning that
finite injective and flat resolutions of complexes exist in the
category of finitely determined $\ZZ^n$-modules because finite
injective resolutions do (Proposition~\ref{p:determined}): any of the
standard constructions that produce injective resolutions of complexes
given that modules have injective resolutions works in this setting,
and then Matlis duality (Definition~\ref{d:matlis}) produces finite
flat resolutions (see Remark~\ref{r:flat}).
\end{proof}

\section{Derived applications: conjectures of Kashiwara and Schapira}\label{s:derived}

The syzygy theorem for poset modules (Theorem~\ref{t:syzygy}) enhances
an arbitrary finite constant subdivision (the tame condition) to a
more structured subdivision (a~finite encoding)---or even an algebraic
presentation or resolution---whose pieces play well with the ambient
combinatorial structure.  In the context of partially ordered real
vector spaces, this enhancement produces a $\cQ_+$-stratification from
an arbitrary subanalytic triangulation.  If the triangulation is
subordinate to a given constructible derived $\cQ_+$-sheaf, meaning an
object in the bounded derived category of constructible sheaves with
microsupport contained in the negative polar cone of~$\cQ_+$, then
this enhancement produces a $\cQ_+$-structured resolutions of the
given sheaf.  These two instances of the syzygy theorem are the
crucial ingredients for proofs of two conjectures due to Kashiwara and
Schapira.  Since the mathematical context is more sophisticated than
the rest of the paper, these notions require review to make precise
statements and proofs.

To avoid endlessly repeating hypotheses, and so readers can quickly
identify when the same hypotheses are in effect, the blanket
assumption in this section is for $Q$ to satisfy the following, where
a positive cone is \emph{full} if it has nonempty interior.  Real
partially ordered groups are partially ordered real vector spaces of
finite dimension~(Example~\ref{e:real-pogroup}).

\begin{hyp}\label{h:full-closed-cone}
$Q$ is a real partially ordered group with closed, full,
subanalytic~$Q_+$.
\end{hyp}

\subsection{Stratifications, topologies, and cones}\label{sub:cones}

This subsection collects the relevant definitions and theorems from
the literature.  The sizeable edifice on which the subject is built
makes it unavoidable that readers seeing some of these topics for the
first time will need to consult the cited sources for additional
background.  The goal here is to bring readers as quickly as possible
to a general statement (Theorem~\ref{t:res}) while circumscribing the
ingredients necessary for its proof in such a way that those already
familiar with the conjectures of Kashiwara and Schapira, specifically
\cite[Conjecture~3.17]{kashiwara-schapira2017} and
\cite[Conjecture~3.20]{kashiwara-schapira2019}, can skip seamlessly to
Section~\ref{sub:res} after skimming Section~\ref{sub:cones} for
terminology.

\begin{remark}\label{r:used-without-comment}
Some basic notions are used freely without further comment.
\begin{enumerate}
\item%
The notion of simplicial complex here is the one in
\cite[Definition~8.1.1]{kashiwara-schapira1990}: a collection~$\Delta$
of subsets (called \emph{simplices}) of a fixed vertex set that is
closed under taking subsets (called \emph{faces}), contains every
vertex, and is locally finite in the sense that every vertex
of~$\Delta$ lies in finitely many simplices of~$\Delta$.  Any
simplicial complex~$\Delta$ has a realization~$|\Delta|$ as a
topological space, with each relatively open simplex $|\sigma|$ being
an open convex set in an appropriate affine space.

\item%
The notion of subanalytic set in an analytic manifold is as in
\cite[\S8.2]{kashiwara-schapira1990}.

\item%
The term \emph{sheaf} on a topological space here means a sheaf of
$\kk$-vector spaces.  Sometimes in the literature this word is used to
mean an object in the bounded derived category of sheaves of
$\kk$-vector spaces; for clarity here, the term \emph{derived sheaf}
is always used when an object in the derived category is intended.
\end{enumerate}
\end{remark}

\subsubsection{Subanalytic triangulation}\label{ssub:triangulation}

\begin{defn}\label{d:subanalytic-triangulation}
Fix a real analytic manifold~$X$.
\begin{enumerate}
\item%
A \emph{subanalytic triangulation} of a subanalytic set~$Y \subseteq
X$ is a homeomorphism $|\Delta| \simto Y$ such that the image in~$Y$
of the realization~$|\sigma|$ of the relative interior of each simplex
$\sigma \in \Delta$ is a subanalytic submanifold of~$X$.

\item%
A subanalytic triangulation of~$Y$ is \emph{subordinate} to a
(derived) sheaf $\FF$ on~$X$ if $Y$ contains the support of~$\FF$ and
(every homology sheaf~of)~$\FF$ restricts to a constant sheaf on the
image in~$Y$ of every cell~$|\sigma|$.
\end{enumerate}
\end{defn}

\subsubsection{Subanalytic constructibility}\label{ssub:constructible}

\begin{defn}\label{d:constructible-sheaf}
A (derived) sheaf on a real analytic manifold is \emph{subanalytically
weakly constructible} if there is a subanalytic triangulation
subordinate to it.  $\!$The word\hspace{-.1ex} ``weakly'' is omitted
if, in addition, the stalks have finite dimension as $\kk$-vector
spaces.
\end{defn}

\begin{remark}\label{r:constructible-sheaf}
Readers less familiar with constructibility can safely take
Definition~\ref{d:constructible-sheaf} at face value.  For readers
familiar with constructibility by other definitions,
this characterization is a nontrivial theorem, which rests on the
triangulability of subanalytic sets
\cite[Proposition~8.2.5]{kashiwara-schapira1990} and other results
concerning subanalytic stratification; see
\cite[\S8.4]{kashiwara-schapira1990} for the full proof of
equivalence, especially Theorem~8.4.2, Definition~8.4.3, and part~(a)
of the proof of Theorem~8.4.5(i) there.  Note that the modifier
``subanalytically'' in Definition~\ref{d:constructible-sheaf} does not
appear in \cite{kashiwara-schapira1990}, because the context there is
subanalytic throughout.  Also note that it makes no difference whether
one takes constructible objects in the derived category or the derived
category of constructible objects, since they yield the same result
\cite[Theorem~8.4.5]{kashiwara-schapira1990}: every constructible
derived sheaf is represented by a complex of constructible sheaves.
\end{remark}

The reason to use subanalytic triangulation instead of arbitrary
subanalytic stratification is the following, which is a step on the
way to a constant subdivision.

\begin{lemma}[{\cite[Proposition~8.1.4]{kashiwara-schapira1990}}]\label{l:tri-constant}
For a simplex~$\sigma$ in a subanalytic triangulation subordinate to a
constructible sheaf~$\FF$, there is a natural isomorphism
$\Gamma(|\sigma|,\FF) \simto \FF_x$ from the sections over~$|\sigma|$
to the stalk at every point~$x \in \sigma$.
\end{lemma}

The reason for specifically including the piecewise linear (PL)
condition in previous sections is for its application here, as one of
the conjectures is in that setting.  For this purpose, the sheaf
version of this particularly strong type of constructibility is
needed.

\begin{defn}\label{d:PL-constructible}
Fix $Q$ satisfying Hypothesis~\ref{h:full-closed-cone}.
\begin{enumerate}
\item\label{i:stratification}%
A~subanalytic subdivision
(Definition~\ref{d:auxiliary-hypotheses}.\ref{i:subanalytic}) of~$Q$
is \emph{subordinate} to a (derived) sheaf~$\FF$ on~$Q$ if the
restriction of~$\FF$ to every \emph{stratum} (meaning subset in the
subdivision) is constant of finite rank.

\item\label{i:PL-constructible}%
If the subanalytic subdivision~is PL
(Definition~\ref{d:auxiliary-hypotheses}.\ref{i:PL}) and $Q$ is
polyhedral (Definition~\ref{d:face}), then $\FF$ is said to be
\emph{piecewise linear}, abbreviated \emph{PL}.
\end{enumerate}
\end{defn}

\begin{remark}\label{r:PL}
Definition~\ref{d:PL-constructible}.\ref{i:PL-constructible} is not
verbatim the same as \cite[Definition~2.3]{kashiwara-schapira2019},
which only requires $Q$ to be a (nondisjoint) union of finitely
polyhedra on which $\FF$ is constant.  However, the notion of PL
(derived) sheaf thus defined is the same, since any finite union of
polyhedra can be refined to a finite union that is disjoint---that is,
a partition.  This refinement can be done, for example, by
expressing~$Q$ as the union of (relatively open) faces in the
arrangement of all hyperplanes bounding halfspaces defining the given
polyhedra, of which there are only finitely many.
\end{remark}

\subsubsection{Conic and Alexandrov topologies}\label{ssub:topology}

\begin{defn}\label{d:conic-topology}
Fix a real partially ordered group~$Q$ with closed positive
cone~$Q_+$.
\begin{enumerate}
\item%
The \emph{conic topology} on~$Q$ induced by~$Q_+$ (or induced by the
partial order) consists of the upsets that are open in the ordinary
topology on~$Q$.

\item%
The \emph{Alexandrov topology} on~$Q$ induced by~$Q_+$ (or induced by
the partial order) consists of all the upsets in~$Q$.
\end{enumerate}
To avoid confusion when it might occur, write
\begin{enumerate}
\item%
$Q^\con$ for the set~$Q$ with the conic topology induced by~$Q_+$,

\item%
$Q^\ale$\, for the set~$Q$ with the Alexandrov topology induced
by~$Q_+$,~and

\item%
$Q^\ord$ for the set~$Q$ with its ordinary topology.
\end{enumerate}
\end{defn}

\begin{remark}\label{r:conic-topology}
The conic topology in Definition~\ref{d:conic-topology} is also known
as the \emph{$\gamma$-topology}, where $\gamma = Q_+$
\cite{kashiwara-schapira1990, kashiwara-schapira2018,
kashiwara-schapira2019}.  The Alexandrov topology makes just as much
sense on any poset.
\end{remark}

The type of stratification Kashiwara and Schapira specify
\cite[Conjecture~3.17]{kashiwara-schapira2017} is not quite the same
as subanalytic subdivision in
Definition~\ref{d:auxiliary-hypotheses}.\ref{i:subanalytic}.  To be
precise, first recall two standard topological concepts.

\begin{defn}\label{d:locally-closed}
A subset of a topological space~$Q$ is \emph{locally closed} if it is
the intersection of an open subset and a closed subset.  A family of
subsets of~$Q$ is \emph{locally finite} if each compact subset of~$Q$
meets only finitely many members of the family.
\end{defn}

\begin{defn}[{\cite[Definition~3.15]{kashiwara-schapira2017}}]\label{d:stratification}
Fix $Q$ satisfying Hypothesis~\ref{h:full-closed-cone}.
\begin{enumerate}
\item%
A \emph{conic stratification} of a closed subset $S \subseteq Q$ is a
locally finite family of pairwise disjoint subanalytic subsets, called
\emph{strata}, which are locally closed in the conic topology and have
closures whose union is~$S$.

\item%
The stratification is \emph{subordinate} to a (derived) sheaf~$\FF$
on~$Q$ if $S$ equals the support of~$\FF$ and the restriction of (each
homology sheaf of)~$\FF$ to every stratum is locally constant of
finite~rank.
\end{enumerate}
\end{defn}

\begin{remark}\label{r:stratification}
A conic stratification is called a $\gamma$-stratification in
\cite[Definition~3.15]{kashiwara-schapira2017}, with $\gamma = Q_+$.
The only differences between conic stratification and subanalytic
partition of a subset~$S$ in
Definition~\ref{d:auxiliary-hypotheses}.\ref{i:subanalytic} are that
\begin{itemize}
\item%
conic stratifications are only required to be locally finite, not
necessarily finite;

\item%
conic strata are required to be locally closed in the conic topology
(that is, an intersection of an open upset in~$Q^\ord$ with a closed
downset in~$Q^\ord$); and

\item%
the union need not actually equal all of~$S$, because the only the
union of the stratum closures is supposed to equal~$S$.
\end{itemize}
\end{remark}

\begin{prop}\label{p:stalks}
Fix a real partially ordered group~$Q$ with closed
positive~cone~$Q_+$.
\begin{enumerate}
\item\label{i:identity-continuous}%
The identity on~$Q$ yields continuous maps of topological spaces
$$%
  \iota: Q^\ord \to Q^\con
  \qquad\text{and}\qquad
  \jmath: Q^\ale \to Q^\con.
$$

\item\label{i:ord-pulled-back}%
Any sheaf~$\FF$ on~$Q^\ord$ pulled back from $Q^\con$ has natural maps
$$%
  \FF_q \to \FF_{q'}\text{ for }q \preceq q'\text{ in }Q
$$
on stalks that functorially define a $Q$-module $\bigoplus_{q\in Q}
\FF_q$.

\item\label{i:alexdrov=mod}%
Similarly, any sheaf~$\GG$ on~$Q^\ale$ has natural maps
$$%
  \GG_q \to \GG_{q'}\text{ for }q \preceq q'\text{ in }Q
$$
on stalks that functorially define a $Q$-module $\bigoplus_{q\in Q}
\GG_q$.  This functor from sheaves on~$Q^\ale$ to $Q$-modules is an
equivalence of categories.

\item\label{i:same-sheaf}%
If sheaves $\FF$ on~$Q^\ord$ and~$\GG$ on $Q^\ale$ are both pulled
back from the same sheaf $\EE$ on~$Q^\con$, then the $Q$-modules in
items~\ref{i:ord-pulled-back} and~\ref{i:alexdrov=mod} are the same.

\item\label{i:exact-pushforward}%
The pushforward functor $\jmath_*$ is exact, and $\jmath_* \jmath^{-1}
\EE \cong \EE$.
\end{enumerate}
\end{prop}
\begin{proof}
The maps in item~\ref{i:identity-continuous} are continuous by
definition: the inverse image of any open set is open because the
ordinary topology refines each of the target topologies.

For item~\ref{i:ord-pulled-back}, if $\FF = \iota^{-1}\EE$ is pulled
back to~$Q^\ord$ from a sheaf $\EE$ on~$Q^\con$, then $\FF$ has the
same stalks as~$\EE$ (as a sheaf pullback in any context does), so the
natural morphisms are induced by the restriction maps of~$\EE$ from
open neighborhoods of~$q$~to~those~of~$q'$.

The result in~\ref{i:alexdrov=mod} holds for arbitrary posets; for an
exposition in a context relevant to persistence, see
\cite[Theorem~4.2.10 and Remark~4.2.11]{curry-thesis} and
\cite{curry2019}.

For item~\ref{i:same-sheaf}, the stalks $\FF_q = \EE_q = \GG_q$ are
the same.

For item~\ref{i:exact-pushforward}, exactness is proved in passing in
the proof of \cite[Lemma~3.5]{berkouk-petit2019}, but it is also
elementary to check that a surjection $\GG \onto \GG'$ of sheaves
on~$Q^\ale$ yields a surjection of stalks for the pushforwards
to~$Q^\con$ because direct limits (filtered colimits) are exact.  That
$\jmath_* \jmath^{-1} \EE \cong \EE$ is because the natural morphism
is the identity on stalks.
\end{proof}

\subsubsection{Conic microsupport}\label{ssub:microsupport}\mbox{}

\noindent
The \emph{microsupport} of a (derived) sheaf on an analytic
manifold~$X$ is a certain closed conic isotropic subset of the
cotangent bundle $T^*X$.  The notion of microsupport is a central
player in \cite{kashiwara-schapira1990}, to which the reader is
referred for background on the topic.  However, although the main
result in this section (Theorem~\ref{t:res}) is stated in terms of
microsupport, the next theorem allows the reader to ignore it
henceforth, as pointed out by Kashiwara and Schapira themselves
\cite[Remark~1.9]{kashiwara-schapira2018}, by immediately translating
to the more elementary context of sheaves in the conic topology in
Section~\ref{ssub:topology}.

\begin{thm}[{\cite[Theorem~1.5 and Corollary~1.6]{kashiwara-schapira2018}}]\label{t:microsupport}
Fix $Q$ satisfying Hypothesis~\ref{h:full-closed-cone}.  The
pushforward $\iota_*$ of the map~$\iota$ from
Proposition~\ref{p:stalks}.\ref{i:identity-continuous} induces an
equivalence from the category of sheaves with microsupport contained
in the negative polar cone~$Q_+^\vee$ to the category of sheaves in
the conic topology.  The pullback~$\iota^{-1}$ is a quasi-inverse.
The same assertions hold for the bounded derived categories.
\end{thm}

\begin{remark}\label{r:push-and-pull}
The pushforward $\iota_*$ and the pullback $\iota^{-1}$ have concrete
geometric descriptions.  Since $\iota$ is the identity on~$Q$, the
pushforward of a sheaf~$\FF$ on $Q$ has~sections
$$%
  \Gamma(U, \iota_*\FF) = \Gamma(U,\FF)
$$
for any open upset~$U$, where ``open upset'' means the same things as
``upset that is open in the usual topology'' and ``subset that is open
in the conic topology''.  On the other hand, over any convex
ordinary-open set~$\OO$, the pullback to the ordinary topology of a
sheaf~$\EE$ in the conic topology has sections
$$%
  \Gamma(\OO,\iota^{-1}\EE) = \Gamma(\OO + Q_+, \EE),
$$
namely the sections of~$\EE$ over the upset generated by~$\OO$
\cite[(3.5.1)]{kashiwara-schapira1990}.
\end{remark}

\begin{remark}\label{r:microsupport}
What Theorem~\ref{t:microsupport} does in practice is allow a given
(derived) sheaf with microsupport contained in the negative polar
cone~$Q_+^\vee$ to be replaced with an isomorphic object that is
pulled back from the conic topology induced by the partial order.  The
reason for mentioning the notion of microsupport at all is to
emphasize that constructibility in the sense of
Definition~\ref{d:constructible-sheaf} requires the ordinary topology.
This may seem a fine distinction, but the conjectures of Kashiwara and
Schapira proved in Section~\ref{sub:strat} entirely concern the
transition from the ordinary to the conic topology, so it is crucial
to be clear on this point.
\end{remark}

\subsection{Resolutions of constructible sheaves}\label{sub:res}

\begin{defn}\label{d:flasque}
Fix $Q$ satisfying Hypothesis~\ref{h:full-closed-cone}.
\begin{enumerate}
\item%
A \emph{subanalytic upset sheaf} on~$Q$ is the extension by zero of
the rank~$1$ constant sheaf on an open subanalytic upset in~$Q^\ord$.

\item%
A \emph{subanalytic downset sheaf} on~$Q$ is the pushforward of the
rank~$1$ locally constant sheaf on a closed subanalytic downset
in~$Q^\ord$.

\item%
A \emph{subanalytic upset resolution} of a complex~$\FF^\spot$ of
sheaves on~$Q^\ord$ is a homomorphism $\UU^\spot \to \FF^\spot$ of
complexes inducing an isomorphism on homology, with each $\UU^i$ being
a direct sum of subanalytic upset sheaves.

\item%
A \emph{subanalytic downset resolution} of a complex~$\FF^\spot$ of
sheaves on~$Q^\ord$ is a homomorphism $\FF^\spot \to \DD^\spot$ of
complexes inducing an isomorphism on homology, with each each $\DD^i$
being a direct sum of subanalytic downset sheaves.
\end{enumerate}
Either type of resolution is
\begin{itemize}
\item%
\emph{finite} if the total number of summands across all homological
degrees is finite;

\item%
\emph{PL} if $Q$ is polyhedral and the upsets or downsets are PL.
\end{itemize}
\end{defn}

\begin{prop}\label{p:upset}
Fix an upset $U$ in a real partially ordered group~$Q$ with closed
positive cone.  If $U^\circ$ is the interior of~$U$ in~$Q^\ord$, then
the sheaves on~$Q^\ale$ corresponding to~$\kk[U]$ and $\kk[U^\circ]$
push forward to the same sheaf on~$Q^\con$.
\end{prop}
\begin{proof}
The stalk at~$q$ of any sheaf on~$Q^\con$ is the direct limit over
points $p \in q - Q_+^\circ$ of the sections over $p + Q_+^\circ$.  In
the case of the pushforward of the sheaf on~$Q^\ale$ corresponding to
an upset module, these sections are~$\kk$ if $p$ lies interior to the
upset and~$0$ otherwise.  The result holds because the upsets $U$
and~$U^\circ$ have the same interior, namely~$U^\circ$.
\end{proof}

\begin{prop}\label{p:downset}
Fix a downset $D$ in a real partially ordered group~$Q$ with closed
positive cone.  If $\oD$ is the closure of $D$ in~$Q^\ord$, then the
sheaves on~$Q^\ale$ corresponding to~$\kk[D]$ and $\kk[\oD]$ push
forward to the same sheaf on~$Q^\con$.
\end{prop}
\begin{proof}
Calculating stalks as in the previous proof, in the case of the
pushforward of the sheaf on~$Q^\ale$ corresponding to a downset
module, the sections over $p + Q_+^\circ$ are~$\kk$ if $p$ lies
interior to the downset and~$0$ otherwise.  The result holds because
the downsets $D$ and~$\oD$ have the same interior.
\end{proof}

\begin{remark}\label{r:interior}
The fundamental difference between Alexandrov and conic topologies
reflected by\hspace{-.9pt} Propositions~\ref{p:upset}
and~\ref{p:downset} is explored in detail by\hspace{-.9pt} Berkouk
and~\mbox{Petit}~\cite{berkouk-petit2019}.
\end{remark}

Here is the main result of Section~\ref{s:derived}.  It is little more
than a restatement of the relevant part of
Theorem~\ref{t:syzygy-complexes} in the language of sheaves.

\begin{thm}\label{t:res}
Fix $Q$ satisfying Hypothesis~\ref{h:full-closed-cone}.  If
$\FF^\spot$ is a complex of compactly supported constructible sheaves
on~$Q^\ord$ with microsupport in the negative polar cone~$Q_+^\vee$
then $\FF^\spot$ has a finite subanalytic upset resolution and a
finite subanalytic downset resolution.  If $Q$ is polyhedral and
$\FF^\spot$ is PL, then $\FF^\spot$ has PL such resolutions.
\end{thm}
\begin{proof}
Using Theorem~\ref{t:microsupport}, assume that $\FF^\spot$ is pulled
back to~$Q^\ord$ from~$Q^\con$, say $\FF^\spot = \iota^{-1}\EE^\spot$.
Since $\FF^\spot$ has compact support, any subordinate subanalytic
triangulation (Definition~\ref{d:subanalytic-triangulation}) afforded
by Definition~\ref{d:constructible-sheaf} is necessarily finite
because it is locally finite.  The complex $F^\spot = \bigoplus_{q\in
Q} \FF^\spot_q$ of $Q$-modules that comes from
Proposition~\ref{p:stalks}.\ref{i:ord-pulled-back} is tamed by the
triangulation, which is a constant subdivision
(Definition~\ref{d:constant-subdivision}) because
\begin{itemize}
\item%
simplices are connected, so locally constant sheaves on them are
constant, and

\item%
$\Gamma(|\sigma_p|,\FF^i) \to \FF^i_p \to \FF^i_q \from
\Gamma(|\sigma_q|,\FF^i)$ is locally constant---and hence constant, as
simplices are connected---when $p \preceq q$ in~$Q$.  Here $\sigma_x$
is the simplex containing~$x$, the middle arrow is from
Proposition~\ref{p:stalks}.\ref{i:ord-pulled-back}, and the outer
arrows are the natural isomorphisms from Lemma~\ref{l:tri-constant}.
\end{itemize}
Hence the complex $F^\spot$ of $Q$-modules has resolutions of the
desired sort by Theorem~\ref{t:syzygy-complexes}.  Viewing any of
these resolutions as a complex of sheaves on~$Q^\ale$ via
Proposition~\ref{p:stalks}.\ref{i:alexdrov=mod}, push it forward from
the Alexandrov topology to the conic topology via the exact
functor~$\jmath_*$ in
Proposition~\ref{i:exact-pushforward}.\ref{p:stalks}.  The resulting
complex of sheaves on~$Q^\con$ is a resolution of a complex isomorphic
to~$\EE^\spot$ by Proposition~\ref{p:stalks}.\ref{i:same-sheaf}
and~\ref{p:stalks}.\ref{i:exact-pushforward}.  The upsets or downsets
in the summands of the resolution may as well be assumed open or
closed, respectively, by Propositions~\ref{p:upset}
or~\ref{p:downset}.  The proof is concluded by pulling back the
resolution from~$Q^\con$ to~$Q^\ord$ via the equivalence of
Theorem~\ref{t:microsupport}.
\end{proof}

\begin{remark}\label{r:compact}
Theorem~\ref{t:res} assumes compact support to get finite instead of
locally finite subdivisions.  The application in
Section~\ref{sub:strat} to constructible sheaves without any
assumption of compact support yields a locally finite subdivision by
reducing to the case of compact support.
\end{remark}

\begin{remark}\label{r:semialgebraic}
The final sentence of Theorem~\ref{t:res} is true with ``polyhedral''
and ``PL'' all replaced by ``semialgebraic'', with the same proof, as
long as the definitions of these semialgebraic concepts in the
constructible sheaf setting are made appropriately.  The semialgebraic
constructible sheaf versions are not treated here because they are not
relevant to the conjectures proved in Section~\ref{sub:strat}.
\end{remark}

\subsection{Stratifications from constructible sheaves}\label{sub:strat}%

\begin{cor}[{\cite[Conjecture~3.20]{kashiwara-schapira2019}}]\label{c:PL}
Fix $Q$ satisfying Hypothesis~\ref{h:full-closed-cone} with~$Q_+$
polyhedral.  If $\FF^\spot$ is a PL object in the derived category of
compactly supported constructible sheaves on~$Q^\ord$ with
microsupport contained in the negative polar cone~$Q_+^\vee$ then the
isomorphism class of~$\FF^\spot$ is represented by a complex that is a
finite direct sum of constant sheaves on bounded polyhedra that are
locally closed in the conic topology.
\end{cor}
\begin{proof}
The statement would directly be a special case of Theorem~\ref{t:res}
were it not for the boundedness hypothesis on the polyhedra, since
either a PL upset or PL downset resolution would satisfy the
conclusion.  That said, boundedness is easy to impose:
since~$\FF^\spot$ has compact support, and the resolution has
vanishing homology outside of the support of~$\FF^\spot$, each upset
or downset sheaf can be restricted to the support of~$\FF^\spot$ and
extended by~$0$.
\end{proof}

\begin{cor}[{\cite[Conjecture\,3.17]{kashiwara-schapira2017}}]\label{c:strat}
Fix $Q$ satisfying Hypothesis~\ref{h:full-closed-cone}.  If a
compactly supported derived sheaf with microsupport in the negative
polar cone~$Q_+^\vee$ is subanalytically constructible, then its
support has a subordinate conic stratification.
\end{cor}
\begin{proof}
Part~(ii) in the proof of \cite[Theorem~3.17]{kashiwara-schapira2018}
reduces to the case where the support of the given derived sheaf is
compact.  The argument is presented in the case where $Q$ is
polyhedral and the derived sheaf is~PL, but the argument works
verbatim for $Q$ satisfying Hypothesis~\ref{h:full-closed-cone},
without any polyhedral or PL assumptions, because the requisite lemma,
namely \cite[Lemma~3.5]{kashiwara-schapira2018}---and indeed, all of
\cite[\S3.1]{kashiwara-schapira2018}---is stated and proved in
this non-polyhedral generality.  So henceforth assume the given
derived sheaf has compact support.

Remark~\ref{r:constructible-sheaf} allows the assumption that the
given derived sheaf is represented by a complex~$\FF^\spot$ of
constructible sheaves.  Theorem~\ref{t:res} produces a subanalytically
indicator resolution, which for concreteness may as well be an upset
resolution.  Each upset that appears as a summand in the resolution
partitions~$Q$ into the upset itself, which is open subanalytic, and
its complement, which is a closed subanalytic downset.  The common
refinement of the partitions induced by the finitely many open
subanalytic upsets in the resolution and their closed subanalytic
downset complements is a partition of~$Q$ into finitely many strata
such that
\begin{itemize}
\item%
each stratum is subanalytic and locally closed in the conic topology,
and

\item%
the restriction of~$\FF^\spot$ to each stratum has constant homology.
\end{itemize}
The strata with nonvanishing homology form the desired conic
stratification.
\end{proof}

\begin{remark}\label{r:lambda}
The reference in \cite[Conjecture\,3.17]{kashiwara-schapira2017} to a
cone~$\lambda$ contained in the interior of the positive cone union
the origin appears to be unnecessary, since (in the notation there)
any $\gamma$-stratification is automatically a
$\lambda$-stratification by
\cite[Definition~3.15]{kashiwara-schapira2017} and the fact that
$\lambda \subseteq \gamma$.
\end{remark}

\begin{remark}\label{r:flexibility}
One main point of Corollary~\ref{c:strat} is that, while the notion of
a sheaf with microsupport contained in the negative polar cone
of~$\cQ_+$ is equivalent to the notion of a sheaf in the conic
topology, the notion of constructibility has until now only been
available on the microsupport side, where simplices from arbitrary
subanalytic triangulations achieve constancy of the sheaves in
question.  One way to interpret Corollary~\ref{c:strat} is that
constructibility can be detected entirely with the more rigid conic
topology, without the flexibility of appealing to arbitrary
subanalytic triangulations.
\end{remark}

\addtocontents{toc}{\protect\setcounter{tocdepth}{2}}

\vspace{-1.8ex}

\end{document}